\documentclass[11pt,twoside]{article}
\pdfoutput=1

\usepackage{fullpage}

\usepackage{epsf}
\usepackage{fancyheadings}
\usepackage{graphics}
\usepackage{graphicx}
\usepackage{psfrag}

\usepackage{color}

\usepackage{amsthm}
\usepackage{amsfonts}
\usepackage{amsmath}
\usepackage{amssymb}

\usepackage{graphicx,epsfig,graphics,epsf}
\usepackage{float,amsmath,amsfonts,amssymb,amsthm}
\usepackage{mathtools}
\usepackage{subcaption}

\usepackage{amsmath,amssymb}
\usepackage{amsthm}
\usepackage{mathtools}
\usepackage[noend]{algorithmic}
\usepackage[ruled,vlined]{algorithm2e}
\usepackage{url}
\usepackage{fullpage}
\usepackage{makeidx}
\usepackage{enumerate}
\usepackage[top=1in, bottom=1.25in, left=1in, right=1in]{geometry}
\usepackage{graphicx,float,psfrag,epsfig,caption}
\usepackage[usenames,dvipsnames,svgnames,table]{xcolor}
\definecolor{darkgreen}{rgb}{0.0,0,0.9}

\theoremstyle{plain}

\newtheorem{innercustomgeneric}{\customgenericname}
\providecommand{\customgenericname}{}
\newcommand{\newcustomtheorem}[2]{%
  \newenvironment{#1}[1]
  {%
   \renewcommand\customgenericname{#2}%
   \renewcommand\theinnercustomgeneric{##1}%
   \innercustomgeneric
  }
  {\endinnercustomgeneric}
}

\newcustomtheorem{customthm}{Theorem}
\newcustomtheorem{customlemma}{Lemma}
\newcustomtheorem{customproposition}{Proposition}
\newcustomtheorem{customcorollary}{Corollary}

\newlength{\widebarargwidth}
\newlength{\widebarargheight}
\newlength{\widebarargdepth}
\DeclareRobustCommand{\widebar}[1]{%
  \settowidth{\widebarargwidth}{\ensuremath{#1}}%
  \settoheight{\widebarargheight}{\ensuremath{#1}}%
  \settodepth{\widebarargdepth}{\ensuremath{#1}}%
  \addtolength{\widebarargwidth}{-0.3\widebarargheight}%
  \addtolength{\widebarargwidth}{-0.3\widebarargdepth}%
  \makebox[0pt][l]{\hspace{0.3\widebarargheight}%
    \hspace{0.3\widebarargdepth}%
    \addtolength{\widebarargheight}{0.3ex}%
    \rule[\widebarargheight]{0.95\widebarargwidth}{0.1ex}}%
  {#1}}

\makeatletter
\long\def\@makecaption#1#2{
        \vskip 0.8ex
        \setbox\@tempboxa\hbox{\small {\bf #1:} #2}
        \parindent 1.5em  
        \dimen0=\hsize
        \advance\dimen0 by -3em
        \ifdim \wd\@tempboxa >\dimen0
                \hbox to \hsize{
                        \parindent 0em
                        \hfil 
                        \parbox{\dimen0}{\def\baselinestretch{0.96}\small
                                {\bf #1.} #2
                                } 
                        \hfil}
        \else \hbox to \hsize{\hfil \box\@tempboxa \hfil}
        \fi
        }
\makeatother

\long\def\comment#1{}

\newcommand{\Prob}{\ensuremath{{\mathbb{P}}}}

\DeclareMathOperator{\diag}{diag}

\DeclareSymbolFont{rsfs}{U}{rsfs}{m}{n}
\DeclareSymbolFontAlphabet{\mathscrsfs}{rsfs}

\numberwithin{equation}{section}

\newtheoremstyle{myexample} 
    {\topsep}                    
    {\topsep}                    
    {\rm }                   
    {}                           
    {\bf }                   
    {.}                          
    {.5em}                       
    {}  

\newtheoremstyle{myremark} 
    {\topsep}                    
    {\topsep}                    
    {\rm}                        
    {}                           
    {\bf}                        
    {.}                          
    {.5em}                       
    {}  

\newtheorem{claim}{Claim}[section]
\newtheorem{lemma}[claim]{Lemma}
\newtheorem{fact}[claim]{Fact}

\newtheorem{theorem}{Theorem}
\newtheorem{proposition}[claim]{Proposition}
\newtheorem{corollary}[claim]{Corollary}
\newtheorem{definition}[claim]{Definition}

\theoremstyle{myremark}
\newtheorem{remark}{Remark}[section]
 \newenvironment{proofof}[1]{{\bf {\em Proof of #1.}}}{\hfill \rule{2mm}{2mm} 
 }

\usepackage{tikz}
\usepackage{tkz-graph}
\usetikzlibrary{shapes, arrows, calc, positioning,matrix}
\tikzset{
data/.style={circle, draw, text centered, minimum height=3em ,minimum width = .5em, inner sep = 2pt},
empty/.style={circle, text centered, minimum height=3em ,minimum width = .5em, inner sep = 2pt},
}
\usepackage{pgfplots}
\pgfplotsset{compat=1.5}
\usepgfplotslibrary{fillbetween}
\usetikzlibrary{patterns}

\newcommand{\norm}[1]{\left\lVert#1\right\rVert}

\DeclarePairedDelimiterX{\inp}[2]{\langle}{\rangle}{#1, #2}
\DeclareMathOperator{\E}{\mathbb{E}}
\DeclareMathOperator{\pr}{\mathbb{P}}

\DeclareMathOperator*{\argmin}{arg\,min}

\newcommand{\beps}{\boldsymbol{\epsilon}}

\newcommand{\bbe}{\boldsymbol{\beta}}
\newcommand{\bw}{\boldsymbol{w}}
\newcommand{\bX}{\boldsymbol{X}}
\newcommand{\by}{\boldsymbol{y}}
\newcommand{\bM}{\boldsymbol{M}}
\newcommand{\bZ}{\boldsymbol{Z}}
\newcommand{\bz}{\boldsymbol{z}}
\newcommand{\bW}{\boldsymbol{W}}

\newcommand{\bI}{\boldsymbol{I}}
\newcommand{\bSig}{\boldsymbol{\Sigma}}
\newcommand{\be}{\boldsymbol{e}}
\newcommand{\bg}{\boldsymbol{g}}
\newcommand{\bv}{\boldsymbol{v}}
\newcommand{\bu}{\boldsymbol{u}}

\newcommand{\bx}{\boldsymbol{x}}
\newcommand{\bA}{\boldsymbol{A}}

\newcommand{\bOmega}{\boldsymbol{\Omega}}

\def\normal{{\sf N}}

\usepackage{bbm}
\usepackage[inline]{enumitem}
\usepackage{enumitem}

\usepackage[pagebackref,letterpaper=true,colorlinks=true,pdfpagemode=none,citecolor=OliveGreen,linkcolor=BrickRed,urlcolor=BrickRed,pdfstartview=FitH]{hyperref}

\title{Imputation for High-Dimensional Linear Regression}
\author{Kabir Aladin Chandrasekher\thanks{Department of Electrical Engineering,
Stanford University} \and Ahmed El Alaoui\footnotemark[1]~\thanks{Department of Statistics, Stanford University} \and
Andrea Montanari\footnotemark[1]~\footnotemark[2]}
\begin{document}

%
%
%
%
\date{}
%
\maketitle
\begin{abstract}

We study high-dimensional regression with missing entries in the covariates.  A
common strategy in practice is to \emph{impute} the missing entries with an
appropriate substitute and then implement a standard statistical procedure
acting as if the covariates were fully observed. Recent literature on this
subject proposes instead to design a specific, often complicated or non-convex, algorithm tailored to the case of
missing covariates. We investigate a simpler approach where we fill-in the
missing entries with their conditional mean given the observed covariates.  We
show that this imputation scheme coupled with standard off-the-shelf procedures such as the LASSO and square-root
LASSO retains the minimax estimation rate in the random-design setting where the
covariates are i.i.d.\ sub-Gaussian.  We further show that the square-root
LASSO remains \emph{pivotal} in this setting.

It is often the case that  the conditional expectation cannot be computed
exactly and must be approximated from data. We study two cases where the
covariates either follow an  autoregressive (AR) process, or are jointly
Gaussian with sparse precision matrix. We propose tractable estimators for the
conditional expectation and then perform linear regression via LASSO, and show
similar estimation rates in both cases. We complement our theoretical results with
simulations on synthetic and semi-synthetic examples, illustrating not only the sharpness
of our bounds, but also the broader utility of this strategy beyond our theoretical
assumptions.        
\end{abstract}


\section{Introduction}
\label{sec:Introduction}
Statistical estimation procedures are usually designed under the assumption that
data is fully observed. It is however common that some
portion of the data is missing or observed through a noisy channel.  A natural strategy to address this problem is to replace the missing data with a sensible proxy.  Such a strategy, known as imputation, is widely used in practice.

We observe a response vector $\by \in \mathbb{R}^n$ from a linear model with design matrix $\bX \in \mathbb{R}^{n \times p}$ in following manner:
\begin{align}
\by = \bX \bbe_0 + \beps,
\end{align} 
where $\beps$ is zero-mean sub-Gaussian noise.  We are interested in recovering the unknown regression vector $\bbe_0 \in
\mathbb{R}^p$.  We are specifically interested in the high dimensional regime,
when $p \gg n$.  In this underdetermined setting, it is necessary to assume some
structure on the regression vector $\bbe_0$.  It is of particular interest when
$\bbe_0$ has a few non-zero entries. In this case we say the vector $\bbe_0$ is sparse and denote its number of non-zeros by $s$.  
Sparse regression in the high-dimensional regime has been widely studied in the last
decade and has witnessed a beautiful line of results showing that convex programs
such as the Lasso~\cite{tibshirani1996regression} and the Dantzig
Selector~\cite{candes2007dantzig} give rate-optimal statistical guarantees~\cite{raskutti2011minimax,candes2013well,bickel2009simultaneous}.

These estimation procedures typically require full knowledge of the design matrix $\bX$. 
We consider a setting where a corrupted version of the data, $\bZ$ is observed instead of $\bX$.
This models the scenario where some covariates have missing entries.
We will be mainly interested in two types of patterns of missingness, outlined for
instance in the book~\cite{little2014statistical}: 
\begin{enumerate}[align=left]
\item {\bf Missing Completely at Random (MCAR)}: The most benign mechanism in which
each entry is missing independently of everything with probability $1-\alpha$:
\begin{align}
\label{eq:MCAR}
    Z_{ia} = \begin{cases} 
            X_{ia} & \text{ with probability } \alpha, \\
            \star & \text{ with probability } 1 - \alpha.
            \end{cases}
\end{align}
(The symbol $\star$ indicates that an entry is missing.)

\item {\bf Missing Not at Random  (MNAR)}: The general case where the
missingness pattern $(\mathbf{1}_{\{Z_{ia} = \star\}})_{1 \le i \le n, 1 \le a \le p}$ is composed of i.i.d.\ arbitrarily distributed rows.
\end{enumerate}
Given the response vector $\by$ and the observed data $\bZ$, how can we estimate
$\bbe_0$?

\paragraph{An example: standard Gaussian design.} We illustrate our approach in
the simple case where the design matrix $\bX$ is i.i.d.\ normal $X_{ij} \sim
\normal(0,1)$, and assume MCAR with parameter $\alpha$. Let us additionally
consider the case that there is
no additive noise. We thus observe the pair $(\by,\bZ)$ given by 
\begin{align*}
    \by = \bX\bbe_0, \qquad Z_{ia} = \begin{cases} 
            X_{ia} & \text{ with probability } \alpha, \\
            \star & \text{ with probability } 1 - \alpha.
            \end{cases}
\end{align*}
One natural strategy is to then impute the missing entries with some reasonable proxy.  
Since we know the data is standard normal, a first
attempt could be to replace each missing entry with its mean; that
is, whenever an entry is missing, give it the value $0$.  We thus construct the
imputed matrix
\begin{align*}
    \widehat{X}_{ia} = \begin{cases} 
        Z_{ia} & \text{if} \ Z_{ia} \neq \star, \\
            0 & \text{ otherwise}. 
            \end{cases}
\end{align*}
We can now run the LASSO with data $(\by,\widehat{\bX})$ and regularization parameter 
$\lambda = \sqrt{\frac{(1 - \alpha) \log{p}}{\alpha n}}$ (this choice is justified in Corollary~\ref{cor:IdentityCovariance}):
\begin{align}
\label{eq:ImputeLasso}
\widehat{\bbe} \in \argmin_{\bbe \in \mathbb{R}^p}\left\{\frac{1}{2n}\norm{\by -
\widehat{\bX}\bbe}_2^2 + \lambda\norm{\bbe}_1\right\}.
\end{align}
We plot the error $\|\widehat{\bbe} - \bbe_0\|_2$ in
Figure~\ref{fig:IdentityCovarianceGaussian} as a function of $\alpha$, compared
to the scaling of the theoretical error bound of Corollary~\ref{cor:IdentityCovariance} in
solid line (letting the universal constant pre-factor be one). The figure shows that this strategy recovers the regression vector $\bbe_0$ at the optimal rate $\sqrt{\frac{s \log p}{n}}$ adjusted with a $\alpha$-dependent term accounting for the missingness.   

\begin{figure}[h!]
\centering
\begin{tikzpicture}
\begin{axis}[
    width=0.7\textwidth,
    height=0.4\textwidth,
    title={MSE Imputation for Identity Covariance Gaussian},
    xlabel={Density of observed entries $\alpha$},
    ylabel={$\norm{\widehat{\bbe} - \bbe_0}_2$},
    ymode=log,
    xmin=0.5, xmax=1,
    xtick={0.5, 0.6, 0.7, 0.8, 0.9, 1.0},
    legend pos=south west,
    ymajorgrids=true,
    grid style=dashed,
]

\addplot[
    color=black,
    mark=*,
    mark size=1.25pt   
    ]
    table[x=alpha,y=err] {data/IdentityGaussianBig.dat};
    \legend{Empirical Error}

\addplot [name path=upper,draw=none] table[x=alpha,y=max_err]
    {data/IdentityGaussianBig.dat};
\addplot [name path=lower,draw=none] table[x=alpha,y=min_err]
    {data/IdentityGaussianBig.dat};
\addplot [fill=black!10] fill between[of=upper and lower];

\addplot [
    domain=0.5:0.99, 
    samples=100, 
    color=black,
    line width=1.5pt
]
    {0.498149264028527105 * sqrt{(1-x)/(x)}};
    \addlegendentry{$\alpha \mapsto \sqrt{\frac{\left(1 - \alpha\right)s
    \log{p}}{\alpha n}}$}
\end{axis}
\end{tikzpicture}
\caption{MSE as function of $\alpha$. We used
    $n=1000,\ p=1200,\ s=\lceil \sqrt{p} \rceil = 35$ (square root sparsity).
    Each data point is an average of 100 trials.  Shaded area: errors obtained
    over the $50$ trials. The solid line shows the scaling of the bound shown
    in Corollary~\ref{cor:IdentityCovariance}.  Note that this bound holds for $\alpha >
    0$, and we appeal to classical results on sparse linear regression when
    $\alpha = 0$.}
\label{fig:IdentityCovarianceGaussian}
\end{figure}
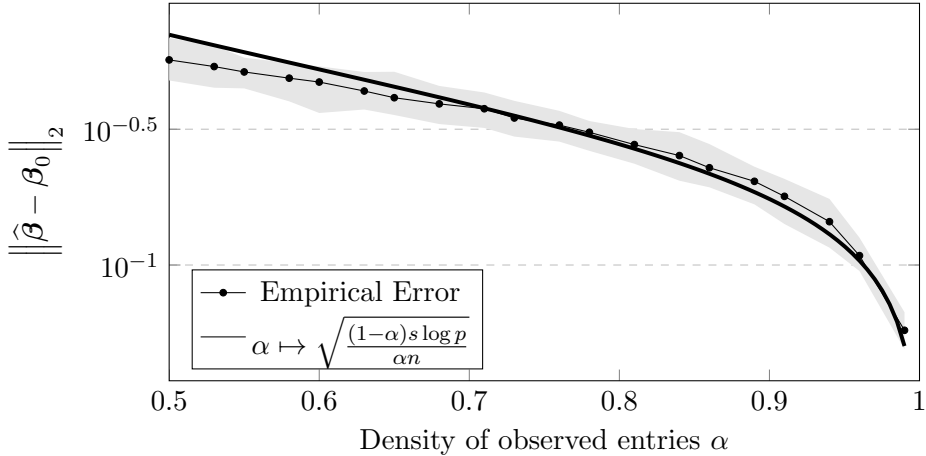

How should we generalize the above success? Consider re-writing
\begin{align}
    \label{eq:decomp}
    \by = \bX\bbe_0 + \beps = \widehat{\bX}\bbe_0 + \left(\bX -
    \widehat{\bX}\right)\bbe_0 + \beps \equiv \widehat{\bX}\bbe_0 +
\widetilde{\beps}.
\end{align}
Now, notice that if we take $\widehat{\bX} = \E\{\bX \mid \bZ \}$, by the orthogonality of conditional expectation,
$\widehat{\bX}\bbe_0$ and  $(\bX -
\widehat{\bX})\bbe_0$ are uncorrelated and we can intuitively think of the
problem as a linear model with data $\widehat{\bX}$ and noise
$\widetilde{\beps}$.  Using this intuition, we analyze the strategy of
imputation by conditional expectation and present the following contributions:

\vspace{2mm}
\textbf{Rate-optimal estimation via imputation and the LASSO.} Assuming the rows
of $\bX$ are i.i.d.\ sub-Gaussian, we show that imputation with conditional
expectation followed by LASSO, Eq.~\eqref{eq:ImputeLasso}, retains rate-optimal
statistical guarantees, regardless of the model of missingness. 

\vspace{2mm}
 \textbf{Rate-optimal and pivotal estimation via imputation and the square-root
 LASSO.} Using the same imputation strategy, we show that the minimizers of
 square-root LASSO program~\eqref{eq:ImputeSqrtLasso} introduced by Belloni et al.~\cite{belloni2011square} 
 \begin{align}
\label{eq:ImputeSqrtLasso}
\widehat{\bbe} \in \argmin_{\bbe \in \mathbb{R}^p}\left\{\frac{1}{\sqrt{n}}\norm{\by -
\widehat{\bX}\bbe}_2 + \lambda\norm{\bbe}_1\right\},
\end{align}
retain rate-optimal guarantees and are moreover \emph{pivotal} with respect to
noise variance $\sigma^2$ and the radius of the problem $\norm{\bbe_0}_2$.  That
is, the appropriate choice of the regularization parameter $\lambda$ in
Eq.~\eqref{eq:ImputeSqrtLasso} does not require the knowledge of these
quantities, nor do they need to be estimated in order to optimally implement the
square-root LASSO.
We emphasize that both \eqref{eq:ImputeLasso} and \eqref{eq:ImputeSqrtLasso} can be implemented using existing packages for the LASSO and square-root LASSO; the statistician needs only to operate on the observed data matrix $\bZ$.

\vspace{2mm}
 \textbf{Imputation with estimated conditional expectation.} We provide two examples of data from Gaussian graphical models and leverage this structure to approximately compute the conditional expectation given the observed data.  We provide statistical guarantees for both of these cases, noting that in our results we use the entire dataset to estimate the conditional expectation, and then run the LASSO on this data. The structure of the graphical model is used in a crucial way to design tractable estimators; we expect similar results to hold for more general models retaining such structure.    

\vspace{2mm}
We organize our results as follows:
\begin{itemize}
\item Section~\ref{sec:ConditionalExpectation} contains our main results in
    the case in which we give rates of statistical
error of the convex program~\eqref{eq:ImputeLasso} under MCAR and the assumption that a
conditional expectation can be computed exactly.  In particular, we show that our estimator $\widehat{\bbe}$ 
 has the optimal dependence on each on the parameters $R, \sigma, p, n, s$ for a design with i.i.d.\ sub-Gaussian rows. Moreover, when the covariance is the identity, our upper bound is optimal in the parameter $\alpha$ as well. 
 We additionally state our result regarding rates when the data is MNAR. The proofs for the results of this
section can be found in appendix~\ref{sec:ProofsConditionalExpectation}. 

\item Section~\ref{sec:pivotal} states our results on pivotal estimation; the
proof for the main result of this section can be found in
appendix~\ref{sec:proofssqrtlasso}.

\item Section~\ref{sec:ApproximateConditionalExpectation} provides two
examples of cases where the conditional expectation may not be available, but an
approximate conditional expectation may be computed.  In particular, our first
example of autoregressive--AR(1)--model provides a case in which an
approximate conditional expectation yields optimal results.  We then provide rates for Gaussian design with sparse inverse
covariance matrix.  The proofs of the results of this section can be found in
appendix~\ref{sec:ProofsApproximateConditionalExpectation}.
\item Section~\ref{sec:Simulations} provides simulation results, giving
    numerical evidence of our results; both on synthetic and semi-synthetic data.
\end{itemize}


\section{Related literature}
Estimation in the presence of missing data has been studied for decades.  An
in-depth overview of techniques can be found in~\cite{little2014statistical}.  The problem of high-dimensional
linear regression with missing data was first studied theoretically in~\cite{rosenbaum2010sparse}.   The first rate optimal
theoretical results in this direction were given in a sequence of papers by
Loh and Wainwright~\cite{loh2012high,loh2012corrupted}.  The main idea in these
papers is to recast the LASSO optimization as a quadratic program taking the
covariance matrix of the (fully observed) covariates as input. In the missing
data case, the authors follow a plug-in principle and construct an unbiased
estimator of this covariance to be used in lieu of the full covariance. The fact
that the some covariates are not observed is accounted for by subtracting a
diagonal term in their estimator which causes it to be non-positive
semidefinite. Consequently the resulting optimization problem is no longer
convex. Remarkably, the authors are able to show that a simple projected gradient
descent procedure reaches a near-optimal point with high probability, thereby
producing a good estimate of $\bbe_0$. Following this result, many authors have
proposed using this plug-in principle for various sparse recovery
algorithms such as orthogonal matching pursuit (OMP)~\cite{chen2013noisy} and
the Dantzig Selector~\cite{wang2017rate}. Other convex
surrogates for this problem have been given
by~\cite{rosenbaum2013improved,belloni2017linear, datta2017cocolasso} and the more complicated case of dependent measurements has been tackled by Rudelson and Zhou~\cite{rudelson2017errors}.  Notably, Belloni et al.~\cite{belloni2017pivotal} proposed a pivotal estimator for this problem based on the idea of self-normalization.  Our approach differs from these in its use of imputation.  In particular, we do not design any new algorithms for linear regression; rather, we use existing algorithms (LASSO and square-root LASSO) in their most `vanilla' version and analyze their statistical guarantees under a particular imputation strategy.  

More recently Agarwal et al.~\cite{agarwal2019model} consider a different setting where no sparsity is assumed on $\bbe_0$ but the covariates have a low-dimensional structure. They propose a matrix estimation approach for imputation followed by a simple least squares method. We finally mention that beyond linear regression, other models of high-dimensional statistical problems have been studied in the setting of missing data, such as covariance estimation~\cite{lounici2014high}, sparse principal component analysis~\cite{elsener2018sparse,lounici2013sparse}, and precision matrix estimation~\cite{kolar2012estimating,fan2019precision}.

\paragraph{Notation and basic notions.}
We use bold-face lower-case letters to denote vectors ($\bw, \bv, \dots$) and
bold-face upper-case letters to denote matrices ($\bX, \bZ, \dots$).  Rows of
matrices will be denoted by the letters $(i, j, \dots)$ and columns by $(a, b,
\dots)$.  Accordingly we will refer to rows of the matrix $\bX$ as $(\bX_i,
\bX_j, \dots)$ and its columns as $(\bX_a, \bX_b, \dots)$.  Additionally, for a
vector $\bx \in \mathbb{R}^{d}$, and a set $T = \{i_1, i_2, \dots, i_{\lvert T
\rvert}\} \subseteq [d]$, we will use the notation $\bx_T$ to denote the vector
$(x_{i_1}, x_{i_2}, \dots, x_{i_{\lvert T \rvert}})$ of length $\lvert T
\rvert$.  

A mean zero random variable $\bx \in \mathbb{R}$ is called $\sigma_X^2$-sub-Gaussian if
\[\E\{e^{\theta X}\} \leq e^{\frac{\sigma_X^2 \theta^2}{2}} \qquad \forall \theta \in \mathbb{R}.\]  
It is instead called $(\sigma_X^2,b)$-subexponential if the above holds for all $\theta \le 1/b$.
By extension, a random vector $\bx$ is sub-Gaussian (sub-exponential) if $\inp{\bu}{\bx}$ is sub-Gaussian (sub-exponential) for all unit norm vectors $\bu$.

Throughout the paper, we will use $C$ to denote
a universal constant that may change from line to line, and $C(\cdot)$ as a constant
depending only on its arguments.  We will write $f \lesssim g$ to denote $f \leq Cg$ (the symbol $\gtrsim$ is similarly defined), and we write $f \asymp g$ if $f \lesssim g$ and $f \gtrsim g$.

We will write MCAR($\alpha$) to indicate that the data is MCAR with probability
of observing an entry equal to $\alpha$.  We will assume $\alpha$ to be a
constant throughout the paper.

\section{Imputation by conditional expectation}
\label{sec:ConditionalExpectation}
In this section, we state our results on the statistical error of our estimator
$\widehat{\bbe}$; the proofs of each statement can be found in
Appendix~\ref{sec:ProofsConditionalExpectation}.  Recall that $\widehat{\bX} =
\E \left\{\bX \mid \bZ\right\}$. We will denote its covariance matrix  by $\Sigma_{\widehat{\bX}} \equiv \frac{1}{n}\E \widehat{\bX}^\top\widehat{\bX} \in \mathbb{R}^{p \times p}$.  Throughout this section, we will make the following assumptions:
\begin{itemize}
    \item[\bf A1.]\label{assumption:A1} The rows of $\bX$ are i.i.d.\ zero-mean and $\sigma_X^2$ sub-Gaussian with covariance matrix $\Sigma_{\bX}$.   
    \item[\bf A2.] The regression vector $\bbe_0$ is $s$-sparse: $|\text{supp}(\bbe_0)| = s$. Additionally, $\norm{\bbe_0}_2 \leq R$.
    \item[\bf A3.] There exists a constant $c_{\min} > 0$ such that the minimum
        eigenvalue of $\Sigma_{\widehat{\bX}}$ satisfies
        $\lambda_{\min}(\Sigma_{\widehat{\bX}}) \geq c_{\min}$.
    \item[\bf A4.] There exists a constant $c_0 > 0$ such that 
    $n \geq c_0\max\left(\frac{\sigma_X^4}{\lambda_{\min}(\Sigma_{\widehat{\bX}})^2},
\frac{\sigma_X^2}{\lambda_{\min}(\Sigma_{\widehat{\bX}})}\right)s\log{p}$.
\end{itemize}
 
\begin{theorem}
\label{thm:mainres}
Assume \textbf{A1--A4} and that the data is MCAR($\alpha$) for $\alpha < 1$.  Assume
additionally that $\sqrt{\frac{\log{p}}{n}} \leq c_1\sqrt{1 -
    \alpha}$ for a positive constant $c_1$.
Then, there exist positive constants $c_2, c_3, c_4$ such that with probability
at least $1 - c_2p^{-c_3}$, $\widehat{\bbe}$
as defined in~\eqref{eq:ImputeLasso} with regularization parameter
    \[\lambda = c_4\left(\sigma_X \sigma + \sigma_X^2\sqrt{1 - \alpha}R\right)\sqrt{\frac{\log{p}}{n}},\]
satisfies
    \begin{align}\norm{\widehat{\bbe} - \bbe_0}_2
        \lesssim\frac{\sigma_X \sigma + \sigma_X^2 \sqrt{1 - \alpha}R}{\lambda_{\min}(\Sigma_{\widehat{\bX}})}\sqrt{\frac{s\log{p}}{n}}\, .
    \end{align}
\end{theorem}

\noindent Specializing the analysis for the special case of Gaussian design with
identity covariance, $\bX_i \sim \normal(0, \bI_{p\times p})$, we have the following
corollary which establishes that for this special case, our imputation estimator is
rate optimal, in a minimax sense, with respect to \textit{all} the parameters $\sigma, \alpha, n, p,
s$.

\begin{corollary}
\label{cor:IdentityCovariance}
Assume the data is MCAR($\alpha$) for $\alpha < 1$ and $X_{ij} \sim_{iid}
\normal(0,1)$. 
Assume additionally that $\sqrt{\frac{\log{p}}{n}} \leq 
    c_1\sqrt{\alpha(1-\alpha)}$ for a positive constant $c_1$.  Then, there exist positive constants $c_2,
    c_3, c_4$ such that with probability at least $1 - c_2p^{-c_3}$, $\widehat{\bbe}$ as defined in~\eqref{eq:ImputeLasso} with regularization parameter
    \[\lambda =  c_4 \sqrt{\alpha}\left(\sigma + \sqrt{1
    - \alpha}R\right)\sqrt{\frac{\log{p}}{n}},\]
satisfies
    \begin{align}\norm{\widehat{\bbe} - \bbe_0}_2
       \lesssim \frac{\sigma +
        \sqrt{1-\alpha}R}{\sqrt{\alpha}}\sqrt{\frac{s\log{p}}{n}}\, .
        \end{align}
\end{corollary}

Indeed, the bound displayed in Corollary~\ref{cor:IdentityCovariance} is matched by a lower bound  Wang et al.~\cite{wang2017rate}:
\begin{theorem}\textbf{\cite{wang2017rate}}
    \label{thm:Wang}
    Suppose $4 \leq s < 4p/5, \frac{s
    \log{p/s}}{\alpha n} \rightarrow 0$ and $X_{ij} \sim_{iid} \normal(0,1)$.
    Then, 
    \begin{align*}
    \inf_{\widehat{\bbe}}\sup_{\bbe_0 \in B_2(R) \cap B_0(s)} \E
    \norm{\widehat{\bbe} - \bbe_0}_2^2 &\gtrsim 
    \phi(R,\sigma,\alpha) \min\left\{\sqrt{\frac{s\log{p/s}}{(1 - \alpha)^2n}},
    \frac{s\log{p/s}}{\alpha n}\right\},
    \end{align*}
    where $\phi(R,\sigma,\alpha) = \min\big\{\sigma^2 + (1 - \alpha)R^2, e^{(1-\alpha)s}
        \sigma^2\big\}$.
\end{theorem}
\begin{remark}
 We observe that in the regime where $\sqrt{\frac{s\log{p/s}}{(1 - \alpha)^2n}}
 \geq 
    \frac{s\log{p/s}}{\alpha n}$ and $\sigma^2 + (1 -
        \alpha)R^2 \leq e^{(1-\alpha)s}
        \sigma^2$, Corollary~\ref{cor:IdentityCovariance} matches the lower bound in dependence on all
        parameters.  Indeed, the lower bound simplifies to
        \[\frac{\sigma^2 + (1 -\alpha)R^2}{\alpha}\frac{s\log{p/s}}{n}.\]
\end{remark}
\noindent Finally, we give analogous rates when the missing data mechanism is MNAR.  
\begin{theorem}
\label{thm:mnar}
Assume that the data is MNAR and \textbf{A1--A4}.  Assume additionally that
$\sqrt{\frac{\log{p}}{n}} \leq c_1$ for a positive constant $c_1$.  Then, there
exist positive constants $c_2, c_3, c_4$ such that with probability at least $1
- c_2p^{-c_3}$, $\widehat{\bbe}$ as defined in Eq.~\eqref{eq:ImputeLasso} with $\widehat{\bX} = \E\left\{ \bX \mid \bZ \right\}$ and regularization parameter
    \[\lambda = c_4\left(\sigma_X\sigma + \sigma_X^2 R\right)\sqrt{\frac{\log{p}}{n}},\]
satisfies
    \begin{align}\norm{\widehat{\bbe} - \bbe_0}_2
        \lesssim \frac{\sigma_X \sigma +
         \sigma_X^2 R}{\lambda_{\min}(\Sigma_{\widehat{\bX}})}\sqrt{\frac{s\log{p}}{n}}.
    \end{align}
\end{theorem}
\begin{remark}
    Consider the scenario in which $X_{ij} \sim \normal(0,1)$ and the data is
    MCAR($\alpha$).  The bound one gets from Theorem~\ref{thm:mnar} is $\frac{\sigma +
    R}{\alpha}\sqrt{\frac{s\log{p}}{n}}$, whereas the tighter result
    from Corollary~\ref{cor:IdentityCovariance} gives $\frac{\sigma +
        \sqrt{1-\alpha}R}{\sqrt{\alpha}}\sqrt{\frac{s\log{p}}{n}}$.

\end{remark}

\noindent Two remarks are in order: 
\begin{enumerate}
\item Our results dictate a specific choice of the regularization parameter $\lambda$ which depends on $\sigma_X, \sigma, R$ and $\alpha$ (in the MCAR case). Although $\alpha, \sigma_X$ may be estimated, it is difficult to estimate the noise variance $\sigma$, and in many situations it is difficult to determine $R$, which is a function of the regression vector we are aim to recover.
\item We assume that the conditional expectation $\widehat{\bX}$ can be computed exactly. This may not be realistic in practice; even with Gaussian designs, this quantity requires the knowledge of the covariance matrix.
\end{enumerate}
We address these two questions in the following two sections.  Regarding the first question, we show that the square-root LASSO with conditional expectation imputation is pivotal with respect to $\sigma, R$ and as such these parameters need not be estimated to set the regularization $\lambda$.  Regarding the second question, we show that if the covariates come from a graphical model, we can approximate conditional expectation efficiently to an accuracy sufficient for consistency of LASSO.



\section{Pivotal estimation: the square-root LASSO}
\label{sec:pivotal}
When the data matrix $\bX$ is known, the insight of~\cite{belloni2011square,antoniadis2010comments} was that the procedure
\[
\widehat{\bbe}_{\text{S}} \in \argmin_{\bbe \in \mathbb{R}^{p}}\left\{\frac{1}{
\sqrt{n}}\norm{\by - \bX \bbe}_2 + \lambda \norm{\bbe}_1\right\},
\]
known as the square-root LASSO,
retains the statistical guarantees of the LASSO, but does not require knowledge
of the noise
standard deviation $\sigma$ to choose $\lambda$.  The intuition that guided us
in the decomposition~\eqref{eq:decomp} leads to the insight that in the imputed
linear model, the effective noise standard deviation will be $\sigma + \sigma_X \sqrt{(1 -
\alpha)}R$.
The following theorem formalizes this intuition and shows
that the square-root LASSO with conditional expectation imputation,
\begin{align}
    \label{eq:squarerootlasso}
\widehat{\bbe} \in \argmin_{\bbe \in \mathbb{R}^{p}}\frac{1}{
\sqrt{n}}\norm{\by - \widehat{\bX} \bbe}_2 + \lambda \norm{\bbe}_1, 
\end{align}
retains such
a guarantee, showing that when using the square-root LASSO, the statistician
need not know $\sigma$ nor $R$ when picking the regularization constant
$\lambda$.
\begin{theorem}
\label{thm:sqrtlasso}
Assume \textbf{A1--A4} and that the data is MCAR($\alpha$).  Assume
additionally that there exists a constant $c_1 < 1$ such that $n \geq
\frac{1}{c_1}\left(\ln{\delta^{-1}}\right)\left(R^2\sigma_X^2 +
    \sigma^2\right)^2$.  Then, there
exist positive constants $c_2, c_3, c_4$ such that with probability at least $1
- \delta - c_2p^{-c_3}$, $\widehat{\bbe}$ as defined in~\eqref{eq:squarerootlasso} with
regularization parameter
\[\lambda = c_4\sigma_X\sqrt{\frac{\log{p}}{n}},\]
satisfies
\[
    \norm{\widehat{\bbe} - \bbe_0}_2 \lesssim \frac{\sigma_X\left(\sigma + \sigma_X
    R\right)}{\lambda_{\min}(\Sigma_{\widehat{\bX}})}\sqrt{\frac{s\log{p}}{n}}.
\]
\end{theorem}

\section{Imputation by approximate conditional expectation}
\label{sec:ApproximateConditionalExpectation}
A strong assumption in Section~\ref{sec:ConditionalExpectation} is the exact knowledge of
a conditional expectation $\E\left\{\bX_{i,S^c} \mid \bX_{i, S}\right\}$ for any
$S \subseteq [p]$.  In this section, we give two examples in which the exact
conditional expectation is not available, but an approximate version can be
computed.  
Both examples are Gaussian graphical models and our results exploit 
the graphical structure in a crucial way, allowing us to compute conditional expectations using only a small subset of the variables. 

In both examples, we deal with the uncertainty in the same way.  First, we use
\textit{all} of the data to estimate the covariance, and use this estimate
 to compute the conditional expectation. We then
run the LASSO  based on the imputed matrix in order to recover the regression vector $\bbe_0$.  
The Markov structure of these models is exploited algorithmically:
whenever an entry is missing, we need only to consider a small number of
observed nodes to estimate the missing entry.     

Let us note that in both cases, similar results can be shown through sample splitting.
For instance, if the statistician reserves $n/2$ samples for learning
the covariance and uses the remaining $n/2$ samples for regression,
rate-optimality can be shown in a manner similar to that of
Section~\ref{sec:ConditionalExpectation}.  We analyze the more challenging case in which
samples are re-used in order to keep higher fidelity to statistical practice. 

\subsection{Example \#1: AR(1) model}
\label{sec:FiniteStateMarkov}
We consider the autoregressive real-valued stationary process $X_t = \phi
X_{t-1} + Z_t$ where $Z_t \sim_{\text{i.i.d.}} \normal(0,1)$ with unknown
coefficient $\phi$ satisfying $\left\lvert{\phi}\right\rvert < 1$.  We form the rows of
the data matrix $\bX$ by sampling $p$ consecutive points $X_{0},
\dots, X_{p-1}$ from the stationary chain $(X_t)$, independently for
each row. The covariance matrix of each row is $\left(\Sigma_{\bX}\right)_{ij} =
\frac{1}{1 - \phi^2}\phi^{\left\lvert i - j \right\rvert}$.
 
Let $M_{ik}$ be the indicator of whether entry $(i,k)$ is observed or not.  
We find an estimate $\hat{\phi}$ of the true parameter $\phi$
from the observed data:
\begin{align}
\label{eq:ARParamApprox}
\hat{\phi} =
\frac{\frac{1}{\alpha^2 np}\sum_{i=1}^{n}\sum_{a=1}^{p-1}X_{ia}X_{i(a+1)}M_{ia}M_{i(a+1)}}{\frac{1}{\alpha n
p}\sum_{i=1}^{n}\sum_{a=1}^{p-1}X_{ia}^2 M_{ia}}.
\end{align}

 Suppose the entry $(i,k)$ is missing.  By the Markov property satisfied by this model (see
Figure~\ref{fig:graph_mod_AR1})  the conditional expectation $ \E\left\{X_{ik} \mid \bX_{iS}\right\}$ is a function on the closest
observed entries on either side of node $k$.  Using the formula for the conditional expectation of a multivariate Gaussian random variable, we have
\begin{align}
\label{eq:ARhat}
    \E\left\{X_{ik} \mid \bX_{iS}\right\} = \widehat{X}_{ik} = \frac{\phi^{d_1 + d_2}}{1 - \phi^{2(d_1 + d_2)}}
    \left(X_{i,L(k)}\left(\phi^{-d_2} - \phi^{d_2}\right) + X_{i, R(k)} \left(\phi^{-d_1} -
    \phi^{d_1}\right)\right),
\end{align}
where  $d_1 = k - L(k)$ and $d_2 = R(k) -k$ with $L(k)$ and $R(k)$ the positions of the closed observed entries to the left and right of $k$, respectively. 
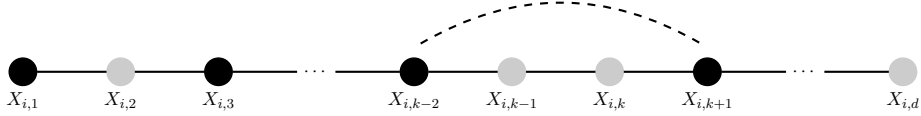
\begin{figure}[t!]
\centering
\begin{tikzpicture} 
\begin{scope}[scale=0.65, transform shape]
\GraphInit[vstyle=Classic]

\begin{scope}[VertexStyle/.append style = {minimum size = 16pt, 
                                           inner sep = 0pt,
                                           color=black,}]
    \Vertex[Lpos=below, L=$X_{i,1}$, x = -5, y=0]{1}
    \Vertex[Lpos=below, L=$X_{i,3}$, x = -1, y=0]{3}
    \Vertex[Lpos=below, L=$X_{i,k-2}$, x = 3, y=0]{k-2}
    \Vertex[Lpos=below, L=$X_{i,k+1}$, x = 9, y=0]{k+1}
\node[draw=none] (ellipsis1) at (1,0) {$\cdots$};
\node[draw=none] (ellipsis2) at (11,0) {$\cdots$};
    \node[draw=none] (sentinel1) at (3,0.5) {};
    \node[draw=none] (sentinel2) at (9,0.5) {};
\end{scope}
\begin{scope}[VertexStyle/.append style = {minimum size = 16pt, 
                                           inner sep = 0pt,
                                           color=gray!40,}]
    \Vertex[Lpos=below, L=$X_{i,2}$, x=-3, y=0]{2}
    \Vertex[Lpos=below, L=$X_{i,k-1}$, x = 5, y=0]{k-1}
    \Vertex[Lpos=below, L=$X_{i,k}$, x = 7, y=0]{k}
    \Vertex[Lpos=below, L=$X_{i,d}$, x = 13, y=0]{d}
\end{scope}
\begin{scope}[VertexStyle/.append style={rectangle,
                                             minimum size = 10pt,
                                             inner sep = 0pt,
                                             color=black,}]
\end{scope}
\begin{scope}[VertexStyle/.append style={rectangle,
                                             minimum size = 10pt,
                                             inner sep = 0pt,
                                             color=gray,}]
\end{scope}
    \Edge[](1)(2)
    \Edge[](2)(3)
    \Edge[](3)(ellipsis1)
    \Edge[](ellipsis1)(k-2)
    \Edge[](k-2)(k-1)
    \Edge[](k-1)(k)
    \Edge[](k)(k+1)
    \Edge[](k+1)(ellipsis2)
    \Edge[](ellipsis2)(d)
    \tikzset{EdgeStyle/.append style = {bend left, dashed}}
    \Edge[color=green, style = {-,bend left}](sentinel1)(sentinel2)
\end{scope}
\end{tikzpicture} 
    \caption{The graphical model for row $i$.  The gray nodes represent missing
    entries whereas the black nodes represent observed entries.  The observed
    entries at
    the end of the dashed path are used to estimate the missing entries
    $X_{i,k-1}, X_{i,k}$.  In this case, $L(k) = L(k-1) = k-2$ and $R(k) =
    R(k-1) = k+1$.}
\label{fig:graph_mod_AR1}
\end{figure}
We plug in the estimate $\hat{\phi}$ in lieu of $\phi$ to approximate conditional expectation:
\begin{align}
\label{eq:ARtilde}
&\widetilde{X}_{ik} = \frac{\hat{\phi}^{d_1 + d_2}}{1 - \hat{\phi}^{2(d_1 + d_2)}}
    \left(X_{i, L(k)}\left(\hat{\phi}^{-d_2} - \hat{\phi}^{d_2}\right) +
    X_{i, R(k)} \left(\hat{\phi}^{-d_1} -
    \hat{\phi}^{d_1}\right)\right).
\end{align}
We repeat this process for every missing entry to create the matrix $\widetilde{\bX}$ and proceed as in the previous section.  The
following result shows that this procedure indeed leads to rate optimal
estimation.


\begin{theorem}
\label{thm:ARres}
Assume \textbf{A1--A4}, the data is MCAR($\alpha$), that the sample size $n
\geq c_1\frac{1}{\alpha^8}s\log{p}$ for a positive constant $c_1$, and that the rows of $\bX$ are generated from the
stationary auto-regressive process described above with $\lvert \phi \rvert <
1$.  Then, there exist positive constants $c_2, c_3, c_4$ such that with
probability at least $1 - c_2p^{-c_3}$.
$\widehat{\bbe}$ as defined in~\eqref{eq:ImputeLasso}, using $\widetilde{\bX}$
in place of $\widehat{\bX}$, with regularization parameter
\begin{align}
    \label{eq:lambdaAR}
    \lambda = c_4\left(\frac{\sigma_X\sigma}{\alpha^2} +
    \frac{\sigma_X^2}{\alpha^4} R\right)\sqrt{\frac{\log{p}}{n}},
\end{align}
satisfies
    \begin{align}\norm{\widehat{\bbe} - \bbe_0}_2
        \lesssim \frac{\left(\frac{\sigma_X\sigma}{\alpha^2} + \frac{\sigma_X^2}{\alpha^4} R\right)}{\lambda_{\min}(\Sigma_{\widehat{\bX}})}\sqrt{\frac{s\log{p}}{n}}.
    \end{align}
\end{theorem}

\begin{remark}
    Note that in this case the model is fully explicit and $\Sigma_{\bX}$ is
    Toeplitz.  One can thus compute for
instance $\sigma_X^2 \leq (1 - \phi^2)^{-1}(1 - \phi)^{-1}$.  Computing
$\lambda_{\min}(\Sigma_{\widehat{\bX}})$ is more difficult; we provide in
Appendix~\ref{subsec:verifylambdamin} a lower bound for a restricted range of
$\alpha$, $\phi$.
\end{remark}

The proof of this theorem can be found in appendix~\ref{sec:ProofsApproximateConditionalExpectation}.

\subsection{Example \#2: Gaussian design with sparse precision matrix}
\label{sec:SparsePrecision}
We now generalize the strategy of the previous section to the case of Gaussian rows 
with sparse precision matrices.  We consider a zero mean normal distribution 
with covariance $\bSig$ and precision
matrix $\bOmega = \bSig^{-1}$.  We make the following assumptions:
\begin{itemize}
    \item[\bf C1.] There exist positive constants $0 < \underline{c} <
        \overline{c}$ such that the eigenvalues of $\bSig$ satisfy $\underline{c} \leq \lambda_{\min}(\bSig)
        \leq \lambda_{\max}(\bSig) \leq \overline{c}$.  Additionally, rows of $\bX$ are drawn i.i.d.\ from the distribution $\normal(0,
        \bSig)$.
    \item[\bf C2.]  Each row of $\bOmega$ as at most $d_{\max}$ non-zero
        entries, where $d_{\max}$ satisfies $(1 - \alpha)(d_{\max} - 1) < 1$ and
        the sparsity pattern of $\bOmega$ is known to the statistitian.
    \item[\bf C3.] There exists a constant $C(\alpha, d_{\max}) > 0$ depending only
        on $\alpha, d_{\max}$ such that the sample size $n$ satisfies $n \geq  C(\alpha, d_{\max})\max\left(\frac{\sigma_X^4}{\lambda_{\min}(\Sigma_{\widehat{\bX}})^2},
            \frac{\sigma_X^2}{\lambda_{\min}(\Sigma_{\widehat{\bX}})}\right)\max\left(s^2
        \log{p}, \log ^7{p}\right)$.
 \end{itemize}
\begin{remark}
Although the sparsity pattern is assumed to be known to the statistician, under
the additional technical assumption of \textit{irrepresentability} of the
graphical model \cite{ravikumar2011high}, the sparsity pattern can be found
with high probability using for example the graphical
LASSO \cite{kolar2012estimating,ravikumar2011high}.
\end{remark}

We are interested in computing the conditional expectation of the missing entries $\E \{ \bX_{iS^c} \mid
\bX_{iS} \}$, which in this Gaussian setting is given by the formula $\bSig_{S^c, S} \big(\bSig_{S, S}\big)^{-1} \bX_{iS}$.  We require some definitions.
\begin{definition}
\label{def:sparsitygraph}
Let the matrix $\bA$ with $A_{ij} = \mathbbm{1}\left\{\Omega_{ij} \neq 0\right\}$ denote the adjacency matrix of the graph $G = (V=[p], E)$.  Let $\left(G_i\right)_{i \in [n]}$ be $n$ copies of $G$, where for each $a \in [p]$, vertex $a$ of $G_i$ is ``closed" if $Z_{ia} = \star$ and ``open" otherwise. 
\end{definition}


\begin{definition}
\label{def:markovblanket}
The Markov blanket $S_{(i,a)}$ of the vertex $a$ in graph $G_i$ is the set of
first open nodes encountered by all walks in $G_i$ starting at vertex $a$.
\end{definition}
\begin{remark}
    The assumption $(1 - \alpha)(d_{\max} - 1) < 1$ in \textbf{C2}
    corresponds to the threshold of the Bernoulli site percolation process on the infinite 
    $d_{\max}$-regular tree, and is used to control the size of the Markov
    blanket of a given vertex.
\end{remark}

Now to compute $\E \{ \bX_{iS^c} \mid
\bX_{iS} \}$, we exploit the graphical structure of the model.
If an entry $(i,k)$ is missing, $X_{ik}$ is conditionally
independent of all observed entries not in its Markov blanket, given the latter.  That is, consider a node $X_{i,k}$ and let $S(i,k) \subseteq S$ be the subsets of observed nodes connected to $X_{i,k}$ by paths which contain only missing nodes; see Figure~\ref{fig:graph_mod_precision}.       
It follows from this observation and the conditional independence structure of Gaussian graphical models that 
\[ \widehat{X}_{ik}  =   \E \{X_{ik} \mid \bX_{iS}\} = \bSig_{k, S(i,k)}   \big(\bSig_{S(i,k), S(i,k)}\big)^{-1} \bX_{S(i,k)}.    \]
Now we simply use a plug-in estimator for $ \bSig$ in order to estimate the conditional expectations:
\[ \widetilde{X}_{ik} = \widetilde{\bSig}_{k, S(i,k)}  \big(\widetilde{\bSig}_{S(i,k), S(i,k)}\big)^{-1} \bX_{S(i,k)}.\]
%
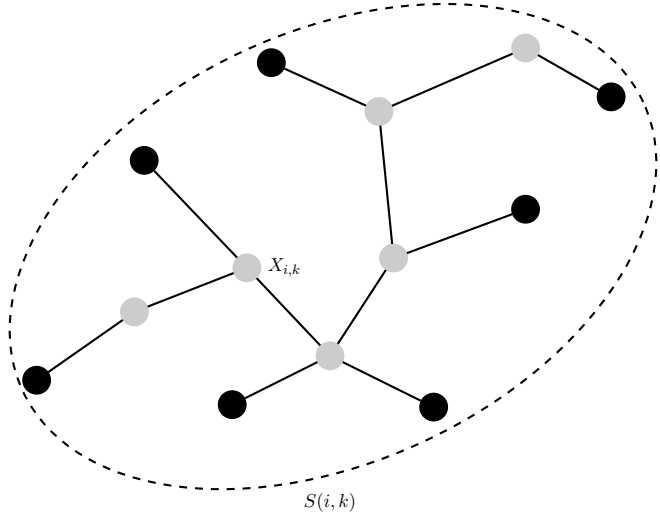
\begin{figure}[t!]
\centering
\begin{tikzpicture} 
\begin{scope}[scale=0.65, transform shape]
\GraphInit[vstyle=Classic]

\begin{scope}[VertexStyle/.append style = {minimum size = 16pt, 
                                           inner sep = 0pt,
                                           color=black,}]
    \Vertex[x=-2,y=-1, NoLabel]{1}
    \Vertex[x=2.12, y=-1.05, NoLabel]{2}
    \Vertex[x=4, y=3, NoLabel]{5}
    \Vertex[x=5.75, y=5.3, NoLabel]{8}
    \Vertex[x=-1.2, y=6, NoLabel]{9}
    \Vertex[x=-3.8, y=4, NoLabel]{11}
    \Vertex[x=-6, y=-0.5, NoLabel]{12}
\end{scope}
\begin{scope}[VertexStyle/.append style = {minimum size = 16pt, 
                                           inner sep = 0pt,
                                           color=gray!40,}]

\Vertex[x=0, y = 0, NoLabel]{0}
    \Vertex[x=1.3, y=2, NoLabel]{3}
    \Vertex[L=$X_{i,k}$, x=-1.7, y=1.8]{k}
    \Vertex[x=1, y=5, NoLabel]{6}
    \Vertex[x=4, y=6.3, NoLabel]{7}
    \Vertex[x=-4, y=0.9, NoLabel]{10}

\end{scope}
\begin{scope}[VertexStyle/.append style={rectangle,
                                             minimum size = 10pt,
                                             inner sep = 0pt,
                                             color=black,}]
\end{scope}
\begin{scope}[VertexStyle/.append style={rectangle,
                                             minimum size = 10pt,
                                             inner sep = 0pt,
                                             color=gray,}]
\end{scope}
    \Edge[](0)(1)
    \Edge[](0)(2)
    \Edge[](0)(3)
    \Edge[](0)(k)
    \Edge[](3)(5)
    \Edge[](3)(6)
    \Edge[](6)(7)
    \Edge[](7)(8)
    \Edge[](6)(9)
    \Edge[](k)(10)
    \Edge[](k)(11)
    \Edge[](10)(12)
    
    \draw[thick,dashed,rotate=25] (1cm,2cm) ellipse (7 and 4.4);
    \node[draw=none] (label) at (0, -3) {$S(i,k)$};
\end{scope}
\end{tikzpicture} 
    \caption{The Markov blanket used to estimate the missing entry $X_{i,k}$.
    Each of the gray nodes in the neighborhood of $X_{i,k}$ denotes a missing
    entry whereas each of the black nodes denotes an observed entry.  
    Each of the values of the observed entries on the boundary are used to
    estimate $X_{i,k}$.}
\label{fig:graph_mod_precision}
\end{figure}
In order to define our estimator $\widetilde{\bSig}$ of the covariance, we let the zero-imputed design matrix $\bX_0$ be such that
\[X_{0, (i,a)} = \begin{cases} 
        X_{ia} & \text{if} \ Z_{ia} \ \text{observed}, \\
            0 & \text{ otherwise}. 
            \end{cases}\]
We then take our estimator $\widetilde{\bSig}$  of the covariance to be the empirical covariance matrix of 
 the zero-imputed design matrix $\bX_0$:
\begin{align*}
    \widetilde{\bSig} = \frac{1}{\alpha^2 n}\sum_{i=1}^{n}\bX_{0,i} \bX_{0,i}^\top -
        \frac{1 - \alpha}{\alpha^2 n}\sum_{i=1}^{n}\diag\left(\bX_{0,i}
        \bX_{0,i}^\top\right).
\end{align*}
This is an unbiased estimator of the covariance matrix and is constructed by modifying the usual empirical covariance to account for the missing data.  
Before stating the main theorem of this section, we summarize the algorithm above:

\vspace{.2cm}
\begin{algorithm}[H]
\SetAlgoLined
\KwData{Matrix $\bZ$, vector $\by$, adjacency matrix for graphical model $\bA$, probability of observed entry $\alpha$, parameter $\lambda$}
\KwResult{Estimated regression vector $\widehat{\bbe}$}
Initialize $\widetilde{\bX} = 0$\;
Compute $\widetilde{\bSig} = \frac{1}{\alpha^2 n}\sum_{i=1}^{n}\bX_{0,i} \bX_{0,i}^\top -
        \frac{1 - \alpha}{\alpha^2 n}\sum_{i=1}^{n}\diag\left(\bX_{0,i}
        \bX_{0,i}^\top\right)$\;
\For{$i \in [n]$}{
    \For{$a \in [p]$}{
        \If{$\bM_{ia} = 1$}{
            Let $\widetilde{X}_{ia} = Z_{ia}$\;
        }
        \Else{
            Let $G_i \leftarrow \bA$ \;
            Color each node in $G_i$ green if observed, red if missing \;
            Let $S(i,k) \leftarrow $ boundary of breadth-first search in $G_i$ started from node $k$ and terminated when each path reaches a green node\;
            Let $\widetilde{X}_{ia} = \widetilde{\bSig}_{k, S(i,k)}  \big(\widetilde{\bSig}_{S(i,k), S(i,k)}\big)^{-1} \bX_{S(i,k)}$\;
        }
    }
}
\KwRet $\widehat{\bbe} \in \argmin_{\bbe \in \mathbb{R}^p} \frac{1}{2n}\norm{\by - \widetilde{\bX}\bbe}_2^2 + \lambda \norm{\bbe}_1$
\caption{Imputation and regression for sparse gaussian graphical models}
\end{algorithm}
\vspace{.2cm}

The strategy outlined above leads to the following
theorem.

\begin{theorem}
    \label{thm:GaussianRes}
Assume \textbf{A2--A3}, \textbf{C1--C3}, the data
is MCAR($\alpha$), and $\sqrt{\frac{\log{2np}}{n}} \leq C(\alpha, d_{\max},
\underline{c}, \overline{c})$ for $C(\alpha, d_{\max},
\underline{c}, \overline{c})$ >0.  Then, there exist positive constants $c_1,c_2,
c_3, C(\alpha, d_{\max})$, such that with probability at least $1 - c_1n^{-1} - c_2p^{c_3}$,
$\widehat{\bbe}$ as defined in ~\eqref{eq:ImputeLasso} with regularization parameter
\[
\lambda = C(\alpha, d_{\max})\left(\sigma + R\right)\sqrt{\frac{s\log{np}}{n}},
\]
satisfies
\begin{align}
\norm{\widehat{\bbe} - \bbe_0}_2
        \leq \frac{C(\alpha, d_{\max})\left(\sigma + R\right)
        }{\lambda_{\min}(\bSig_{\hat{\bX}})}s\sqrt{\frac{\log{2np}}{n}}.
\end{align}
\end{theorem}
\begin{remark}
    Observe the excess multiplicative factor $\sqrt{s}$ in the last display of the theorem. This can be avoided
    using sample-splitting: using the first $n/2$ samples to estimate the
    covariance and the remaining $n/2$ for regression.  This extra
    factor is a byproduct of the analysis of our procedure which reuses  in regression the data already used for covariance estimation. 
\end{remark}

\section{Simulations}
\label{sec:Simulations}
We now provide five numerical examples to support our theoretical findings:
\begin{enumerate}
\item \textbf{Standard Gaussian design: LASSO.}  Here, we elaborate on the experiment to produce Figure~\ref{fig:IdentityCovarianceGaussian}. 
\item \textbf{AR(1) Approximate conditional expectation.}  We generate data according to an AR(1) model and then use our imputed LASSO after estimating the covariance of the model.
\item \textbf{Banded inverse covariance approximate conditional expectation.} We
generate data with a banded inverse covariance and then use the imputed LASSO after
estimating the covariance of the model.
\item \textbf{Semi-synthetic data: gene expression.} We use the \emph{gene expression
    cancer RNA-Seq}
data from~\cite{Dua2019}.  We artificially induce MCAR data, generate a
synthetic regression vector $\bbe_0$, and then perform our imputed LASSO.
\item \textbf{Real data: communities and crime.} We use the \emph{communities and
    crime} data from~\cite{Dua2019}.  We compare the prediction error of the
    LASSO with the prediction error of the imputed LASSO. Although we impose
    MCAR data artificially, we no longer create a linear model; rather, the
    dataset contains responses.   
\end{enumerate}

Our simulations make use of the package {\fontfamily{qcr}\selectfont
scikit-learn} to compute the LASSO estimate.

\subsection{Standard Gaussian design: LASSO}
We first simulate the simple setting in which each entry of the data matrix
$X_{ij} \sim_{i.i.d} \normal(0,1)$.  We then form $\widehat{\bX} = \E \{\bX \mid
\bZ\}$ by $0$ imputation.
That is,
\[\widehat{X}_{ij} = \begin{cases}
      X_{ij} & \text{if}\ Z_{ij} = X_{ij}, \\
      0 & \text{if}\ Z_{ij} = \star.
    \end{cases} \]
We isolate the effective noise caused by missing data by considering the
noiseless setting (i.e., $\sigma = 0$). Additionally, we generate $\bbe_0$ with
square root sparsity, setting each of the first $\lceil \sqrt{p} \rceil$ entries
of $\bbe_0$ to be $1$ and the remaining ones to be $0$.  We set the regularization
parameter to be
\[\lambda =
    \sqrt{\frac{\alpha \left(1 - \alpha\right) \log{p}}{n}}.\]
The results of the simulation, compared with the value $\lambda \sqrt{s}$
are shown in Figure~\ref{fig:IdentityCovarianceGaussian}.

\subsection{AR(1) Approximate conditional expectation}
We consider the AR(1) model of Section~\ref{sec:FiniteStateMarkov}.  Two cases
are simulated: when the parameter $\phi$ is known, and when it needs to be
estimated.  We consider the noiseless additive regime and plot our results in
Figure~\ref{fig:ARGaussian}.  In our simulations, we take
\[\lambda = \frac{1}{\alpha^4}R\sqrt{\frac{\log{p}}{n}}.\]
We simulate three cases of the parameter $\phi = \{0.05, 0.1, 0.15\}$. 
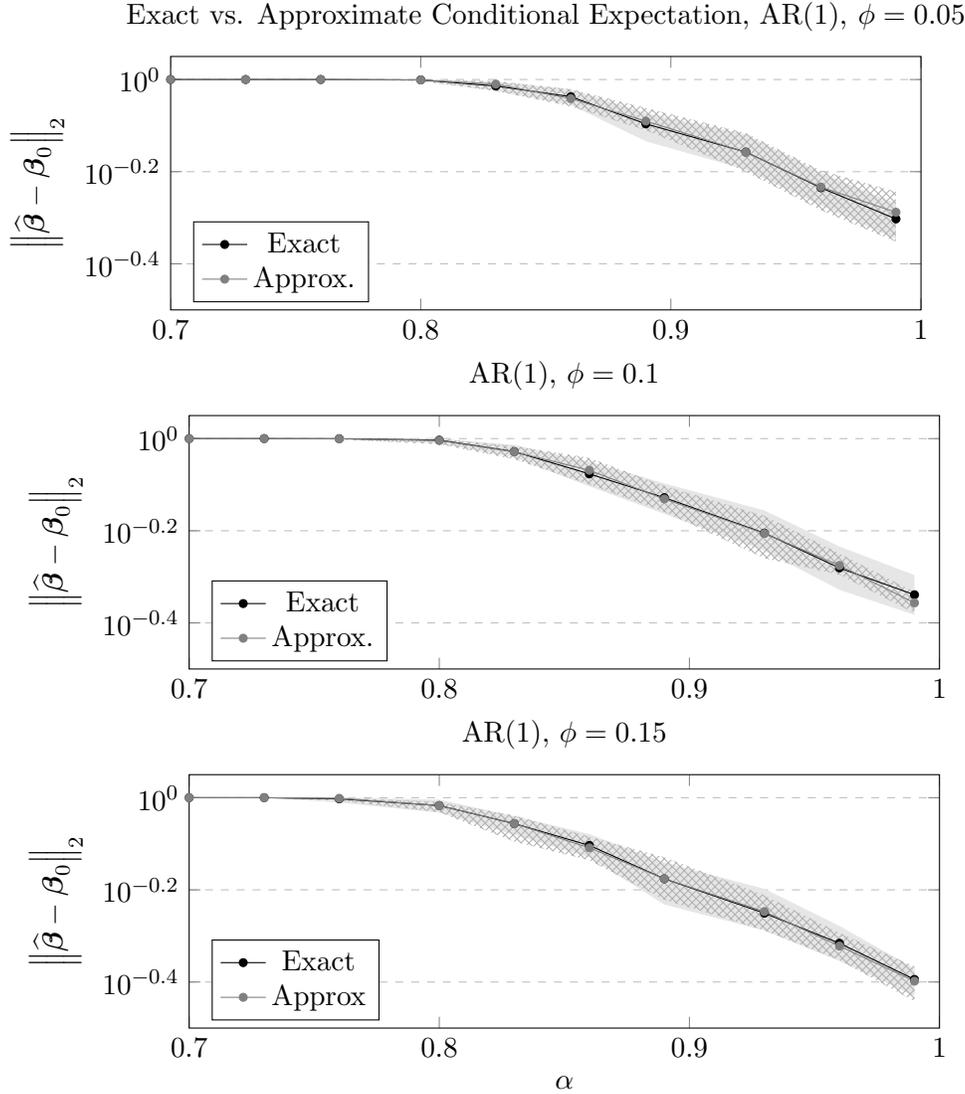
\begin{figure}[t!]
\centering
    \begin{subfigure}[h]{\textwidth}
\centering
\begin{tikzpicture}
\begin{axis}[
    width=0.7\textwidth,
    height=0.3\textwidth,
    title={Exact vs. Approximate Conditional Expectation, AR(1), $\phi = 0.05$},
    ylabel={$\norm{\widehat{\bbe} - \bbe_0}_2$},
    ymode=log,
    xmin=0.7, xmax=1,
    ymin=10^-0.5,
    xtick={0.7, 0.8, 0.9, 1.0},
    legend pos=south west,
    ymajorgrids=true,
    grid style=dashed,
]

\addplot[
    color=black,
    mark=*,
    mark size=1.5pt   
    ]
    table[x=alpha,y=err] {data/ARGaussian0.05.dat};
    \legend{Exact}

\addplot[
    color=black!50,
    mark=*,
    mark size=1.5pt,
    ]
    table[x=alpha,y=apx_err] {data/ARGaussian0.05.dat};
    \addlegendentry{Approx.}

\addplot [name path=upper,draw=none] table[x=alpha,y=max_err]
    {data/ARGaussian0.05.dat};
\addplot [name path=lower,draw=none] table[x=alpha,y=min_err]
    {data/ARGaussian0.05.dat};
\addplot [fill=black!10] fill between[of=upper and lower];

\addplot [name path=upper_apx,draw=none] table[x=alpha,y=apx_max_err]
    {data/ARGaussian0.05.dat};
\addplot [name path=lower_apx,draw=none] table[x=alpha,y=apx_min_err]
    {data/ARGaussian0.05.dat};
\addplot[pattern=crosshatch, 
        pattern color=black!30]  
        fill between[of=upper_apx and lower_apx];

\end{axis}
\end{tikzpicture}
\end{subfigure}

    \begin{subfigure}[h]{\textwidth}
\centering
\begin{tikzpicture}
\begin{axis}[
    width=0.7\textwidth,
    height=0.3\textwidth,
    title={AR(1), $\phi = 0.1$},
    ylabel={$\norm{\widehat{\bbe} - \bbe_0}_2$},
    ymode=log,
    xmin=0.7, xmax=1,
    ymin=10^-0.5,
    xtick={0.7, 0.8, 0.9, 1.0},
    legend pos=south west,
    ymajorgrids=true,
    grid style=dashed,
]

\addplot[
    color=black,
    mark=*,
    mark size=1.5pt   
    ]
    table[x=alpha,y=err] {data/ARGaussian0.10.dat};
    \legend{Exact}

\addplot[
    color=black!50,
    mark=*,
    mark size=1.5pt,
    ]
    table[x=alpha,y=apx_err] {data/ARGaussian0.10.dat};
    \addlegendentry{Approx.}

\addplot [name path=upper,draw=none] table[x=alpha,y=max_err]
    {data/ARGaussian0.10.dat};
\addplot [name path=lower,draw=none] table[x=alpha,y=min_err]
    {data/ARGaussian0.10.dat};
\addplot [fill=black!10] fill between[of=upper and lower];

\addplot [name path=upper_apx,draw=none] table[x=alpha,y=apx_max_err]
    {data/ARGaussian0.10.dat};
\addplot [name path=lower_apx,draw=none] table[x=alpha,y=apx_min_err]
    {data/ARGaussian0.10.dat};
\addplot[pattern=crosshatch, 
        pattern color=black!30]  
        fill between[of=upper_apx and lower_apx];

\end{axis}
\end{tikzpicture}
\end{subfigure}

    \begin{subfigure}[h]{\textwidth}
\centering
\begin{tikzpicture}
\begin{axis}[
    width=0.7\textwidth,
    height=0.3\textwidth,
    title={AR(1), $\phi = 0.15$},
    xlabel={$\alpha$},
    ylabel={$\norm{\widehat{\bbe} - \bbe_0}_2$},
    ymode=log,
    xmin=0.7, xmax=1,
    ymin=10^-0.5,
    xtick={0.7, 0.8, 0.9, 1.0},
    legend pos=south west,
    ymajorgrids=true,
    grid style=dashed,
]

\addplot[
    color=black,
    mark=*,
    mark size=1.5pt   
    ]
    table[x=alpha,y=err] {data/ARGaussian0.15.dat};
    \legend{Exact}

\addplot[
    color=black!50,
    mark=*,
    mark size=1.5pt,
    ]
    table[x=alpha,y=apx_err] {data/ARGaussian0.15.dat};
    \addlegendentry{Approx}

\addplot [name path=upper,draw=none] table[x=alpha,y=max_err]
    {data/ARGaussian0.15.dat};
\addplot [name path=lower,draw=none] table[x=alpha,y=min_err]
    {data/ARGaussian0.15.dat};
\addplot [fill=black!10] fill between[of=upper and lower];

\addplot [name path=upper_apx,draw=none] table[x=alpha,y=apx_max_err]
    {data/ARGaussian0.15.dat};
\addplot [name path=lower_apx,draw=none] table[x=alpha,y=apx_min_err]
    {data/ARGaussian0.15.dat};
\addplot[pattern=crosshatch, 
        pattern color=black!30]  
        fill between[of=upper_apx and lower_apx];

\end{axis}
\end{tikzpicture}
\end{subfigure}
\caption{Estimation error as a function of the density of observed entries $\alpha$.  Each plot
    shows the performance of imputation with exact conditional expectation in
    black compared to the approximated conditional expectation (where the
    parameter $\phi$ is estimated from data) in gray.  This example used
    $n=1000,\ p=1200,\ s=\lceil \sqrt{p} \rceil = 35$ (square root sparsity).  Each
    data point is an average of $10$ trials. The error of the known covariance
    example is shaded in solid between the best and worst error values over $10$
    trials whereas the error of the approximated covariance is crosshatched
    between the best and worst error values.  Each plot corresponds to
 a different value of the parameter $\phi$.}
\label{fig:ARGaussian}
\end{figure}

\subsection{Banded Inverse Covariance}
Here, we simulate two cases: when the exact conditional expectation can be
computed (the covariance is known), and when we need to estimate the covariance
from the data (but the sparsity pattern of the inverse covariance matrix is
known).  Using $\bSig = \bOmega^{-1}$, we take:
\[\Omega_{ij} = \begin{cases}
    \phi^{\lvert i - j \rvert}, & \text{if}\ \lvert i - j \rvert \leq 3 \\
      0, & \text{otherwise},
    \end{cases} \]
Again, we isolate the effective noise caused by missing data by considering the
noiseless setting.  We generate $\bbe_0$ in the same way
as above.  For these simulations, we take $\phi = 0.25$ and we set the
regularization parameter 
\[\lambda = \lambda_{\max}(\Sigma_{\bX})\sqrt{\frac{(1 - \alpha)\log{p}}{n}}\]
We plot the
empirical error as a function of $\alpha$, using the known covariance and the empirical error using the
approximated covariance on the same plot.  The results of the simulation are
shown in Figure~\ref{fig:BandedCovarianceGaussian}, where we simulate $n=1000$
and vary $p = \{600,900,1200\}$.  The two curves show that
even without sample splitting, the quality of estimation is similar between the
setting where $\Sigma_{\bX}$ is known exactly and where it must be approximated.
\begin{figure}[t!]
\centering
    \begin{subfigure}[h]{\textwidth}
        \centering
        \begin{tikzpicture}
\begin{axis}[
    width=0.7\textwidth,
    height=0.3\textwidth,
    title={Exact vs. Approximate Conditional Expectation: Banded, $p=600$},
    ylabel={$\norm{\widehat{\bbe} - \bbe_0}_2$},
    ymode=log,
    xmin=0.5, xmax=1,
    ymin=10^-1.05,
    xtick={0.5, 0.6, 0.7, 0.8, 0.9, 1.0},
    legend pos=south west,
    ymajorgrids=true,
    grid style=dashed,
]

\addplot[
    color=black,
    mark=*,
    mark size=1.5pt   
    ]
    table[x=alpha,y=err] {data/BandedGaussian600.dat};
    \legend{Exact}

\addplot[
    color=black!50,
    mark=*,
    mark size=1.5pt,
    ]
    table[x=alpha,y=apx_err] {data/BandedGaussian600.dat};
    \addlegendentry{Approx.}

\addplot [name path=upper,draw=none] table[x=alpha,y=max_err]
    {data/BandedGaussian600.dat};
\addplot [name path=lower,draw=none] table[x=alpha,y=min_err]
    {data/BandedGaussian600.dat};
\addplot [fill=black!10] fill between[of=upper and lower];

\addplot [name path=upper_apx,draw=none] table[x=alpha,y=apx_max_err]
    {data/BandedGaussian600.dat};
\addplot [name path=lower_apx,draw=none] table[x=alpha,y=apx_min_err]
    {data/BandedGaussian600.dat};
\addplot[pattern=crosshatch, 
        pattern color=black!30]  
        fill between[of=upper_apx and lower_apx];

\end{axis}
\end{tikzpicture}
    \end{subfigure}%
    ~ 

\begin{subfigure}[h]{\textwidth}
       \centering
        \begin{tikzpicture}
\begin{axis}[
    width=0.7\textwidth,
    height=0.3\textwidth,
    title={Banded Covariance, $p=900$},
    ylabel={$\norm{\widehat{\bbe} - \bbe_0}_2$},
    ymode=log,
    xmin=0.5, xmax=1,
    ymin=10^-1.05,
    xtick={0.5, 0.6, 0.7, 0.8, 0.9, 1.0},
    legend pos=south west,
    ymajorgrids=true,
    grid style=dashed,
]

\addplot[
    color=black,
    mark=*,
    mark size=1.5pt   
    ]
    table[x=alpha,y=err] {data/BandedGaussian900.dat};
    \legend{Exact}

\addplot[
    color=black!50,
    mark=*,
    mark size=1.5pt,
    ]
    table[x=alpha,y=apx_err] {data/BandedGaussian900.dat};
    \addlegendentry{Approx.}

\addplot [name path=upper,draw=none] table[x=alpha,y=max_err]
    {data/BandedGaussian900.dat};
\addplot [name path=lower,draw=none] table[x=alpha,y=min_err]
    {data/BandedGaussian900.dat};
\addplot [fill=black!10] fill between[of=upper and lower];

\addplot [name path=upper_apx,draw=none] table[x=alpha,y=apx_max_err]
    {data/BandedGaussian900.dat};
\addplot [name path=lower_apx,draw=none] table[x=alpha,y=apx_min_err]
    {data/BandedGaussian900.dat};
\addplot[pattern=crosshatch, 
        pattern color=black!30]  
        fill between[of=upper_apx and lower_apx];

\end{axis}
\end{tikzpicture}
    \end{subfigure}
    \begin{subfigure}[h]{\textwidth}
       \centering
        \begin{tikzpicture}
\begin{axis}[
    width=0.7\textwidth,
    height=0.3\textwidth,
    title={Banded Covariance, $p=1200$},
    xlabel={$\alpha$},
    ylabel={$\norm{\widehat{\bbe} - \bbe_0}_2$},
    ymode=log,
    xmin=0.5, xmax=1,
    ymin=10^-1.05,
    xtick={0.5, 0.6, 0.7, 0.8, 0.9, 1.0},
    legend pos=south west,
    ymajorgrids=true,
    grid style=dashed,
]

\addplot[
    color=black,
    mark=*,
    mark size=1.5pt   
    ]
    table[x=alpha,y=err] {data/BandedGaussian1200.dat};
    \legend{Exact}

\addplot[
    color=black!50,
    mark=*,
    mark size=1.5pt,
    ]
    table[x=alpha,y=apx_err] {data/BandedGaussian1200.dat};
    \addlegendentry{Approx.}

\addplot [name path=upper,draw=none] table[x=alpha,y=max_err]
    {data/BandedGaussian1200.dat};
\addplot [name path=lower,draw=none] table[x=alpha,y=min_err]
    {data/BandedGaussian1200.dat};
\addplot [fill=black!10] fill between[of=upper and lower];

\addplot [name path=upper_apx,draw=none] table[x=alpha,y=apx_max_err]
    {data/BandedGaussian1200.dat};
\addplot [name path=lower_apx,draw=none] table[x=alpha,y=apx_min_err]
    {data/BandedGaussian1200.dat};
\addplot[pattern=crosshatch, 
        pattern color=black!30]  
        fill between[of=upper_apx and lower_apx];

\end{axis}
\end{tikzpicture}
\end{subfigure}
    \caption{Estimation error as a function of the density of observed
        entries $\alpha$.  Each plot shows the performance with exact conditional
        expectation imputation (in black) versus the performance of approximate
        conditional expectation imputation (in gray) for varying values of the
        dimension $p$.  This example used
    $n=1000,\ s=\lceil \sqrt{p} \rceil$ (square root sparsity) for $p=600$ in
    the top figure, $p=900$ in the middle figure and $p=1200$ in the bottom.  Each
    data point is an average of $10$ trials. The error of the known covariance
    example is shaded in solid between the best and worst error values over $10$
    trials whereas the error of the approximated covariance is crosshatched
    between the best and worst error values.}
\label{fig:BandedCovarianceGaussian}
\end{figure}
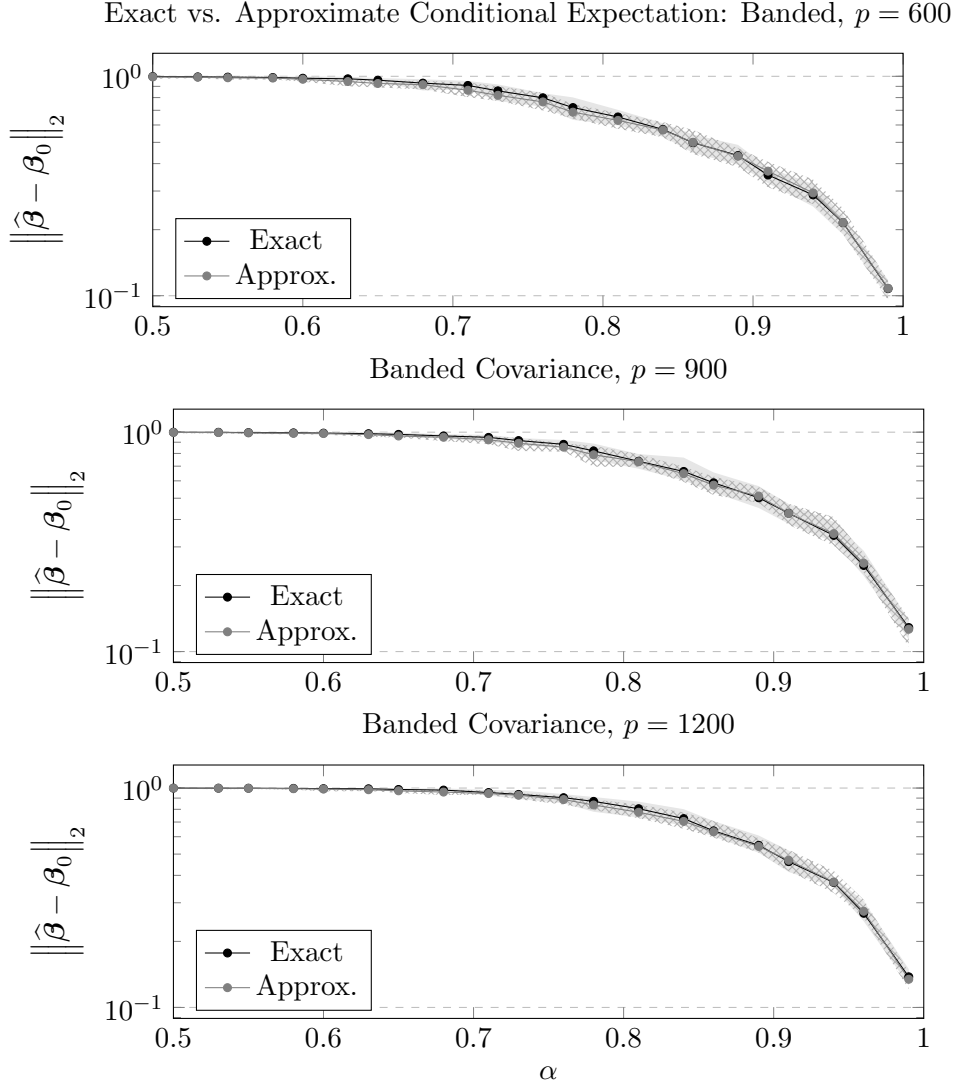

\subsection{Semi-synthetic data: gene expression}
We simulate our imputation procedure on the \emph{gene expression cancer
RNA-Seq} data set from the UCI
repository~\cite{Dua2019}.  The original data set contains $801$ samples and has
dimension $20532$.  Our experiments randomly subsample the columns of the data
and discard columns with small weight: we are left with a random sample of size
$936$ in our simulations.  We then center and normalize the remaining data
matrix and set this to be $\bX$.  We generate a regression vector $\bbe_0$ which
is one on the first $\sqrt{p}$ coordinates and zero everywhere else.  That is,
we let
\[
    (\bbe_0)_i = \begin{cases}
        1 & \text{if}\ i \leq \sqrt{p}, \\
      0 & \text{otherwise},
  \end{cases}
\]
and generate
responses $\by$ according to $\by = \bX \bbe_0$.  Additionally, on the data
matrix $\bX$, we use the graphical LASSO, using the {\fontfamily{qcr}\selectfont
skggm} package~\cite{laska_narayan_2017_830033}, to find the precision matrix
$\bOmega$ and its associated graphical model. Then, for each $\alpha$, we
generate $\bZ$ by deleting each entry of $\bX$ independently with probability $1
- \alpha$.  Accordingly, for each $\alpha$, we find $\widehat{\bX}$ using the
  imputation described in Section~\ref{sec:SparsePrecision} and use the LASSO
  with regularization parameter $\lambda = 0.1 \sqrt{(1 - \alpha)
  \frac{\log{p}}{n}}$.  The results of the simulation are found in
  Figure~\ref{fig:GeneData}.  Notice that even using the approximated graphical
  model and the imputed matrix, the LASSO is able to recover $\bbe_0$ reasonably
  well.  For context, we note that in this simulation, using the ``oracle" LASSO
  with the data matrix $\bX$ achieves error $\norm{\widehat{\bbe} - \bbe_0}_2 =
  0.124$.  
\begin{figure}[t!]
\centering
\begin{tikzpicture}
\begin{axis}[
    width=0.7\textwidth,
    height=0.4\textwidth,
    title={Approximate Imputation Gene Expression Data},
    xlabel={Probability of Observing an Entry ($\alpha$)},
    ylabel={$\norm{\widehat{\bbe} - \bbe_0}_2$},
    ymode=log,
    xmin=0.8, xmax=1,
    xtick={0.8, 0.9, 1.0},
    legend pos=south west,
    ymajorgrids=true,
    grid style=dashed,
]

\addplot[
    color=black,
    mark=*,
    mark size=1.25pt   
    ]
    table[x=alpha,y=err] {data/gene0.75.dat};
    \legend{Empirical Error}

\addplot [name path=upper,draw=none] table[x=alpha,y=max_err]
    {data/gene0.75.dat};
\addplot [name path=lower,draw=none] table[x=alpha,y=min_err]
    {data/gene0.75.dat};
\addplot [fill=black!10] fill between[of=upper and lower];

\end{axis}
\end{tikzpicture}
\caption{Error as a function of the density of observed
        entries $\alpha$.  This example used the
    \emph{gene expression cancer RNA-Seq} data set from the UCI
repository~\cite{Dua2019}.
    $n=801,\ p=936,\ s=\lceil \sqrt{p} \rceil$ (square root sparsity).
    Each data point is an average of 10 trials.  The error is
    shaded between the best and worst error values over the $10$ trials for
    each point.}
\label{fig:GeneData}
\end{figure}
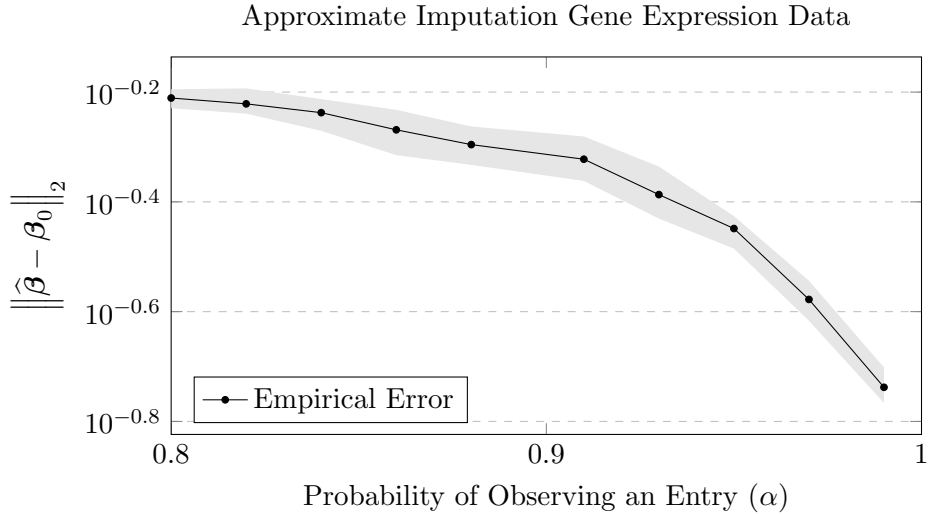

\subsection{Real data: communities and crime}
Whereas in the last section we generated responses $\by$ according to a linear model,
we will now simulate our procedure on a dataset ``Communities and crime" from
the UCI repository~\cite{Dua2019} that contains response variables.   The
original dataset contains $2215$ samples and has dimension $147$.  In order to
isolate the effect of missing data, we will first remove the samples which have
any missing entries.  We additionally remove linearly dependent columns and
are left with a data matrix $\bX$ with $n = 342, p = 123$.  We now perform two
simulations, shown in figure~\ref{fig:CrimeData}:
\begin{enumerate}
    \item We first perform an ``oracle" simulation in which there is no missing
        data.  We vary the regularization parameter $\lambda$ over the interval $(0,
        0.9)$ and perform the following $200$ times.  We randomly take $80\%$ of the
        data for training and leave $20\%$ of the data as holdout.  We plot the
        prediction error on the $20\%$ test set and plot the average value
        as well as the standard error.
    \item For various values of $\alpha$ we perfrom a simulation with missing
        data.  We again vary the regularization parameter over the same
        interval and perform the following $50$ times.  We randomly perform an
        $80/20$ split in the same manner as before.  This time on the training data, we run the graphical
        LASSO to find the sparsity pattern.  We then erase $1 - \alpha$ fraction of
        the training data and run the approximate conditional expectation
        imputed LASSO.  We plot the average prediction error on the $20$ percent test
        set as well as the standard error.
\end{enumerate}
As shown in figure~\ref{fig:CrimeData}, the
performance does not degrade much as a function of $\alpha$. 
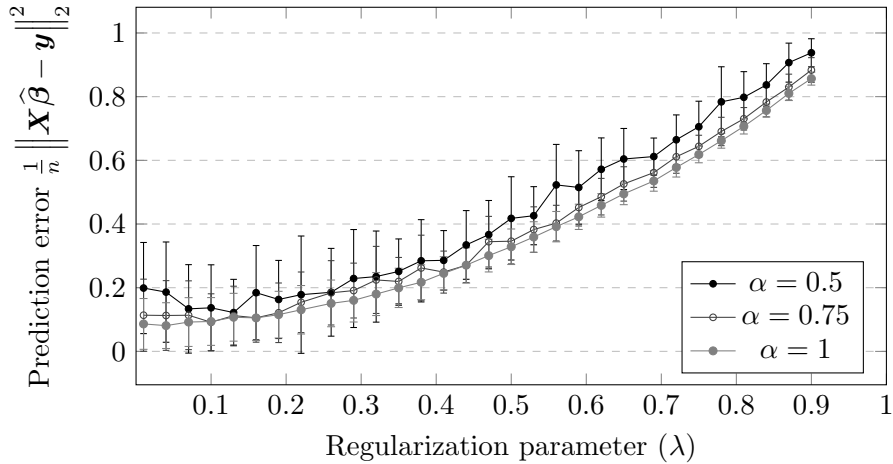
\begin{figure}[t!]
\centering
\begin{tikzpicture}
\begin{axis}[
    width=0.7\textwidth,
    height=0.4\textwidth,
    title={Approximate Imputation Communities and Crime Data},
    xlabel={Regularization parameter ($\lambda$)},
    ylabel={Prediction error $\frac{1}{n}\norm{\bX\widehat{\bbe} - \by}_2^2$},
    xmin=0, xmax=1,
    xtick={0.1, 0.2, 0.3, 0.4, 0.5, 0.6, 0.7, 0.8, 0.9, 1.0},
    legend pos=south east,
    ymajorgrids=true,
    grid style=dashed,
]

\addplot +[
color=black,
    mark=*,
    mark size=1.25pt,
    mark options={black},
error bars/.cd,
error bar style={black},
                    y explicit,
                    y dir=both,
    ]
    table[x=lambda,y=err, y error = std_error,]
                {data/crime_alpha0.5_stderror.dat};
    \legend{$\alpha = 0.5$}

\addplot +[
color=black!75,
    mark=o,
    mark size=1.25pt,
    mark options={black!75},
error bars/.cd,
error bar style={black!75},
                    y explicit,
                    y dir=both,
    ]
    table[x=lambda,y=err, y error = std_error,]
                {data/crime_alpha0.75_stderror.dat};
    \addlegendentry{$\alpha = 0.75$}

\addplot+ [
    color=black!50,
    mark=*,
    mark size=1.5pt,
    mark options={black!50},
    error bars/.cd,
                    y explicit,
                    y dir=both,
    ]
    table[x=lambda,y=err, y error = std_error,]
                {data/crime_no_missing_stderror.dat};
    \addlegendentry{$\alpha = 1$}



\end{axis}
\end{tikzpicture}
\caption{Prediction error for induced missing data as well as the oracle model
    with full data for the ``communities and crime" dataset.  For each
    $\lambda$, the gray closed circles represent the average prediction error
    for the ``oracle" model across $200$ random splits of the data.  The open
    gray circles represent the average prediction error for approximate
    conditional expectation imputation across $50$ random splits of the data when $0.75$-fraction of the data is kept,
    and the black circles when $0.5$-fraction of the data is kept.  The error
    bars give the standard error.}
\label{fig:CrimeData}
\end{figure}

\section{Conclusion}
\label{sec:Conclusion}
We have studied high-dimensional linear regression in the presence of missing data.  In contrast to previous theoretical work in this setting, we focus on the imputation strategy, followed by simple off-the-shelf estimation procedures.  Imputation by conditional expectation is shown to retain the following properties:
\begin{enumerate}
\item \textbf{Rate-optimality.} Obtained broadly with respect to the dimension,
    and with respect to \textit{every} parameter when the covariance of the data
    is the identity matrix.
\item \textbf{Pivotal.} The square-root LASSO retains rate-optimal statistical
    guarantees and is pivotal with respect to the radius of the problem $R$ and
    the noise variance $\sigma$.
\item \textbf{Robust to unknown covariance.} An approximated covariance suffices
    for the purposes of the imputed LASSO when the data comes from a sparse
    gaussian graphical model.
\end{enumerate}

Several potential future directions remain.  For instance, it is unclear what a theoretically principled way to
handle missing data is in, say, the generalized linear model.
Additionally, we have given coarse, non-asymptotic bounds 
for the linear model in the high-dimensional regime. It would be interesting to characterize the exact asymptotic performance of our procedures in the proportional regime.





\bibliographystyle{amsalpha}
\addcontentsline{toc}{section}{References}

\providecommand{\bysame}{\leavevmode\hbox to3em{\hrulefill}\thinspace}
\providecommand{\MR}{\relax\ifhmode\unskip\space\fi MR }
\providecommand{\MRhref}[2]{%
  \href{http://www.ams.org/mathscinet-getitem?mr=#1}{#2}
}
\providecommand{\href}[2]{#2}
\bibliography{ref}



\appendix

\section{Proofs for imputation by conditional expectation}
\label{sec:ProofsConditionalExpectation}
We now prove the main results of section~\ref{sec:ConditionalExpectation}.  We begin by overviewing the proof technique.  We note will explicitly write constants in this section; in subsequent, more complicated sections, we will drop this convention for readability.  The remainder of this section is organized as follows:

\begin{enumerate}
\item \textbf{Proof of theorem~\ref{thm:mainres}.} This is provided in subsection~\ref{subsec:ProofCondExpThm}.  
\item \textbf{Proof of corollary~\ref{cor:IdentityCovariance}.}  This is provided in subsection~\ref{subsec:ProofCorIdentity}.
\item \textbf{Proof of theorem~\ref{thm:mnar}.} This is provided in subsection~\ref{subsec:ProofMNAR}
\end{enumerate}

Our proofs extend the technique first used in~\cite{bickel2009simultaneous}.  We recall briefly the set-up.  We assume the linear model: $\by = \bX \bbe_0 + \beps$.  We observe $\bZ$, an in-exact version of $\bX$ and the response vector $\by$.  The set-up of~\cite{bickel2009simultaneous} assumes exact knowledge of $\bX$ and the response vector $\by$ and analyzes the properties of the LASSO estimator:
\[
\widehat{\bbe}_{\text{Las}} \in \argmin_{\bbe \in \mathbb{R}^p}\left\{\frac{1}{2n}\norm{\by - \bX \bbe}_2^2 + \lambda \norm{\bbe}_1\right\}
\]
We will consider a matrix $\widetilde{\bX}$ which is ``close" to $\bX$ and analyze the properties of:
\[
\widehat{\bbe} \in \argmin_{\bbe \in \mathbb{R}^p}\left\{\frac{1}{2n}\norm{\by - \widetilde{\bX} \bbe}_2^2 + \lambda \norm{\bbe}_1\right\}
\]

In particular, we have the following proposition, whose proof we provide in Appendix ~\ref{sec:GeneralLasso}: 
\begin{proposition} 
\label{prop:GeneralLasso}    
Given $\by = \bX \bbe_0 + \beps$ and $\widetilde{\bX} \in \mathbb{R}^{n \times p}$ and $\lambda>0$, consider the solution to the convex program
\begin{align*}
\widetilde{\bbe} \in \argmin_{\bbe \in \mathbb{R}^p} \left\{\frac{1}{2n}\norm{\by -
\widetilde{\bX}\bbe}_2^2 + \lambda \norm{\bbe}_1\right\}.
\end{align*}
Further, assume
\begin{align*}
&\norm{\frac{1}{n}\widetilde{\bX}^T\left(\widetilde{\bX} -\bX\right)\bbe_0}_{\infty} \leq \frac{\lambda}{4}, 
\hspace{.5cm}
\norm{\frac{1}{n}\widetilde{\bX}^T\beps}_{\infty} \leq \frac{\lambda}{4},\\
&\text{and}~~~~
\inf_{\bw \in \mathcal{C} \cap \mathcal{S}^{p-1}}\frac{1}{n}\norm{\widetilde{\bX}\bw}_2^2 \geq \kappa,
\end{align*} 
where $\mathcal{C}$ is the cone $\mathcal{C} = \left\{\bw \in \mathbb{R}^p: \norm{\bw_{T^c}}_1 \leq
    3\norm{\bw_T}_1\right\}$.
Then $\norm{\widetilde{\bbe} - \bbe_0}_2 \leq \frac{12\lambda\sqrt{s}}{\kappa}$.
\end{proposition}
\begin{remark}
Observe that taking $\widetilde{\bX} = \bX$ recovers the setting of~\cite{bickel2009simultaneous}.  Of course, $\bX$ is not available in the context of this paper.
\end{remark}

We are now prepared to proceed with the proofs of our main results of section~\ref{sec:ConditionalExpectation}.

\subsection{Proof of theorem ~\ref{thm:mainres}}
\label{subsec:ProofCondExpThm}
We recall the statement for the reader's convenience:
\begin{customthm}{1}
Assume \textbf{A1--A4} and that the data is MCAR($\alpha$) for $\alpha < 1$.  Assume
additionally that $\sqrt{\frac{\log{p}}{n}} \leq c_1\sqrt{1 -
    \alpha}$ for a positive constant $c_1$.
Then, there exist positive constants $c_2, c_3, c_4$ such that with probability
at least $1 - c_2p^{-c_3}$, $\widehat{\bbe}$
as defined in~\eqref{eq:ImputeLasso} with regularization parameter
    \[\lambda = c_4\left(\sigma_X \sigma + \sigma_X^2\sqrt{1 - \alpha}R\right)\sqrt{\frac{\log{p}}{n}},\]
satisfies
    \begin{align}\norm{\widehat{\bbe} - \bbe_0}_2
        \lesssim\frac{\sigma_X \sigma + \sigma_X^2 \sqrt{1 - \alpha}R}{\lambda_{\min}(\Sigma_{\widehat{\bX}})}\sqrt{\frac{s\log{p}}{n}}\, .
    \end{align}
\end{customthm}

\begin{proof}
We will take a constant $A > \sqrt{2}$ and define
\begin{align}
\label{eq:mainlambda}
\lambda = 4A\left(8e\sigma_X \sigma + 16\sqrt{2}e\sigma_X^2\sqrt{1 - \alpha}R\right)\sqrt{\frac{\log{p}}{n}}.
\end{align}
Using this choice, we establish the following lemmas, each of which controls one of the two $\ell_{\infty}$ terms required by proposition~\ref{prop:GeneralLasso}.  The proofs of these lemmas are provided in appendix~\ref{sec:appendixConditionalExpectation}.
\begin{lemma}
\label{lem:term1thm1}
Assuming $\lambda$ as defined in equation~\eqref{eq:mainlambda}, and the assumptions of Theorem~\ref{thm:mainres}, we have
\[
\pr\left\{\norm{\frac{1}{n}\widehat{\bX}^T\left(\widehat{\bX} - \bX\right)\bbe_0}_{\infty} \geq \frac{\lambda}{4}\right\} \leq 2p^{1 - \frac{A^2}{2}}.
\]
\end{lemma}

\begin{lemma}
\label{lem:term2thm1}
Assuming $\lambda$ as defined in equation~\eqref{eq:mainlambda}, and the assumptions of Theorem~\ref{thm:mainres}, we have
\[
\pr\left\{\norm{\frac{1}{n}\widehat{\bX}^T \beps}_{\infty} \geq \frac{\lambda}{4}\right\} \leq 2p^{1 - \frac{A^2}{2}}.
\]
\end{lemma}

We now need to control the so-called restricted eigenvalue, that is $\inf_{\bw \in \mathcal{C} \cap \mathcal{S}^{p-1}}\frac{1}{n}\norm{\widehat{\bX}\bw}_2^2$.  To this end, we have the following fact that we will use repeatedly:
\begin{fact}
\label{fact:subGaussianConditionalExpectation}
Assume that the rows of a matrix $\bX$ are $\sigma_X^2$ sub-Gaussian.  Then, for any set $S \subseteq [p]$, the random vector $\widehat{\bX}_i = \E\left\{\bX_i \mid \bX_{S}\right\}$ is sub-Gaussian with
    parameter $\sigma_X^2$.
\end{fact}
\noindent\begin{proofof}{Fact~\ref{fact:subGaussianConditionalExpectation}}
   \begin{align*}
\label{eq:condExpSubGaussian}
\E\left\{e^{\lambda\inp{\bu}{\widehat{\bX}_i}}\right\} \leq
\E\left\{\E\left\{e^{\lambda\inp{\bu}{\bX_i}}\mid \bX_S\right\}\right\} =
\E\left\{e^{\lambda\inp{\bu}{\bX_i}}\right\} \leq e^{\lambda^2\sigma_X^2}.
\end{align*} 
\end{proofof}

We have the following proposition, whose proof is provided in appendix~\ref{sec:GeneralLasso}.
\begin{proposition} 
    \label{prop:RE_Condition}
Let $\bX \in \mathbb{R}^{n \times p}$ be a matrix with i.i.d. rows, each of which is sub-Gaussian with
    parameter $\sigma_X^2$ and has covariance matrix $\Sigma_{\bX}$.  Additionally, let $S \subseteq [p]$ with $\lvert S \rvert = s$.  
    Then  $\bX$ satisfies the
                    restricted eigenvalue condition
    \begin{align*}
        \inf_{\bw \in \mathcal{C} \cap S^{p-1}}\frac{1}{n}\norm{\bX \bw}_2^2
        \geq \frac{\lambda_{\min}\left(\Sigma_{\bX}\right)}{2},
    \end{align*}
    with probability at least 
    \begin{align*}
        1 -
                2\exp\left\{-c n \left(\min\left(\frac{ \lambda_{\min}\left(\Sigma_{\bX}\right)^2}{\sigma_X^4},
                \frac{\lambda_{\min}\left(\Sigma_{\bX}\right)}{\sigma_X^2}\right) - \frac{s \log{p}}{n}\right)\right\},
    \end{align*}
    where $\mathcal{C}$ is the cone $\mathcal{C} = \left\{\bw \in \mathbb{R}^p: \norm{\bw_{S^c}}_1 \leq
    3\norm{\bw_S}_1\right\}$.
\end{proposition}

Let us now conclude the proof of theorem~\ref{thm:mainres}.  Let $\mathcal{A}_\text{RE}$ denote the event that $\inf_{\bw \in \mathcal{C} \cap S^{p-1}}\frac{1}{n}\norm{\widehat{\bX} \bw}_2^2
        \geq \frac{\lambda_{\min}\left(\Sigma_{\bX}\right)}{2}$ where we take the set $S$ in the assumptions of proposition~\ref{prop:RE_Condition} to be $T$.  Then the assumptions of the theorem, fact~\ref{fact:subGaussianConditionalExpectation} and proposition~\ref{prop:RE_Condition} imply that $\pr\{\mathcal{A}_{\text{RE}}\} \geq 1 - 2p^{-c_3}$, where $c_3$ is a constant.  Additionally, let $\mathcal{A}_{\infty}$ denote the event that the results of lemmas~\ref{lem:term1thm1},~\ref{lem:term2thm1} hold.  Then, on the event $\mathcal{A}_{\infty} \cap \mathcal{A}_{\text{RE}}$, which holds with probability at least $1 - 2p^{c_1}$ with $c_1 = (2 - A^2) \vee c_3$, proposition~\ref{prop:GeneralLasso} with $\lambda$ as in \eqref{eq:mainlambda}, $\kappa = \frac{\lambda_{\min}(\Sigma_{\widehat{\bX}})}{2}$ implies the result.
\end{proof}

\subsection{Proof of corollary ~\ref{cor:IdentityCovariance}}
\label{subsec:ProofCorIdentity}
We recall the statement for the reader's convenience:
\begin{customcorollary}{3.1}
Assume the data is MCAR($\alpha$) for $\alpha < 1$ and $X_{ij} \sim_{iid}
\normal(0,1)$. 
Assume additionally that $\sqrt{\frac{\log{p}}{n}} \leq 
    c_1\sqrt{\alpha(1-\alpha)}$ for a positive constant $c_1$.  Then, there exist positive constants $c_2,
    c_3, c_4$ such that with probability at least $1 - c_2p^{-c_3}$, $\widehat{\bbe}$ as defined in~\eqref{eq:ImputeLasso} with regularization parameter
    \[\lambda =  c_4 \sqrt{\alpha}\left(\sigma + \sqrt{1
    - \alpha}R\right)\sqrt{\frac{\log{p}}{n}},\]
satisfies
    \begin{align}\norm{\widehat{\bbe} - \bbe_0}_2
       \lesssim \frac{\sigma +
        \sqrt{1-\alpha}R}{\sqrt{\alpha}}\sqrt{\frac{s\log{p}}{n}}\, .
        \end{align}
\end{customcorollary}

Before we embark on the proof we note some key  differences between corollary~\ref{cor:IdentityCovariance} and theorem~\ref{thm:mainres}.  First, since the data is i.i.d. $\normal(0,1)$, the imputation matrix $\widehat{\bX}$ is drastically simplified:
\begin{align*}
    \widehat{X}_{ia} = \begin{cases} 
        X_{ia} & \text{if} \ Z_{ia} \ \text{observed}, \\
            0 & \text{ otherwise }. 
            \end{cases}
\end{align*}
Immediately, we are able to see that in this model, $\sigma_X = 1$, $\Sigma_{\widehat{\bX}} = \alpha \bI_{p \times p}$.  Finally, the independence between each entry in the data matrix allows us to get better control in the concentration inequalities we use to prove the $\ell_\infty$ bounds needed for proposition~\ref{prop:GeneralLasso}.  We now proceed to the proof.
\begin{proof}
We set
\begin{align}
\label{eq:lambdacor}
\lambda = 4A\sqrt{\alpha}\left(8\sqrt{2}e\sigma + 32e\sqrt{1 - \alpha}R\right)\sqrt{\frac{\log{p}}{n}}.
\end{align}
with $A > \sqrt{2}$.
The proof will use two supplementary lemmas whose proofs are provided in appendix~\ref{sec:appendixConditionalExpectation}:
\begin{lemma}
\label{lem:term1cor1}
Assuming $\lambda$ as defined in equation~\eqref{eq:lambdacor}, and the assumptions of corollary~\ref{cor:IdentityCovariance}, we have
\[
\pr\left\{\norm{\frac{1}{n}\widehat{\bX}^T\left(\widehat{\bX} - \bX\right)\bbe_0}_{\infty} \geq \frac{\lambda}{4}\right\} \leq 2p^{1 - \frac{A^2}{2}}.
\]
\end{lemma}

\begin{lemma}
\label{lem:term2cor1}
Assuming $\lambda$ as defined in equation~\eqref{eq:lambdacor}, and the assumptions of corollary~\ref{cor:IdentityCovariance}, we have
\[
\pr\left\{\norm{\frac{1}{n}\widehat{\bX}^T \beps}_{\infty} \geq \frac{\lambda}{4}\right\} \geq 2p^{1 - \frac{A^2}{2}}.
\]
\end{lemma}
Exactly as in the proof of theorem~\ref{thm:mainres}, we let $\mathcal{A}_{\text{RE}}$ denote the event that $\inf_{\bw \in \mathcal{C}\cap \mathcal{S}^{p-1}}\frac{1}{n}\norm{\widehat{\bX}\bw}_2^2 \geq \frac{\lambda_{\min}(\Sigma_{\widehat{\bX}})}{2} = \frac{\alpha}{2}$.  Then the assumptions of the corollary, fact~\ref{fact:subGaussianConditionalExpectation} and proposition~\ref{prop:RE_Condition} imply $\pr\left\{\mathcal{A}_{\text{RE}}\right\} \geq 1 - 2p^{-c_3}$.  Letting $\mathcal{A}_{\infty}$ denote the event that the results of lemmas~\ref{lem:term1cor1},~\ref{lem:term2cor1} hold, we immediately see that proposition~\ref{prop:GeneralLasso} holds with probability at least $1 - 2p^{-c_1}$, taking $\lambda$ as in~\eqref{eq:lambdacor}, $\kappa = \frac{\alpha}{2}$ and the result follows immediately.
\end{proof}

\subsection{Proof of theorem~\ref{thm:mnar}}
\label{subsec:ProofMNAR}
We re-state the theorem for the reader's convenience:
\begin{customthm}{3}
Assume that the data is MNAR and \textbf{A1--A4}.  Assume additionally that
$\sqrt{\frac{\log{p}}{n}} \leq c_1$ for a positive constant $c_1$.  Then, there
exist positive constants $c_2, c_3, c_4$ such that with probability at least $1
- c_2p^{-c_3}$, $\widehat{\bbe}$ as defined in Eq.~\eqref{eq:ImputeLasso} with $\widehat{\bX} = \E\left\{ \bX \mid \bZ \right\}$ and regularization parameter
    \[\lambda = c_4\left(\sigma_X\sigma + \sigma_X^2 R\right)\sqrt{\frac{\log{p}}{n}},\]
satisfies
    \begin{align}\norm{\widehat{\bbe} - \bbe_0}_2
        \lesssim \frac{\sigma_X \sigma +
         \sigma_X^2 R}{\lambda_{\min}(\Sigma_{\widehat{\bX}})}\sqrt{\frac{s\log{p}}{n}}.
    \end{align}
\end{customthm}
\begin{proof}
Again, noting the freedom of the universal constant $C$, we take (for $A > \sqrt{2}$):
\begin{align}
\label{eq:lambdamnar}
\lambda = 4A\left(8e\sigma_X \sigma + 480e\sigma_X^2 R\right)\sqrt{\frac{\log{p}}{n}}.
\end{align}
We then use the following two lemmas:
\begin{lemma}
\label{lem:term1mnar}
Assuming $\lambda$ as defined in equation~\eqref{eq:lambdamnar}, and the assumptions of corollary~\ref{thm:mnar}, we have:
\[
\pr\left\{\norm{\frac{1}{n}\widehat{\bX}^T\left(\widehat{\bX} - \bX\right)\bbe_0}_{\infty} \geq \frac{\lambda}{4}\right\} \leq 6p^{1 - \frac{A^2}{2}}.
\]
\end{lemma}

\begin{lemma}
\label{lem:term2mnar}
Assuming $\lambda$ as defined in equation~\eqref{eq:lambdamnar}, and the assumptions of corollary~\ref{thm:mnar}, we have:
\[
\pr\left\{\norm{\frac{1}{n}\widehat{\bX}^T \beps}_{\infty} \geq \frac{\lambda}{4}\right\} \leq 2p^{1 - \frac{A^2}{2}}.
\]
\end{lemma}
Using these two lemmas, the theorem statement follows exactly as the proof of theorem~\ref{thm:mainres}.
\end{proof}

\section{Proofs for imputation by approximate conditional expectation}
\label{sec:ProofsApproximateConditionalExpectation}
We now prove the main results of section~\ref{sec:ApproximateConditionalExpectation}.  We will first describe the proof technique; the remainder of the section contains:
\begin{enumerate}
\item \textbf{Proof of theorem~\ref{thm:ARres}.}  This is provided in subsection~\ref{subsec:ARproof}.
\item \textbf{Proof of theorem~\ref{thm:GaussianRes}.}  This is provided in subsection~\ref{subsec:proofGaussianRes}
\end{enumerate}

\vspace{1mm}
\noindent \textbf{Proof strategy.}  Recall that we observe the pair $(\by, \bZ)$ where $\by = \bX \bbe_0 + \beps$.  The proofs of appendix~\ref{sec:ProofsConditionalExpectation} used $\widehat{\bX} = \E\left\{ \bX \mid \bZ \right\}$ as a proxy for $\bX$ in the LASSO estimator given in~\eqref{eq:ImputeLasso}.  The proof then followed by validating the assumptions of proposition~\ref{prop:GeneralLasso} (repeated here for convenience):
\begin{customproposition}{A.1}
Given $\by = \bX \bbe_0 + \beps$ and $\widetilde{\bX} \in \mathbb{R}^{n \times p}$ and $\lambda>0$, consider the solution to the convex program
\begin{align*}
\widetilde{\bbe} \in \argmin_{\bbe \in \mathbb{R}^p} \left\{\frac{1}{2n}\norm{\by -
\widetilde{\bX}\bbe}_2^2 + \lambda \norm{\bbe}_1\right\}.
\end{align*}
Further, assume
\begin{align*}
&\norm{\frac{1}{n}\widetilde{\bX}^T\left(\widetilde{\bX} -\bX\right)\bbe_0}_{\infty} \leq \frac{\lambda}{4}, 
\hspace{.5cm}
\norm{\frac{1}{n}\widetilde{\bX}^T\beps}_{\infty} \leq \frac{\lambda}{4},\\
&\text{and}~~~~
\inf_{\bw \in \mathcal{C} \cap \mathcal{S}^{p-1}}\frac{1}{n}\norm{\widetilde{\bX}\bw}_2^2 \geq \kappa,
\end{align*} 
where $\mathcal{C}$ is the cone $\mathcal{C} = \left\{\bw \in \mathbb{R}^p: \norm{\bw_{S^c}}_1 \leq
    3\norm{\bw_S}_1\right\}$.
Then $\norm{\widetilde{\bbe} - \bbe_0}_2 \leq \frac{12\lambda\sqrt{s}}{\kappa}$.
\end{customproposition}
We will henceforth use $\widetilde{\bX}$ to denote the approximate conditional expectation and $\widehat{\bX}$ to denote the true conditional expectation.  There are two key differences in carrying out the steps required for proposition~\ref{prop:GeneralLasso}:
\begin{enumerate}
	\item Since $\widetilde{\bX}$ is not the true conditional expectation, it loses important properties of $\widehat{\bX}$ such as orthogonality and sub-Gaussianity
	\item $\widetilde{\bX}$ is computed using \textit{all} of the data, and thus its rows are not independent! This lies in stark contrast to $\widehat{\bX}$ which retains independence of the rows.  
\end{enumerate}
The reader may recall that in appendix~\ref{sec:ProofsConditionalExpectation}, we used heavily the independence of the rows of $\widehat{\bX}$ to invoke standard concentration inequalities.  The main technical challenge of this section is working around the fact that the rows of $\widetilde{\bX}$ are no longer independent.  We work around this in a simple way, by writing $\widetilde{\bX} = \underbrace{(\widetilde{\bX} - \widehat{\bX})}_{A.} + \underbrace{\widehat{\bX}}_{B.}$.  The strategy is then to decouple these two terms, argue that term $A.$ is ``small" in an appropriate sense and inherit the analysis of appendix~\ref{sec:ProofsConditionalExpectation} for term $B.$

\subsection{Proof of theorem~\ref{thm:ARres}}
\label{subsec:proofsAR}
\label{subsec:ARproof}
We re-state the theorem here for convenience:
\begin{customthm}{5}
Assume \textbf{A1--A4}, the data is MCAR($\alpha$), that the sample size $n
\geq c_1\frac{1}{\alpha^8}s\log{p}$ for a positive constant $c_1$, and that the rows of $\bX$ are generated from the
stationary auto-regressive process described above with $\lvert \phi \rvert <
1$.  Then, there exist positive constants $c_2, c_3, c_4$ such that with
probability at least $1 - c_2p^{-c_3}$.
$\widehat{\bbe}$ as defined in~\eqref{eq:ImputeLasso}, using $\widetilde{\bX}$
in place of $\widehat{\bX}$, with regularization parameter
\begin{align}
    \label{eq:lambdaAR}
    \lambda = c_4\left(\frac{\sigma_X\sigma}{\alpha^2} +
    \frac{\sigma_X^2}{\alpha^4} R\right)\sqrt{\frac{\log{p}}{n}},
\end{align}
satisfies:
    \begin{align}\norm{\widehat{\bbe} - \bbe_0}_2
        \lesssim \frac{\left(\frac{\sigma_X\sigma}{\alpha^2} + \frac{\sigma_X^2}{\alpha^4} R\right)}{\lambda_{\min}(\Sigma_{\widehat{\bX}})}\sqrt{\frac{s\log{p}}{n}}.
    \end{align}
\end{customthm}

\begin{proof}
We begin by showing $\norm{\frac{1}{n}\widetilde{\bX}^T\left(\widetilde{\bX} -\bX\right)\bbe_0}_{\infty} \leq \frac{\lambda}{4}$ with high probability.  The triangular inequality implies:
\begin{align*}
\norm{\frac{1}{n}\widetilde{\bX}^{T}\left(\widetilde{\bX} -
\bX\right)\bbe_0}_{\infty} &\leq  
\norm{\frac{1}{n}\widehat{\bX}^T\left(\widetilde{\bX} -
\widehat{\bX}\right)\bbe_0}_{\infty} + \norm{\frac{1}{n}\left(\widetilde{\bX} -
\widehat{\bX}\right)^T\left(\widetilde{\bX} - \widehat{\bX}\right)\bbe_0}_{\infty}\\
&+ \norm{\frac{1}{n}\left(\widetilde{\bX} -
\widehat{\bX}\right)^T\left(\widehat{\bX} - \bX\right)\bbe_0}_{\infty}
 + 
\norm{\frac{1}{n}\widehat{\bX}^T\left(\widehat{\bX} - \bX\right)\bbe_0}_{\infty}.
\end{align*}
Recall that our strategy is to prove that terms involving $\widetilde{\bX} - \widehat{\bX}$ are small.  It is thus useful to recall the construction of $\widetilde{\bX}$.  The true covariance matrix is $\left(\Sigma_{\bX}\right)_{ij} = \frac{1}{1 - \phi^2}\phi^{\lvert i - j \rvert}$.  We then take $\widehat{X}_{ik}$ as in equation~\eqref{eq:ARhat} and $\widetilde{X}_{ik}$ as in equation~\eqref{eq:ARtilde}:
\begin{align*}
    \widehat{X}_{ik} &= \frac{\phi^{d_1 + d_2}}{1 - \phi^{2(d_1 + d_2)}}
    \left(X_{i,L(k)}\left(\phi^{-d_2} - \phi^{d_2}\right) + X_{i, R(k)} \left(\phi^{-d_1} -
    \phi^{d_1}\right)\right) \\
    \widetilde{X}_{ik} &= \frac{\hat{\phi}^{d_1 + d_2}}{1 - \hat{\phi}^{2(d_1 + d_2)}}
    \left(X_{i, L(k)}\left(\hat{\phi}^{-d_2} - \hat{\phi}^{d_2}\right) +
    X_{i, R(k)} \left(\hat{\phi}^{-d_1} -
    \hat{\phi}^{d_1}\right)\right),
\end{align*}
where  $d_1 = k - L(k)$ and $d_2 = R(k) -k$ with $L(k)$ and $R(k)$ the positions of the closed observed entries to the left and right of $k$, respectively, and we have defined $\widehat{\phi}$ as in equation~\ref{eq:ARParamApprox}:
\begin{align*}
\widehat{\phi} =
\frac{\frac{1}{\alpha^2 np}\sum_{i=1}^{n}\sum_{a=1}^{p-1}X_{ia}X_{i(a+1)}M_{ia}M_{i(a+1)}}{\frac{1}{\alpha n
p}\sum_{i=1}^{n}\sum_{a=1}^{p-1}X_{ia}^2 M_{ia}}.
\end{align*}
The following lemma, whose proof can be found in ~\ref{subsec:proofARapprox}, is the key ingredient to showing $\widetilde{\bX} - \widehat{\bX}$ is ``small".
\begin{lemma}
\label{lem:ARapprox}
Define $\hat{\phi}$ as in ~\eqref{eq:ARParamApprox}.  Then under the assumptions of theorem~\ref{thm:ARres},
\[\pr\left\{\left\lvert \hat{\phi} - \phi\right\rvert \geq \frac{4}{\alpha^2}\sqrt{\frac{\log{p}}{np}}\right\} \leq c p^{-C}.\]
where $c, C > 0$ are universal constants.
\end{lemma}

\vspace{1mm}
\noindent \textbf{Step 1: Control $\norm{\frac{1}{n}\widetilde{\bX}^T\left(\widetilde{\bX} -\bX\right)\bbe_0}_{\infty}$.}
\vspace{1mm}
We require the following four lemmas, whose proofs we defer to appendix~\ref{sec:appendixARtechnicallemmas}.
\begin{lemma}
\label{lem:term1thmAR}
Under the assumptions of theorem~\ref{thm:ARres}, for any $\bu \in \mathbb{R}^{p}$, we have:
\[
\pr\left\{\norm{\frac{1}{n}\widehat{\bX}^T\left(\widetilde{\bX} - \widehat{\bX}\right)\bu}_{\infty} \geq \frac{C\sigma_X^2}{\alpha^2}\norm{\bu}_1\sqrt{\frac{\log{p}}{np}}\right\} \leq c_0p^{-c_1},
\]
where $C, c_0, c_1$ denote universal constants.
\end{lemma}
The proof of the above lemma yields the identitical corollary:
\begin{corollary}
\label{cor:term1thmAR}
Under the assumptions of theorem~\ref{thm:ARres}, for any $\bu \in \mathbb{R}^{p}$, we have:
\[
\pr\left\{\norm{\frac{1}{n}\left(\widetilde{\bX} - \widehat{\bX}\right)^T\widehat{\bX}\bu}_{\infty} \geq \frac{C\sigma_X^2}{\alpha^2}\norm{\bu}_1\sqrt{\frac{\log{p}}{np}}\right\} \leq c_0p^{-c_1}
\]
where $C, c_0, c_1$ denote universal constants.
\end{corollary}
\begin{lemma}
\label{lem:term2thmAR}
Under the assumptions of theorem~\ref{thm:ARres}, we have:
\[
\pr\left\{\norm{\frac{1}{n}\left(\widetilde{\bX} - \widehat{\bX}\right)^T\left(\widetilde{\bX} - \widehat{\bX}\right)\bu}_{\infty} \geq C\frac{\sigma_X^2}{\alpha^4}\norm{\bu}_1\frac{\log{p}}{np}\right\} \leq c_0p^{-c_1},
\]
where $C, c_0, c_1$ denote universal constants.
\end{lemma}

\begin{lemma}
\label{lem:term3thmAR}
Under the assumptions of theorem~\ref{thm:ARres}, for any $\bu \in \mathbb{R}^{p}$, we have:
\[
\pr\left\{\norm{\frac{1}{n}\left(\widetilde{\bX} - \widehat{\bX}\right)^T\left(\widehat{\bX} - \bX\right)\bu}_{\infty} \geq C\frac{\sigma_X^2}{\alpha^2}\norm{\bu}_1 \sqrt{\frac{\log{p}}{np}}\right\} \leq c_0p^{-c_1},
\]
where $C, c_0, c_1$ denote universal constants.
\end{lemma}

\begin{lemma}
\label{lem:term4thmAR}
Assuming $\lambda$ as defined in equation~\eqref{eq:lambdaAR}, and the assumptions of theorem~\ref{thm:ARres}, we have, we have:
\[
\pr\left\{\norm{\frac{1}{n}\widehat{\bX}^T\left(\widehat{\bX} - \bX\right)\bbe_0}_{\infty} \geq C\sigma_X^2R\sqrt{\frac{\log{p}}{n}}\right\} \leq  c_0p^{-c_1},
\]
where $c_0, c_1$ denote universal constants.
\end{lemma}

In lemmas~\ref{lem:term1thmAR} - ~\ref{lem:term3thmAR}, let $\bu = \bbe_0$ and notice that $\norm{\bbe_0}_1 \leq \sqrt{s}R$ by Cauchy Schwarz.  Lemma~\ref{lem:term1thmAR} and the fact that $s \leq p$ then imply:
\[
\pr\left\{\norm{\frac{1}{n}\widehat{\bX}^T\left(\widetilde{\bX} - \widehat{\bX}\right)\bbe_0}_{\infty} \geq \frac{C\sigma_X^2}{\alpha^2}R\sqrt{\frac{\log{p}}{n}}\right\} \leq c_0p^{-c_1},
\]
Similarly, lemma~\ref{lem:term3thmAR} implies
\[
\pr\left\{\norm{\frac{1}{n}\left(\widetilde{\bX} - \widehat{\bX}\right)^T\left(\widehat{\bX} - \bX\right)\bbe_0}_{\infty} \geq C\frac{\sigma_X^2}{\alpha^2}R \sqrt{\frac{\log{p}}{n}}\right\} \leq c_0p^{-c_1},
\]
lemma~\ref{lem:term2thmAR} and the assumption that $\sqrt{\frac{\log{p}}{n}} <
1$ (this happens as long as $c_{1}$ in the theorem statement is small enough) yields:
\[
\pr\left\{\norm{\frac{1}{n}\left(\widetilde{\bX} - \widehat{\bX}\right)^T\left(\widetilde{\bX} - \widehat{\bX}\right)\bbe_0}_{\infty} \geq C\frac{\sigma_X^2}{\alpha^4}R\frac{\log{p}}{n}\right\} \leq c_0p^{-c_1}.
\]
Picking the universal constant $C$ in ~\eqref{eq:lambdaAR} large enough implies that each of the preceeding events, as well as that of lemma~\ref{lem:term4thmAR} hold with probability at least $1 - c_0 p^{-c_1}$ (where we remind the reader that universal constants $c_0, c_1$ may change line to line).  This and noting that $C\left(\frac{\sigma_X\sigma}{\alpha^2} + \frac{\sigma_X^2}{\alpha^4} R\right)\sqrt{\frac{\log{p}}{n}} \geq C\frac{\sigma_X^2}{\alpha^4} R\sqrt{\frac{\log{p}}{n}}$ imply immediately that $\pr\left\{\norm{\frac{1}{n}\widetilde{\bX}^T\left(\widetilde{\bX} -\bX\right)\bbe_0}_{\infty} \leq \frac{\lambda}{4}\right\} \geq 1 - c_0p^{-c_1}$.  

\vspace{1mm}
\noindent \textbf{Step 2: Control $\norm{\frac{1}{n}\widetilde{\bX}^T\beps}_{\infty} \leq \frac{\lambda}{4}$} 
\vspace{1mm}
We now tackle showing $\norm{\frac{1}{n}\widetilde{\bX}^T\beps}_{\infty} \leq \frac{\lambda}{4}$.  To this end, the triangular inequality implies $\norm{\frac{1}{n}\widetilde{\bX}^T\beps}_{\infty} \leq \norm{\frac{1}{n}\left(\widetilde{\bX} - \widehat{\bX}\right)^T\beps}_{\infty} + \norm{\frac{1}{n}\widehat{\bX}^T\beps}_{\infty}$.  We will use the following two lemmas:

\begin{lemma}
\label{lem:term5thmAR}
Assuming $\lambda$ as defined in equation~\eqref{eq:lambdaAR}, and the assumptions of theorem~\ref{thm:ARres}, we have:
\[
\pr\left\{\norm{\frac{1}{n}\left(\widetilde{\bX} - \widehat{\bX}\right)^T\beps}_{\infty} \geq C\frac{\sigma_X\sigma}{\alpha^2}\sqrt{\frac{\log{p}}{n}}\right\} \leq c_0p^{-c_1},
\]
where $C, c_0, c_1$ are universal constants.
\end{lemma}

\begin{lemma}
\label{lem:term6thmAR}
Assuming $\lambda$ as defined in equation~\eqref{eq:lambdaAR}, and the assumptions of theorem~\ref{thm:ARres}, we have:
\[
\pr\left\{\norm{\frac{1}{n}\widehat{\bX}^T\beps}_{\infty} \leq C\sigma_X\sigma\sqrt{\frac{\log{p}}{n}}\right\} \geq c_0 p^{-c_1},
\]
where $C, c_0, c_1$ are universal constants.
\end{lemma}

Picking the universal constant $C$ in the definition of $\lambda$~\eqref{eq:lambdaAR}, noting that $C\left(\frac{\sigma_X\sigma}{\alpha^2} + \frac{\sigma_X^2}{\alpha^4} R\right)\sqrt{\frac{\log{p}}{n}} \geq C\frac{\sigma_X\sigma}{\alpha^2}\sqrt{\frac{\log{p}}{n}}$, and lemmas~\ref{lem:term5thmAR}, ~\ref{lem:term6thmAR} imply immediately $\norm{\frac{1}{n}\widetilde{\bX}^T\beps}_{\infty} \leq \frac{\lambda}{4}$ with probability at least $1 - c_0 p^{-c_1}$.  

\vspace{1mm}
\noindent \textbf{Step 3: Control $\inf_{\bw \in \mathcal{C} \cap \mathcal{S}^{p-1}} \frac{1}{n}\norm{\widetilde{\bX}\bw}_2^2$}
\vspace{1mm}

We would now like to analyze $\inf_{\bw \in \mathcal{C} \cap \mathcal{S}^{p-1}} \frac{1}{n}\norm{\widetilde{\bX}\bw}_2^2$.  Recalling the useful relation $\widetilde{\bX} = \left(\widetilde{\bX} - \widehat{\bX}\right) + \widehat{\bX}$ and $\frac{1}{n}\norm{\widetilde{\bX}\bw}_2^2 = \left \langle \bw, \frac{1}{n}\widetilde{\bX}^T\widetilde{\bX} \bw\right\rangle$, we note that:
\begin{align*}
\left \langle \bw, \frac{1}{n}\widetilde{\bX}^T\widetilde{\bX} \bw\right\rangle &= \left \langle \bw, \frac{1}{n}\left(\widetilde{\bX} - \widehat{\bX}\right)^T\left(\widetilde{\bX} - \widehat{\bX}\right) \bw\right\rangle + \left \langle \bw, \frac{1}{n}\left(\widetilde{\bX} - \widehat{\bX}\right)^T\widehat{\bX} \bw\right\rangle \\
&+ \left \langle \bw, \frac{1}{n}\widehat{\bX}^T\left(\widetilde{\bX} - \widehat{\bX}\right) \bw\right\rangle + \left \langle \bw, \frac{1}{n}\widehat{\bX}^T\widehat{\bX} \bw\right\rangle.
\end{align*}
Notice that by H{\"o}lder's inequality, for any matrix $\bA$, $\left \langle \bw, \bA \bw \right \rangle \geq -\norm{\bw}_1 \norm{\bA \bw}_{\infty}$.  This then implies that:
\begin{align*}
\left \langle \bw, \frac{1}{n}\widetilde{\bX}^T\widetilde{\bX} \bw\right\rangle &\geq -\norm{\bw}_1 \norm{\frac{1}{n}\left(\widetilde{\bX} - \widehat{\bX}\right)^T\left(\widetilde{\bX} - \widehat{\bX}\right) \bw}_{\infty} - \norm{\bw}_1\norm{\frac{1}{n}\left(\widetilde{\bX} - \widehat{\bX}\right)^T\widehat{\bX} \bw}_{\infty} \\
& - \norm{\bw}_1 \norm{\frac{1}{n}\widehat{\bX}^T\left(\widetilde{\bX} - \widehat{\bX}\right) \bw}_{\infty} + \frac{1}{n}\norm{\widehat{\bX}\bw}_2^2.
\end{align*}
We will lower bound each of these four terms, starting with the last (and recalling that $\bw \in \mathcal{C} \cap \mathcal{S}^{p-1}$).

\begin{enumerate}
\item $\frac{1}{n}\norm{\widehat{\bX}\bw}_2^2 \geq \inf_{\bw \in \mathcal{C} \cap \mathcal{S}^{p-1}} \frac{1}{n}\norm{\widehat{\bX}\bw}_2^2$.  
Let $\mathcal{A}_{RE}$ denote the event that $\inf_{\bw \in \mathcal{C} \cap \mathcal{S}^{p-1}} \frac{1}{n}\norm{\widehat{\bX}\bw}_2^2 \geq \frac{\lambda_{\min}(\Sigma_{\widehat{\bX}})}{2}$.  Then the assumptions of the theorem, fact~\ref{fact:subGaussianConditionalExpectation}, and proposition~\ref{prop:RE_Condition} imply that $\pr\left\{\mathcal{A}_{RE}\right\} \geq 1 - c_0p^{-c_1}$.
\item $-\norm{\bw}_1 \norm{\frac{1}{n}\left(\widetilde{\bX} - \widehat{\bX}\right)^T\left(\widetilde{\bX} - \widehat{\bX}\right) \bw}_{\infty}$.  Recall that since $\bw \in \mathcal{C} \cap \mathcal{S}^{p-1}$, Cauchy Schwarz implies $\norm{\bw}_1 \leq 4\sqrt{s}$.  Additionally, taking $\bu$ as $\bw$ in lemma~\ref{lem:term3thmAR} gives:
\[
\pr\left\{\norm{\frac{1}{n}\left(\widetilde{\bX} - \widehat{\bX}\right)^T\left(\widetilde{\bX} - \widehat{\bX}\right)\bw}_{\infty} \geq C\frac{\sigma_X^2}{\alpha^4}\frac{s\log{p}}{np}\right\} \leq c_0p^{-c_1}.
\]
Noting that $s \leq p$ and using the fact that $\sqrt{\frac{\log{p}}{n}} < 1$, we see that with probability at least $1 - c_0 p^{-c_1}$, 
\[
-\norm{\bw}_1 \norm{\frac{1}{n}\left(\widetilde{\bX} - \widehat{\bX}\right)^T\left(\widetilde{\bX} - \widehat{\bX}\right) \bw}_{\infty} \geq -C \frac{\sigma_X^2}{\alpha^4}\sqrt{\frac{s\log{p}}{n}}.
\]

\item $- \norm{\bw}_1\norm{\frac{1}{n}\left(\widetilde{\bX} - \widehat{\bX}\right)^T\widehat{\bX} \bw}_{\infty}$.  Taking $\bu$ as $\bw$ in cor~\ref{cor:term1thmAR}, using $\norm{\bw}_1 \leq 4\sqrt{s}$ and $s \leq p$, we see that with probability at least $1 - c_0 p^{-c_1}$,
\[
- \norm{\bw}_1\norm{\frac{1}{n}\left(\widetilde{\bX} - \widehat{\bX}\right)^T\widehat{\bX} \bw}_{\infty} \geq -C \frac{\sigma_X^2}{\alpha^2}\sqrt{\frac{s\log{p}}{n}}.
\]
\item $- \norm{\bw}_1 \norm{\frac{1}{n}\widehat{\bX}^T\left(\widetilde{\bX} - \widehat{\bX}\right) \bw}_{\infty}$.  Taking $\bu$ as $\bw$ in lemma~\ref{lem:term1thmAR}, using $\norm{\bw}_1 \leq 4\sqrt{s}$ and $s \leq p$, we see that with probability at least $1 - c_0 p^{-c_1}$, 
\[
- \norm{\bw}_1\norm{\frac{1}{n}\widehat{\bX}^T\left(\widetilde{\bX} - \widehat{\bX}\right) \bw}_{\infty} \geq -C \frac{\sigma_X^2}{\alpha^2}\sqrt{\frac{s\log{p}}{n}}.
\]
\end{enumerate}
Now, under the assumption that $n > C \frac{\sigma_X^4}{\alpha^8}\left(\frac{1}{\lambda_{\min}(\Sigma_{\widehat{\bX}})}\right)^2 s\log{p}$ for sufficiently large $C$, we have shown that with probability at least $1 - c_0 p^{-c_1}$, $\inf_{\bw \in \mathcal{C} \cap \mathcal{S}^{p-1}} \frac{1}{n}\norm{\widetilde{\bX}\bw}_2^2 \geq \frac{\lambda_{\min}(\Sigma_{\widehat{\bX}})}{4}$.  To conclude, combine the results of steps 1-3 and invoke proposition~\ref{prop:GeneralLasso} with $\lambda$ as in ~\eqref{eq:lambdaAR} and $\kappa = \frac{\lambda_{\min}(\Sigma_{\widehat{\bX}})}{4}$.
\end{proof}

\subsection{Proof of theorem~\ref{thm:GaussianRes}}
\label{subsec:proofGaussianRes}
We re-state the theorem here for convenience:
\begin{customthm}{6}
Assume \textbf{C1--C3}, \textbf{A2--A3}, the data
is MCAR($\alpha$), and that $\sqrt{\frac{\log{2np}}{n}} \leq C(\alpha, d_{\max},
\underline{c}, \overline{c})$ for positive constant $C(\alpha, d_{\max},
\underline{c}, \overline{c})$.  Then, there exist positive constants $c_1,c_2,
c_3, C(\alpha, d_{\max})$, such that with probability at least $1 - c_1n^{-1} - c_2p^{c_3}$,
$\widehat{\bbe}$ as defined in ~\eqref{eq:ImputeLasso} with regularization parameter
\[
\lambda = C(\alpha, d_{\max})\left(\sigma + R\right)\sqrt{\frac{s\log{np}}{n}},
\]
satisfies
\begin{align}
\norm{\widehat{\bbe} - \bbe_0}_2
        \leq \frac{C(\alpha, d_{\max})\left(\sigma + R\right)
        }{\lambda_{\min}(\bSig_{\hat{\bX}})}s\sqrt{\frac{\log{2np}}{n}}.
\end{align}
\end{customthm}

\begin{proof}
The strategy is then largely the same as Theorem~\ref{thm:ARres}.  We will rely on the simple equality $\widetilde{\bX} = \widehat{\bX} + \left(\widetilde{\bX} - \widehat{\bX}\right)$.  We will be able to control the difference $\widetilde{X}_{ia} - \widehat{X}_{ia}$.  Recall the Markov blanket $S_{(i,a)}$~\eqref{def:markovblanket}.  Since each row $\bX_i$ is multivariate Gaussian, we have:
\[
\widetilde{X}_{ia} - \widehat{X}_{ia} = \left(\bSig_{a, S_{(i,a)}}\bSig_{S_{(i,a)},
        S_{(i,a)}}^{-1} -  \widetilde{\bSig}_{b, S_{(i,a)}}\widetilde{\bSig}_{S_{(i,a)},
        S_{(i,a)}}^{-1}\right)\bX_{S(i,a)}.
\]

\vspace{1mm}
\noindent \textbf{Step 1: Control $\norm{\frac{1}{n}\widetilde{\bX}^T\left(\widetilde{\bX} -\bX\right)\bbe_0}_{\infty}$.}
\vspace{1mm}
Recall that by the triangular inequality:
\begin{align*}
\norm{\frac{1}{n}\widetilde{\bX}^{T}\left(\widetilde{\bX} -
\bX\right)\bbe_0}_{\infty} &\leq  
\norm{\frac{1}{n}\widehat{\bX}^T\left(\widetilde{\bX} -
\widehat{\bX}\right)\bbe_0}_{\infty} + \norm{\frac{1}{n}\left(\widetilde{\bX} -
\widehat{\bX}\right)^T\left(\widetilde{\bX} - \widehat{\bX}\right)\bbe_0}_{\infty}\\
&+ \norm{\frac{1}{n}\left(\widetilde{\bX} -
\widehat{\bX}\right)^T\left(\widehat{\bX} - \bX\right)\bbe_0}_{\infty}
 + 
\norm{\frac{1}{n}\widehat{\bX}^T\left(\widehat{\bX} - \bX\right)\bbe_0}_{\infty}.
\end{align*}
We require the following four lemmas, whose proofs we defer to appendix~\ref{sec:appendixboundedtechincallemmas}.

\begin{lemma}
\label{lem:term1thmGaussian}
Under the assumptions of theorem~\ref{thm:GaussianRes}, for any $\bu \in \mathbb{R}^{p}$, we have:
\[
\pr\left\{\norm{\frac{1}{n}\widehat{\bX}^T\left(\widetilde{\bX} - \widehat{\bX}\right)\bu}_{\infty} \geq C(\alpha, d_{\max})\norm{\bu}_1\sqrt{\frac{\log{np}}{n}}\right\} \leq n^{-1} + c_0p^{-c_1},
\]
where $c_0, c_1$ denote universal constants and $C(\alpha, d_{\max})$ a constant depending only on $\alpha, d_{\max}$.
\end{lemma}

\begin{lemma}
\label{lem:term2thmGaussian}
Under the assumptions of theorem~\ref{thm:GaussianRes}, we have:
\[
\pr\left\{\norm{\frac{1}{n}\left(\widetilde{\bX} - \widehat{\bX}\right)^T\left(\widetilde{\bX} - \widehat{\bX}\right)\bu}_{\infty} \geq C(\alpha, d_{\max})\norm{\bu}_1\frac{\log{np}}{n}\right\} \leq n^{-1} + c_0p^{-c_1},
\]
where $c_0, c_1$ denote universal constants and $C(\alpha, d_{\max})$ a constant depending only on $\alpha, d_{\max}$.
\end{lemma}

\begin{lemma}
\label{lem:term3thmGaussian}
Under the assumptions of theorem~\ref{thm:GaussianRes}, for any $\bu \in \mathbb{R}^{p}$, we have:
\[
\pr\left\{\norm{\frac{1}{n}\left(\widetilde{\bX} - \widehat{\bX}\right)^T\left(\widehat{\bX} - \bX\right)\bu}_{\infty} \geq C(\alpha, d_{\max})\norm{\bu}_1 \sqrt{\frac{\log{np}}{n}}\right\} \leq n^{-1} + c_0p^{-1},
\]
where $c_0, c_1$ denote universal constants and $C(\alpha, d_{\max})$ a constant depending only on $\alpha, d_{\max}$.
\end{lemma}

The proof of the above lemma yields the identitical corollary:
\begin{corollary}
\label{cor:term1thmGaussian}
Under the assumptions of theorem~\ref{thm:GaussianRes}, for any $\bu \in \mathbb{R}^{p}$, we have:
\[
\pr\left\{\norm{\frac{1}{n}\left(\widetilde{\bX} - \widehat{\bX}\right)^T\widehat{\bX}\bu}_{\infty} \geq C(\alpha, d_{\max})\norm{\bu}_1\sqrt{\frac{\log{np}}{n}}\right\} \leq n^{-1} + c_0p^{-1},
\]
where $c_0, c_1$ denote universal constants and $C(\alpha, d_{\max})$ a constant depending only on $\alpha, d_{\max}$.
\end{corollary}

The following lemma follows by lemma~\ref{lem:term4thmAR}.
\begin{lemma}
\label{lem:term4thmGaussian}
Under the assumptions of theorem~\ref{thm:GaussianRes}, we have, we have:
\[
\pr\left\{\norm{\frac{1}{n}\widehat{\bX}^T\left(\widehat{\bX} - \bX\right)\bbe_0}_{\infty} \geq C\sigma_X^2R\sqrt{\frac{\log{p}}{n}}\right\} \leq  c_0p^{-c_1},
\]
where $c_0, c_1$ denote universal constants.
\end{lemma}

Taking $\bu = \bbe_0$ and noting $\norm{\bbe_0}_1 \leq R\sqrt{s}$, for $C(\alpha, d_{\max})$ large enough, we have:
\[
\norm{\frac{1}{n}\widetilde{\bX}^T\left(\widetilde{\bX} -\bX\right)\bbe_0}_{\infty} \leq C(\alpha, d_{\max})R\sqrt{\frac{s\log{np}}{n}},
\]
with probability at least $1 - c_0p^{-c_1}$.

\vspace{1mm}
\noindent \textbf{Step 2: Control $\norm{\frac{1}{n}\widetilde{\bX}^T\beps}_{\infty}$.	} 
\vspace{1mm}
We now tackle showing $\norm{\frac{1}{n}\widetilde{\bX}^T\beps}_{\infty} \leq \frac{\lambda}{4}$.  To this end, the triangular inequality implies $\norm{\frac{1}{n}\widetilde{\bX}^T\beps}_{\infty} \leq \norm{\frac{1}{n}\left(\widetilde{\bX} - \widehat{\bX}\right)^T\beps}_{\infty} + \norm{\frac{1}{n}\widehat{\bX}^T\beps}_{\infty}$.  We will use the following two lemmas:

\begin{lemma}
\label{lem:term5thmGaussian}
Under the assumptions of theorem~\ref{thm:GaussianRes}, we have:
\[
\pr\left\{\norm{\frac{1}{n}\left(\widetilde{\bX} - \widehat{\bX}\right)^T\beps}_{\infty} \geq C(\alpha, d_{\max})\sigma\sqrt{\frac{\log{np}}{n}}\right\} \leq c_0n^{-1} + c_1p^{-c_2},
\]
where $c_0, c_1$ denote universal constants and $C(\alpha, d_{\max})$ a constant depending only on $\alpha, d_{\max}$.
\end{lemma}

The following lemma follows by lemma~\ref{lem:term6thmAR}.
\begin{lemma}
\label{lem:term6thmGaussian}
Under the assumptions of theorem~\ref{thm:GaussianRes}, we have:
\[
\pr\left\{\norm{\frac{1}{n}\widehat{\bX}^T\beps}_{\infty} \leq C\sigma\sqrt{\frac{\log{p}}{n}}\right\} \geq c_0 p^{-c_1},
\]
where $C, c_0, c_1$ denote universal constants.
\end{lemma}
Thus, with probability at least $1 - c_0 p^{-c_1}$, $\norm{\frac{1}{n}\widetilde{\bX}^T\beps}_{\infty} \leq C(\alpha, d_{\max})\sigma\sqrt{\frac{\log{np}}{n}}$.

The restricted eigenvalue follows by the assumption and the exact same steps as in the proof of theorem~\ref{thm:ARres}.

\end{proof}


\section{Proofs for square-root LASSO}
\label{sec:proofssqrtlasso}
This section contains the proof of theorem~\ref{thm:sqrtlasso}.  The proof largely follows the recipe of the main result of~\cite{belloni2011square}.  
To recall the set-up, we have the linear model $\by = \bX \bbe_0 + \beps$ and observe the response vector $\by$ as well as $\bZ$.  We use the square root LASSO  
\begin{align}
\label{eq:sqrtlasso}
\widehat{\bbe} \in \argmin_{\bbe \in \mathbb{R}^{p}}\frac{1}{
\sqrt{n}}\norm{\by - \widehat{\bX} \bbe}_2 + \lambda \norm{\bbe}_1,
\end{align}
to recover $\bbe_0$.  We will use the notation $f(\bbe) = \frac{1}{\sqrt{n}}\norm{\by - \widehat{\bX} \bbe}_2$ where $\nabla f(\bbe_0) = -\frac{\frac{1}{n}\widehat{\bX}^T\left(\left(\bX - \widehat{\bX}\right)\bbe_0 + \beps\right)}{\frac{1}{\sqrt{n}}\norm{\left(\bX - \widehat{\bX}\right)\bbe_0 + \beps}_2}$.  We require the following lemma:

\begin{lemma}
\label{lem:lambdasqrt}
Under the assumptions of Theorem~\ref{thm:sqrtlasso} and taking $\lambda = c \sigma_X\sqrt{\frac{\log{p}}{n}}$, then $\lambda \geq c \norm{\nabla f(\bbe_0)}_{\infty}$ with probability at least $1 - \delta - 2p^{-2}$.
\end{lemma}
\begin{proof}
We first note that $\left(\bX - \widehat{\bX}\right)\bbe_0 + \beps$ is $8R^2\sigma_X^2 + 2\sigma^2$ sub-gaussian.  Note that Veryshynin~\cite{vershynin2018high} theorem 3.1.1 implies:
\begin{align}
    \label{ineq:fbeta0concentration}
\pr\left\{\left\lvert \frac{1}{\sqrt{n}}\norm{\left(\bX - \widehat{\bX}\right)\bbe_0 + \beps}_2 - \sqrt{\gamma} \right \rvert \geq \frac{t\sqrt{\gamma}}{\sqrt{n}}\right\} \leq \exp\left\{-c\frac{t^2}{\left(8R^2\sigma_X^2 + 2\sigma^2\right)^2}\right\},
\end{align}
where we let $\gamma = \E \frac{1}{n}\norm{\left(\bX - \widehat{\bX}\right)\bbe_0 + \beps}_2^2$.  Under the assumption $\frac{\left(8R^2\sigma_X^2 + 2\sigma^2\right)\sqrt{\frac{\ln{\delta^{-1}}}{c}}}{\sqrt{n}} \leq c_0$, this implies that with probability at least $1 - \delta$:
\[
\frac{1}{\sqrt{n}}\norm{\left(\bX - \widehat{\bX}\right)\bbe_0 + \beps}_2 \geq \sqrt{\gamma}(1 - c_0) \geq \sqrt{8R^2\sigma_X^2 + 2\sigma^2}(1 - c_0).
\]
Notice now that:
\[
\norm{\frac{1}{n}\widehat{\bX}^T\left(\left(\bX - \widehat{\bX}\right)\bbe_0 + \beps\right)}_{\infty} \leq \norm{\frac{1}{n}\widehat{\bX}^T\left(\left(\bX - \widehat{\bX}\right)\bbe_0\right)}_{\infty} + \norm{\frac{1}{n}\widehat{\bX}^T\beps}_{\infty}.\]
By lemmas~\ref{lem:term1thm1} and~\ref{lem:term2thm1}, the right hand side is upper bounded by $C\sigma_X\left(\sigma + \sigma_X R\right)\sqrt{\frac{\log{p}}{n}}$ with probability at least $1 - 2p^{-2}$.   Combining these implies that with probability at least $1 - 2p^{-2} - \delta$:
\[
\frac{\norm{\frac{1}{n}\widehat{\bX}^T\left(\left(\bX - \widehat{\bX}\right)\bbe_0 + \beps\right)}_{\infty}}{\frac{1}{\sqrt{n}}\norm{\left(\bX - \widehat{\bX}\right)\bbe_0 + \beps}_2} \leq C \sigma_X \sqrt{\frac{\log{p}}{n}}.
\]
This completes the proof.
\end{proof}

\begin{lemma}
\label{lem:sqrtlassocone}
Assume $\bbe_0$ satisfies \textbf{A3} and let $\widehat{\bw} = \widehat{\bbe} - \bbe_0$, where $\widehat{\bbe}$ is defined as in~\eqref{eq:sqrtlasso}.  Then, $\norm{\widehat{\bw}_{T^c}}_1 \leq \frac{c + 1}{c - 1}\norm{\widehat{\bw}_T}_1$.
\end{lemma}
\begin{proof}
Noting that $\widehat{\bbe}$ is a minimizer yields the following simple inequality:
\begin{align}
\label{ineq:oraclesqrt}
\frac{1}{\sqrt{n}}\norm{\widehat{\bX} \widehat{\bbe} - \by}_2 - \frac{1}{\sqrt{n}}\norm{\widehat{\bX} \bbe_0 - \by}_2 \leq \lambda\norm{\bbe_0}_1 - \lambda\norm{\widehat{\bbe}}_1.
\end{align}
Let $f(\bbe) = \frac{1}{\sqrt{n}}\norm{\by - \widehat{\bX} \bbe}_2$ and notice that by convexity,
\begin{align*}
\frac{1}{\sqrt{n}}\norm{\widehat{\bX} \widehat{\bbe} - \by}_2 - \frac{1}{\sqrt{n}}\norm{\widehat{\bX} \bbe_0 - \by}_2 &\geq \left(\nabla f(\bbe_0)\right)^T\left(\widehat{\bbe} - \bbe_0\right),
\end{align*}
The preceding two inequalities then imply:
\[
\lambda\left(\norm{\bbe_0}_1 - \norm{\widehat{\bw} + \bbe_0}_1\right) \geq \nabla f(\bbe_0)^T \widehat{\bw} \geq -\norm{\nabla f(\bbe_0)}_{\infty}\norm{\widehat{\bw}}_1.
\]
Noting that by lemma ~\ref{lem:lambdasqrt}, $\lambda \geq c\norm{\nabla f(\bbe_0)}_{\infty}$, we see that $-\frac{1}{c}\norm{\widehat{\bw}}_1 \leq \norm{\bbe_0}_1 - \norm{\widehat{\bw} + \bbe_0}_1$.  Thus, by the triangular inequality we see that:
\[
-c^{-1} \norm{\widehat{\bw}_T}_1 - c^{-1}\norm{\widehat{\bw}_{T^C}}_1 \leq \norm{\widehat{\bw}_T}_1 - \norm{\widehat{\bw}_{T^C}}_1.
\]
Re-arranging yields the result.
\end{proof}

It is useful to recall the theorem statement:
\begin{customthm}{4}
Assume \textbf{A1--A4} and that the data is MCAR($\alpha$).  Assume
additionally that there exists a constant $c_1 < 1$ such that $n \geq
\frac{1}{c_1}\left(\ln{\delta^{-1}}\right)\left(R^2\sigma_X^2 +
    \sigma^2\right)^2$.  Then, there
exist positive constants $c_2, c_3, c_4$ such that with probability at least $1
- \delta - c_2p^{-c_3}$, $\widehat{\bbe}$ as defined in~\eqref{eq:squarerootlasso} with
regularization parameter
\[\lambda = c_4\sigma_X\sqrt{\frac{\log{p}}{n}},\]
satisfies
\[
    \norm{\widehat{\bbe} - \bbe_0}_2 \lesssim \frac{\sigma_X\left(\sigma + \sigma_X
    R\right)}{\lambda_{\min}(\Sigma_{\widehat{\bX}})}\sqrt{\frac{s\log{p}}{n}}.
\]
\end{customthm}

\begin{proof}
We begin by writing the difference  $\frac{1}{n}\norm{\widehat{\bX} \widehat{\bbe} - \by}_2^2 - \frac{1}{n}\norm{\widehat{\bX} \bbe_0 - \by}_2^2$ in two different ways:
\begin{enumerate}
\item $\frac{1}{n}\norm{\widehat{\bX}\widehat{\bw}}_2^2 - \frac{2}{n}\left \langle \widehat{\bw}, \widehat{\bX}^T \left[\left(\bX - \widehat{\bX}\right)\bbe_0 + \beps\right] \right \rangle$, 
\item $\left(\frac{1}{\sqrt{n}}\norm{\widehat{\bX} \widehat{\bbe} - \by}_2 - \frac{1}{\sqrt{n}}\norm{\widehat{\bX} \bbe_0 - \by}_2\right)\left(\frac{1}{\sqrt{n}}\norm{\widehat{\bX} \widehat{\bbe} - \by}_2 + \frac{1}{\sqrt{n}}\norm{\widehat{\bX} \bbe_0 - \by}_2\right)$.
\end{enumerate}
Combining these yields:
\begin{align*}
\frac{1}{n}\norm{\widehat{\bX}\widehat{\bw}}_2^2 &= \frac{2}{n}\left \langle \widehat{\bw}, \widehat{\bX}^T \left[\left(\bX - \widehat{\bX}\right)\bbe_0 + \beps\right] \right \rangle + \left(f(\widehat{\bbe}) - f(\bbe_0)\right)\left(f(\widehat{\bbe}) + f(\bbe_0)\right) \\
&\leq \frac{2}{n}\left \langle \widehat{\bw}, \widehat{\bX}^T \left[\left(\bX - \widehat{\bX}\right)\bbe_0 + \beps\right] \right \rangle + \left(f(\widehat{\bbe}) + f(\bbe_0)\right) \lambda\left(\norm{\widehat{\bw}_T}_1 - \norm{\widehat{\bw}_{T^C}}_1 \right),
\end{align*}
where the inequality follows by~\eqref{ineq:oraclesqrt}.  Recalling the explicit calculation of $\nabla f(\bbe_0)$ and using H{\"o}lder's inequality implies:
\begin{align*}
\frac{2}{n}\left \langle \widehat{\bw}, \widehat{\bX}^T \left[\left(\bX - \widehat{\bX}\right)\bbe_0 + \beps\right] \right \rangle &\leq 2\norm{\widehat{\bw}}_1 \norm{\nabla f(\bbe_0)}_{\infty} \frac{1}{\sqrt{n}}\norm{\left(\bX - \widehat{\bX}\right)\bbe_0 + \beps}_2 \\
&\leq \frac{4\lambda \sqrt{s}}{(c - 1)n}\norm{\bw}_2 \frac{1}{\sqrt{n}}\norm{\left(\bX - \widehat{\bX}\right)\bbe_0 + \beps}_2,
\end{align*}
where the second inequality follows by lemma~\ref{lem:lambdasqrt}, lemma~\ref{lem:sqrtlassocone}, and Cauchy-Schwarz.  Note additionally that~\eqref{ineq:oraclesqrt} implies $f(\widehat{\bbe}) \leq f(\bbe_0) + \frac{\lambda}{n}\left(\norm{\widehat{\bw}_T}_1 - \norm{\widehat{\bw}_{T^C}}_1 \right)$ and $\norm{\widehat{\bw}_T}_1 - \norm{\widehat{\bw}_{T^C}}_1 \leq \norm{\widehat{\bw}}_1$ so that
\[
\left(f(\widehat{\bbe}) + f(\bbe_0)\right) \lambda\left(\norm{\widehat{\bw}_T}_1 - \norm{\widehat{\bw}_{T^C}}_1 \right) \leq 2f(\bbe_0)\lambda \norm{\widehat{\bw}}_1 +\lambda^2\norm{\widehat{\bw}}_1^2. 
\]
Combining and using $\norm{\widehat{\bw}}_1 \leq \frac{2c}{c-1}\sqrt{s}\norm{\widehat{\bw}}_2$, we see that:
\[
\frac{1}{n}\norm{\widehat{\bX}\widehat{\bw}}_2^2 \leq \frac{4\lambda
\sqrt{s}}{(c - 1)}\norm{\widehat{\bw}}_2 f(\bbe_0) +
\frac{4c}{c-1}\lambda\sqrt{s}f(\bbe_0)\norm{\widehat{\bw}}_2 +
\lambda^2\frac{4c^2}{(c - 1)^2}s\norm{\widehat{\bw}}_2^2.
\]

Let $\mathcal{A}_\text{RE}$ denote the event that $\inf_{\bw \in \mathcal{C} \cap S^{p-1}}\frac{1}{n}\norm{\widehat{\bX} \bw}_2^2
\geq \frac{\lambda_{\min}\left(\Sigma_{\widehat{\bX}}\right)}{2}$ where we take
the set $S$ in the assumptions of proposition~\ref{prop:RE_Condition} to be $T$.
Then the assumptions of the theorem,
fact~\ref{fact:subGaussianConditionalExpectation} and
proposition~\ref{prop:RE_Condition} imply that $\pr\{\mathcal{A}_{\text{RE}}\}
\geq 1 - 2p^{-c_3}$, where $c_3$ is a constant.  Now,  plugging in $\lambda =
c\sigma_X^2 \sqrt{\frac{\log{p}}{n}}$ and taking the constant $c_{0}$ in
assumption \textbf{A4} large enough implies that $
\frac{\lambda_{\min}\left(\Sigma_{\widehat{\bX}}\right)}{2} -
\frac{4c^4}{(c-1)^2} \frac{s\log{p}}{n} \geq
\frac{\lambda_{\min}\left(\Sigma_{\widehat{\bX}}\right)}{4}$.  This implies:
\[
    \frac{\lambda_{\min}\left(\Sigma_{\widehat{\bX}}\right)}{4}\norm{\widehat{\bw}}_2^2
    \leq C\sigma_X f(\bbe_0) \sqrt{\frac{s\log{p}}{n}}.
\]
Finally, noting that inequality~\eqref{ineq:fbeta0concentration} and the
assumptions of the theorem imply
$f(\bbe_0) \leq C\left(\sigma + \sigma_X R\right)$ with probability at least $1
- \delta$. The result follows by
re-arranging the above inequality and plugging in this upper bound for
$f(\bbe_0)$.
\end{proof}

\section{General results for the LASSO}
\label{sec:GeneralLasso}
This appendix provides proofs for the outline of the general technique to prove
error bounds for the Lasso.

\begin{customproposition}{A.1}
Given a design matrix $\widetilde{\bX} \in \mathbb{R}^{n \times p}$ and $\lambda>0$, consider the solution to the convex program
\begin{align*}
\widetilde{\bbe} \in \argmin_{\bbe \in \mathbb{R}^p} \left\{\frac{1}{2n}\norm{\by -
\widetilde{\bX}\bbe}_2^2 + \lambda \norm{\bbe}_1\right\}.
\end{align*}
Further, assume
\begin{align*}
&\norm{\frac{1}{n}\widetilde{\bX}^T\left(\widetilde{\bX} -\bX\right)\bbe_0}_{\infty} \leq \frac{\lambda}{4}, 
\hspace{.5cm}
\norm{\frac{1}{n}\widetilde{\bX}^T\beps}_{\infty} \leq \frac{\lambda}{4},\\
&\text{and}~~~~
\sup_{\bw \in \mathcal{C} \cap \mathcal{S}^{p-1}}\frac{1}{n}\norm{\widetilde{\bX}\bw}_2^2 \geq \kappa,
\end{align*} 
where $\mathcal{C}$ is the cone $\mathcal{C} = \left\{\bw \in \mathbb{R}^p: \norm{\bw_{S^c}}_1 \leq
    3\norm{\bw_S}_1\right\}$.
Then $\norm{\widetilde{\bbe} - \bbe_0}_2 \leq \frac{12\lambda\sqrt{s}}{\kappa}$.
\end{customproposition}

\noindent\begin{proofof}{Proposition~\ref{prop:GeneralLasso}}
We first re-write the objective $\frac{1}{2n}\norm{\by - \widetilde{\bX} \bbe}_2^2 + \lambda \norm{\bbe}_1$ by introducing the variable
    $\bw = \bbe - \bbe_0$:
    \begin{align*}
        \frac{1}{2n}\norm{\by - \widetilde{\bX} \bbe}_2^2 + \lambda \norm{\bbe}_1 &=
        \frac{1}{2n}\norm{\widetilde{\bX}\left(\bbe - \bbe_0\right) +
            \left(\widetilde{\bX}
        - \bX\right)\bbe_0 - \beps}_2^2 + \lambda\norm{\bbe}_1 \\
        &= \frac{1}{2n}\norm{ \widetilde{\bX}\bw + \left(\widetilde{\bX} - \bX\right)\bbe_0
        - \beps}_2^2 + \lambda\norm{\bbe_0 + \bw}_1 \\
        &= \frac{1}{2n}\norm{ \widetilde{\bX}\bw }_2^2 +
        \frac{1}{2n}\norm{\left(\widetilde{\bX} - \bX\right)\bbe_0 - \beps}_2^2 +
        \left\langle\bw, \frac{1}{n}\left[\widetilde{\bX}^T\left(\widetilde{\bX} -
        \bX\right)\bbe_0 - \beps\right]\right\rangle  \\
        & + \lambda\norm{\bbe_0 +
        \bw}_1
    \end{align*}
Writing $\mathcal{L}(\bw) = \frac{1}{2n}\norm{ \widetilde{\bX}\bw }_2^2 +
        \frac{1}{2n}\norm{\left(\widetilde{\bX} - \bX\right)\bbe_0 - \beps}_2^2 +
        \left\langle\bw, \frac{1}{n}\left[\widetilde{\bX}^T\left(\widetilde{\bX} -
        \bX\right)\bbe_0 - \beps\right]\right\rangle  
         + \lambda\norm{\bbe_0 +
        \bw}_1$ and letting $\widetilde{\bw} = \argmin_{\bw}\mathcal{L}(\bw)$, we see
        that $\mathcal{L}(\widetilde{\bw}) \leq \mathcal{L}(0)$.  Thus:
    \begin{align*}
        \frac{1}{2n}\norm{ \widetilde{\bX}\widetilde{\bw} }_2^2 &\leq -
        \left\langle\bw, \frac{1}{n}\left[\widetilde{\bX}^T\left(\widetilde{\bX} -
        \bX\right)\bbe_0 - \beps\right]\right\rangle + \lambda\left(
        \norm{\bbe_0}_1 - \norm{\bbe_0 + \widetilde{\bw}}_1 \right) \\
        &\leq \norm{ \widetilde{\bw} }_1\left\{\underbrace{\norm{
                    \frac{1}{n}\widetilde{\bX}^T\left(\widetilde{\bX} -
        \bX\right)\bbe_0 }_{\infty}}_{I.} + \underbrace{\norm{
        \frac{1}{n}\widetilde{\bX}^T\beps
        }_{\infty}}_{II.}\right\}  + \lambda\left(
        \norm{\bbe_0}_1 - \norm{\bbe_0 + \widetilde{\bw}}_1 \right) \\
        &\leq \frac{\lambda}{2} \norm{\widetilde{\bw}}_1 + \lambda\left(
        \norm{\bbe_0}_1 - \norm{\bbe_0 + \widetilde{\bw}}_1 \right)
    \end{align*}
    where the second inequality follows by H{\"o}lder's inequality and the
    triangle inequality and the upper bounds on terms I. and II. are by
    assumption.  Letting $S$ denote the set of entries on which $\bbe_0$ is supported and
    noting that $\frac{1}{2n}\norm{\hat{\bX}\widetilde{\bw}}_2^2 \geq 0$, we have:
    \begin{align*}
    \lambda\norm{\bbe_{0, S} + \widetilde{\bw}_{S}}_1 +
        \lambda\norm{\widetilde{\bw}_{S^c}}_1 \leq
        \frac{\lambda}{2}\norm{\widetilde{\bw}_S}_1 +
        \frac{\lambda}{2}\norm{\widetilde{\bw}_{S^c}}_1 + \lambda \norm{\bbe_{0, S}}_1
    \end{align*}
    Thus, by the triangle inequality, we see that $\norm{\widetilde{\bw}_{S^c}}_1 \leq
    3\norm{\widetilde{\bw}_S}_1$.  Thus, we see that the solution $\widetilde{\bw}$ belongs
    to the cone $\mathcal{C}$.  Now, by the third assumption, we see that
    $\frac{1}{2n}\norm{\widetilde{\bX}\widetilde{\bw}}_2^2 \geq
    \frac{\kappa}{2}$.  Thus,
    we have:
    \begin{align*}
        \frac{\kappa}{2}\norm{\widetilde{\bw}}_2^2 &\leq
        \frac{\lambda}{2}\norm{\widetilde{\bw}}_1 + \lambda\left(\norm{\bbe_0}_1 -
        \norm{\widetilde{\bw} + \bbe_0}_1\right) \\
        &\leq \frac{3\lambda}{2}\norm{\widetilde{\bw}_S + \widetilde{\bw}_{S^c}}_1 \\
        &\leq 6\sqrt{s}\lambda\norm{\widetilde{\bw}_S}_2
    \end{align*}
    Re-arranging the inequality gives the result.
\end{proofof}

\subsection{Restricted Eigenvalue Condition}
\noindent We restate the result for convenience.
\begin{customproposition}{A.5} 
    Let $\bX$ be a matrix with i.i.d. rows, each of which is sub-Gaussian with
    parameter $\sigma_X^2$ and has covariance matrix $\Sigma_{\bX}$.
    Then  $\bX$ satisfies a
                    restricted eigenvalue condition:
    \begin{align*}
        \sup_{\bw \in \mathcal{C} \cap S^{p-1}}\frac{1}{n}\norm{\bX \bw}_2^2
        \geq \frac{\lambda_{\text{min}}\left(\Sigma_X\right)}{2}
    \end{align*}
    with probability at least 
    \begin{align*}
        1 -
                2\exp\left\{-c n \left(\min\left(\frac{t^2}{\sigma_X^4},
                \frac{t}{\sigma_X^2}\right) + \frac{s \log{p}}{n}\right)\right\}
    \end{align*}
    where $\mathcal{C}$ is the cone $\mathcal{C} = \left\{\bw \in \mathbb{R}^p: \norm{\hat{\bw}_{S^c}}_1 \leq
    3\norm{\hat{\bw}_S}_1\right\}$. 
\end{customproposition}

\noindent\begin{proofof}{Proposition~\ref{prop:RE_Condition}}
    The proof can be found in Loh and Wainwright~\cite{loh2012high}, however we
    repeat much of the argument here for clarity.  The proof will follow three
    main steps:
    \begin{itemize}
        \item[1: ] Noting that $\mathcal{C} \cap B_2(1) \subseteq
            B_1(\sqrt{16s}) \cap B_2(1)$, we show $B_1(\sqrt{s}) \cap B_2(1) \subseteq
        3\text{cl}\left\{\text{conv}\left\{B_0(s) \cap B_2(1)\right\}\right\}$
    \item[2: ] Argue that on the simpler set $\bw \in B_0(s) \cap B_2(1)$,
        the behavior of the desired quadratic form concentrates around its
            expectation
    \item[3: ] Show that any $\bw \in
        3\text{cl}\left\{\text{conv}\left\{B_0(s) \cap B_2(1)\right\}\right\}$
            maintains the same bounds from Step 2
    \end{itemize}
    We now prove each step.
    \begin{itemize}
    \item[1: ] This can be found in Loh and
        Wainwright~\cite{loh2012high} Lemma 11.
    \item[2: ] We claim:
        \begin{align*}
        \pr\left\{\sup_{\bw \in B_0(s) \cap B_2(1)}\frac{1}{n}\left\lvert
            \norm{\bX \bw}_2^2 - \E\norm{\bX \bw}_2^2\right\rvert \geq
            t\right\} \leq 2 p^{s} 9^{s}\exp\left\{-c n
                \min\left(\frac{t^2}{\sigma_X^4}, \frac{t}{\sigma_X^2}\right)\right\}
        \end{align*}
            To do this, we first consider sets $S_U = \left\{\bw \in
            \mathbb{R}^p: \norm{\bw}_2 \leq 1, \text{supp}(\bw) \subseteq
            U\right\}$ and notice that $B_0(s) \cap
            B_2(1) = \bigcup_{\lvert U \rvert = s}S_U$. Now, let
            $\mathcal{N}_{\epsilon}$ be a $\epsilon$-cover of $S_U$ and note that
            there exists such a set with $\lvert \mathcal{N}_{\epsilon} \rvert \leq
            (\frac{3}{\epsilon})^{s}$ \cite{vershynin2018high}.  Taking
            $\widebar{\bX} \equiv \frac{1}{n}\bX^{T}\bX -
            \frac{1}{n}\E\bX^{T}\bX$, we re-write $\frac{1}{n}\left\lvert
            \norm{\bX \bw}_2^2 - \E\norm{\bX \bw}_2^2\right\rvert = \left\lvert
            \left\langle \bw, \widebar{\bX}\bw \right\rangle \right\rvert$ and
            see $\sup_{\bw \in S_U}\left\lvert
            \left\langle \bw, \widebar{\bX}\bw \right\rangle \right\rvert \leq
            \frac{1}{1 - 2\epsilon}\sup_{\bw \in \mathcal{N}_{\epsilon}}\left\lvert
            \left\langle \bw, \widebar{\bX}\bw \right\rangle \right\rvert$.
            Thus, we have:
            \begin{align*}
            \pr\left\{\sup_{\bw \in B_0(s) \cap B_2(1)}\frac{1}{n}\left\lvert
            \norm{\bX \bw}_2^2 - \E\norm{\bX \bw}_2^2\right\rvert \geq
                t\right\} &\leq \binom{p}{s}\pr\left\{\sup_{\bw \in S_U}\left\lvert
            \left\langle \bw, \widebar{\bX}\bw \right\rangle \right\rvert \geq
                t\right\} \\
                &\leq \binom{p}{s}\pr\left\{\sup_{\bw \in
                \mathcal{N}_{\epsilon}}\left\lvert
            \left\langle \bw, \widebar{\bX}\bw \right\rangle \right\rvert \geq
                (1 - 2\epsilon)t\right\} \\
                &\leq
                \binom{p}{s}\left\lvert\mathcal{N}_{\epsilon}\right\rvert\pr\left\{\left\lvert
                \left\langle \bw, \widebar{\bX}\bw \right\rangle \right\rvert
                \geq (1 - 2\epsilon)t\right\}
            \end{align*}
            Noting that the square of a sub-Gaussian random variable is
            sub-Exponential, we see that if $t \leq \frac{4e\sigma_X^2}{1 -
            2\epsilon}$, then, by Bernstein's Inequality (see e.g.
            Vershynin~\cite{vershynin2018high} Theorem 2.8.1):
            \begin{align*}
            \pr\left\{\left\lvert
                \left\langle \bw, \widebar{\bX}\bw \right\rangle \right\rvert
                \geq (1 - 2\epsilon)t\right\} \leq 2\exp\left\{-c n
                \min\left(\frac{t^2(1 - 2\epsilon)^2}{\sigma_X^4}, \frac{t (1 -
            2\epsilon)}{\sigma_X^2}\right)\right\} 
            \end{align*}
            Taking $\epsilon = \frac{1}{3}$ and noting that $\binom{p}{s} \leq p^s$
            implies the claim.
        \item[3: ] Suppose that for a fixed matrix $\bX$, $\left\lvert
            \left\langle \bw, \bX \bw \right\rangle \right\rvert \leq \delta$
            for all $\bw \in B_0(16s) \cap
            B_2(1)$. We claim that for all $\bw \in \mathcal{C}$ that $\left\lvert
            \left\langle \bw, \bX \bw \right\rangle \right\rvert \leq 27\delta$.
            To see this, let $\bw \in 3\text{conv}\left\{B_0(16s) \cap B_2(1)\right\}$.
            Then we can write $\bw = \sum_{i=1}^{k}\alpha_i \bw_i$ where
            $\alpha_i$ are non-negative weights satisfying $\sum_{i=1}^k
            \alpha_i = 1$ and $\bw_i \in B_0(16s) \cap B_2(3)$.  Then, we can see
            that:
            \begin{align*}
            \left\lvert
                \left\langle \bw, \bX \bw \right\rangle \right\rvert & = \left\lvert
                \left\langle \sum_{i=1}^{k}\alpha_i\bw_i, \bX
                \sum_{i=1}^{k}\alpha_i\bw_i \right\rangle \right\rvert \\
                &= \left\lvert \sum_{i, j}
            \alpha_i\alpha_j\left\langle \bw_i, \bX \bw_j \right\rangle
                \right\rvert \\
                &= \left\lvert \sum_{i,j}
                \frac{\alpha_i \alpha_j}{2}\left(\left\langle \bw_i + \bw_j, \bX
                (\bw_i + \bw_j)
                \right\rangle  - \left\langle \bw_i, \bX \bw_i
                \right\rangle - \left\langle \bw_j, \bX \bw_j
                \right\rangle\right)\right\rvert \\
                &\leq \sum_{i,j}\frac{\alpha_i \alpha_j}{2} \cdot \left\lvert
                36\delta + 9 \delta + 9 \delta\right\rvert = 27\delta
            \end{align*}
            Now since $\left \lvert \left\langle \bw, \bX \bw
                \right\rangle \right\rvert$ is a continuous function of $\bw$
                and by step $1$, we see that if for all $\bv \in B_0(16s) \cap
                B_2(1)$, $\left\lvert\left\langle \bv, \bX \bv
                \right\rangle \right\rvert \leq \delta$, then for all $\bw \in
                B_1(\sqrt{16s})
                \cap B_2(1)$, $\left\lvert\left\langle \bw, \bX \bw
                \right\rangle \right\rvert \leq 27\delta$.  This implies the
                claim since $\mathcal{C} \cap B_2(1) \subseteq B_1(\sqrt{16s})
                \cap B_2(1)$
    \end{itemize}
To conclude, let us work on the event $\mathcal{A} = \left\{\widebar{\bX}: \left\lvert
                \left\langle \bw, \widebar{\bX}\bw \right\rangle \right\rvert
                \leq \frac{\lambda_{\text{min}}(\Sigma_{\bX})}{2} \quad \forall
                \bw \in \mathcal{C} \cap B_2(1)\right\}$.  Conditioned on this
                event, we see that
                \begin{align*}
                \left\lvert
                    \left\langle \bw, \frac{1}{n}\bX^T\bX\bw \right\rangle -
                    \left\langle \bw, \frac{1}{n}\E \bX^T\bX\bw \right\rangle
                    \right\rvert \leq \frac{\lambda_{\text{min}}(\Sigma_{\bX})}{2}
                \end{align*}
                An application of the triangle inequality yields:
                \begin{align*}
                \left\lvert
                    \left\langle \bw, \frac{1}{n}\E\bX^T\bX\bw \right\rangle
                    \right\rvert - \left\lvert
                    \left\langle \bw, \frac{1}{n}\bX^T\bX\bw \right\rangle
                    \right\rvert \leq \frac{\lambda_{\text{min}}(\Sigma_{\bX})}{2}
                \end{align*} 
                and now noting that $\left\lvert\left\langle \bw, \frac{1}{n}\E\bX^T\bX\bw \right\rangle
                    \right\rvert \geq \lambda_{\text{min}}(\Sigma_{\bX})$ and
                    re-arranging implies $\left\langle \bw, \frac{1}{n} \bX^{T}
                    \bX \bw \right\rangle \geq
                    \frac{\lambda_{\text{min}}(\Sigma_{\bX})}{2} \quad \forall
                    \bw \in \mathcal{C} \cap B_2(1)$.  We now control
                    $\pr\{\mathcal{A}^c\}$.  By step 3, it suffices to use the
                    bound of step 2:
                    \begin{align*}
                   \pr\left\{\sup_{\bw \in B_0(s) \cap B_2(1)}\frac{1}{n}\left\lvert
            \norm{\bX \bw}_2^2 - \E\norm{\bX \bw}_2^2\right\rvert \geq
            t\right\} &\leq 2 p^{s} 9^{s}\exp\left\{-c n
                \min\left(\frac{t^2}{\sigma_X^4}, \frac{t}{\sigma_X^2}\right)\right\} \\
                    \end{align*}
                Taking $t = \frac{\lambda_{\text{min}}(\Sigma_{\bX})}{54}$, we
                have the result:
                \begin{align*}
                \pr\left\{\mathcal{A}\right\} \geq 1 -
                2\exp\left\{-c n \left(\min\left(\frac{t^2}{\sigma_X^4},
                \frac{t}{\sigma_X^2}\right) + \frac{s \log{p}}{n}\right)\right\}
                \end{align*}
                and we are done.
\end{proofof}

\section{Proofs of technical lemmas from appendix~\ref{sec:ProofsConditionalExpectation}}
\label{sec:appendixConditionalExpectation}
This appendix is dedicated to the proofs of the $\ell_\infty$ terms needed in theorem~\ref{thm:mainres}, corollary~\ref{cor:IdentityCovariance}, and theorem~\ref{thm:mnar}.  In particular, we devote one subsection to each of lemmas~\ref{lem:term1thm1},~\ref{lem:term2thm1},~\ref{lem:term1cor1},~\ref{lem:term2cor1},~\ref{lem:term1mnar}, and~\ref{lem:term2mnar}.  The strategy for each of the proofs is essentially the same, although the techniques vary significantly.  Each of these lemmas is concerned with one of the two random vectors $\widehat{\bX}^T\beps$ or $\widehat{\bX}^T\left(\widehat{\bX} - \bX\right)\bbe_0$.  We will be interested in the concentration of $\frac{1}{n}\norm{\cdot}_{\infty}$ for each of these random vectors.  In particular, we will aim to show in each of these lemmas that the random vector is the empirical average of a sum of sub-exponential with as tight a sub-exponential parameter as possible and we will conclude using a standard concentration inequality for sub-exponential random variables, such as Bernstein's inequality~\cite{vershynin2018high}.

\subsection{Proof of lemma~\ref{lem:term1thm1}}
\label{subsec:term1thm1}
Recall the definition of $\lambda$ in equation~\eqref{eq:mainlambda}:
\begin{align*}
\lambda = 4A\left(8e\sigma_X \sigma + 16\sqrt{2}e\sigma_X^2\sqrt{1 - \alpha}R\right)\sqrt{\frac{\log{p}}{n}}.
\end{align*}
We re-state the lemma for the reader's convenience:
\begin{customlemma}{A.2}
Assuming $\lambda$ as defined in equation~\eqref{eq:mainlambda}, and the assumptions of theorem~\ref{thm:mainres}, we have:
\[
\pr\left\{\norm{\frac{1}{n}\widehat{\bX}^T\left(\widehat{\bX} - \bX\right)\bbe_0}_{\infty} \geq \frac{\lambda}{4}\right\} \leq 2p^{1 - \frac{A^2}{2}}.
\]
\end{customlemma}
\begin{proof}
To clarify our strategy, let us write:
\begin{align*}
\norm{ \frac{1}{n}\widehat{\bX}^T\left(\widehat{\bX} -
                \bX\right)\bbe_0 }_{\infty} &= \norm{\bbe_0}_2 \max_{a=1, 2, \dots, p} \left\lvert
                \frac{1}{n}\sum_{i=1}^{n}\widehat{X}_{ia}\left\langle
                    \bX_i - \widehat{\bX}_i, \widetilde{\bbe}\right\rangle\right \rvert,
\end{align*}
where we have taken $\widetilde{\bbe_0} = \bbe_0/\norm{\bbe_0}_2$.  We will now show that the random variable $\widehat{X}_{ia}\left\langle
                    \bX_i - \widehat{\bX}_i, \widetilde{\bbe}\right\rangle$ is sub-exponential.  Take $\theta$ such that $\lvert \theta \rvert \leq \frac{1}{16e\sigma_X^2}$ and let $c_X = 16e\sigma_X^2$.  We now control the moment generating function (using the notation $S^{c} = [p] \ S$ for $S \subseteq [p]$):
             		\begin{align*}
             		\E \exp\left\{\theta\widehat{X}_{ia}\left\langle
                    \bX_i - \widehat{\bX}_i, \widetilde{\bbe}\right\rangle\right\} &= \sum_{S
                \subseteq [p]}\alpha^{\lvert S\rvert}(1 - \alpha)^{p - \lvert S
                \rvert} \E \exp\left\{\theta\widehat{X}_{ia}\left\langle
                \widehat{\bX}_{i,S^c} -
                \bX_{i, S^c}, \widetilde{\bbe}_{0,S^c} \right \rangle\right\}. 
                \end{align*}
                Notice now that by the orthogonality property of the conditional expectation,  
                $\E \widehat{X}_{ia}\left\langle
                \widehat{\bX}_{i,S^c} -
                \bX_{i, S^c}, \widetilde{\bbe}_{0,S^c} \right \rangle = 0$.  Additionally, note that $\widehat{X}_{ia}$ is $\sigma_X^2$ sub-gaussian and by fact~\ref{fact:sumsubgaussian}, 
                $\left\langle
                \widehat{\bX}_{i,S^c} -
                \bX_{i, S^c}, \widetilde{\bbe}_{0,S^c} \right\rangle$ is $4\sigma_X^2\norm{\widetilde{\bbe}_{0,S^c}}_2^2$                 sub-gaussian.  Thus, by lemma~\ref{lem:productsubgaussian}, 
                \begin{align*}
                \sum_{S
                \subseteq [p]}\alpha^{\lvert S\rvert}(1 - \alpha)^{p - \lvert S
                \rvert} \E \exp\left\{\theta\widehat{X}_{ia}\left\langle
                \widehat{\bX}_{i,S^c} -
                \bX_{i, S^c}, \widetilde{\bbe}_{0,S^c} \right \rangle\right\} &\leq \sum_{S \subseteq [p]} \alpha^{\lvert S\rvert}(1 -
                \alpha)^{p - \lvert S
                \rvert} e^{\frac{1}{2}\theta^2 c_X^2 \norm{\widetilde{\bbe}_{0, S^c}}_2^2}
                \\ 
                &= \E_S \exp\left\{\frac{1}{2}\theta^2 c_X^2 \norm{\widetilde{\bbe}_{0,
                S^c}}_2^2\right\}.
             		\end{align*}
             		We now control the term $\E_S \exp\left\{\frac{1}{2}\theta^2 c_X^2 \norm{\widetilde{\bbe}_{0,
                S^c}}_2^2\right\}$ using a technique similar to Herbst's argument~\cite{boucheron2013concentration}.  Define $\phi(\theta) \equiv \log \E_S
            \text{exp}\left\{\frac{1}{2}\theta^2 c_X^2 \norm{\widetilde{\bbe}_{0,
                S^c}}_2^2\right\}$.  Notice that $\phi(0) = 0$ and $\phi'(0) = 0$.  Thus, we
                have $\phi(\theta) = \int_{0}^{\theta}
                \int_{0}^{t_1}\phi''(t_2)d t_2 d t_1$.  Now, 
                \begin{align*}
                \phi''(\theta) &= c_X^2\E_S\left(\norm{\widetilde{\bbe}_{0,
                    S^c}}_2^2 z_S\right) + \theta^2
                    c_X^4\left(\E_S\left(\norm{\widetilde{\bbe}_{0,
                    S^c}}_2^4z_S\right) - \left[\E_S\left(\norm{\widetilde{\bbe}_{0,
                    S^c}}_2^2z_S\right)\right]^2\right) \\
                    &\leq c_X^2\E_S\left(\norm{\widetilde{\bbe}_{0,
                    S^c}}_2^2 z_S\right) + \theta^2
                    c_X^4\left(\E_S\left(\norm{\widetilde{\bbe}_{0,
                    S^c}}_2^4z_S\right)\right)
                \end{align*}
                where $z_S = \frac{e^{\frac{1}{2}\theta^2c_X^2\norm{\widetilde{\bbe}_{0,
                    S^c}}_2^2}}{\E_Se^{\frac{1}{2}\theta^2c_X^2\norm{\widetilde{\bbe}_{0,
                    S^c}}_2^2}}$.  Notice that
                    $e^{-\frac{1}{2}c_X^2\theta^2} \leq z_S \leq
                    e^{\frac{1}{2}c_X^2\theta^2}$.  Additionally, notice
                    that $\E_S\norm{\widetilde{\bbe}_{0,
                    S^c}}_2^2 = 1 - \alpha$ and $\E_S \norm{\widetilde{\bbe}_{0,
                    S^c}}_2^4 \leq 1 - \alpha$ (since $\widetilde{\bbe}_0$ is a unit
                    norm vector).  We thus see that:
                    \begin{align*}
                        \frac{\phi''(\theta)}{c_X^2(1 - \alpha)} \leq
                        e^{\frac{1}{2}c_X^2 \theta^2} + \theta^2
                        c_X^2 e^{\frac{1}{2}c_X^2\theta^2}.
                    \end{align*}
                    Now, take $\theta$ such that $\lvert \theta \rvert \leq \frac{1}{\sqrt{2}c_X}$ thus giving $\frac{\phi''(\theta)}{c_X^2(1 - \alpha)} \leq 2$.  Therefore:
                    \[\phi(\theta) = \int_{0}^{\theta}
                \int_{0}^{t_1}\phi''(t_2)d t_2 d t_1 \leq \theta^2 (1 - \alpha) c_X^2.\]
                This implies that for $\lvert \theta \rvert \leq \frac{1}{\sqrt{2}c_X}$
                \begin{align}
                \label{eq:subexpterm}
                \E \exp\left\{\theta\widehat{X}_{ia}\left\langle
                    \bX_i - \widehat{\bX}_i, \widetilde{\bbe}\right\rangle\right\} \leq e^{\theta^2 (1 - \alpha) c_X^2},
                \end{align}
                 so the desired term is sub-exponential.  To conclude, notice that:
                 \begin{align*}
                 \pr\left\{\frac{1}{n}\norm{\widehat{\bX}^T\left(\widehat{\bX} - \bX\right)\bbe_0}_{\infty} \geq \frac{\lambda}{4}\right\} &\leq  \pr\left\{\frac{1}{n}\norm{\widehat{\bX}^T\left(\widehat{\bX} - \bX\right)\widetilde{\bbe}_0}_{\infty} \geq \sqrt{2}Ac_X\sqrt{1 - \alpha} \sqrt{\frac{\log{p}}{n}}\right\} \\
                &= \pr\left\{\max_{a=1, 2, \dots, p} \left\lvert
                \frac{1}{n}\sum_{i=1}^{n}\widehat{X}_{ia}\left\langle
                    \bX_i - \widehat{\bX}_i, \widetilde{\bbe}\right\rangle\right \rvert \geq \sqrt{2}Ac_X\sqrt{1 - \alpha} \sqrt{\frac{\log{p}}{n}}\right\}.
                 \end{align*}
                 Under the assumption that $\sqrt{\frac{\log{p}}{n}} \leq \frac{\sqrt{1 - \alpha}}{A}$, we conclude by a union bound over $a \in [p]$ and invoking lemma~\ref{lem:subexpconcentration} with $Z_i = \widehat{X}_{ia}\left\langle
                    \bX_i - \widehat{\bX}_i, \widetilde{\bbe}\right\rangle$, $\widebar{\theta} = \frac{1}{\sqrt{2}c_X}$ and $\sigma_Z^2 = 2(1 - \alpha)c_X^2$, and $t = \sqrt{2}Ac_X\sqrt{1 - \alpha} \sqrt{\frac{\log{p}}{n}}$.

\end{proof}
\subsection{Proof of lemma~\ref{lem:term2thm1}}
\label{subsec:term2thm1}
We repeat the lemma for the reader's convenience:
\begin{customlemma}{A.3}
Assuming $\lambda$ as defined in equation~\eqref{eq:mainlambda}, and the assumptions of theorem~\ref{thm:mainres}, we have:
\[
\pr\left\{\norm{\frac{1}{n}\widehat{\bX}^T \beps}_{\infty} \geq \frac{\lambda}{4}\right\} \leq 2p^{1 - \frac{A^2}{2}}.
\]
\end{customlemma}
\begin{proof}
We write $\norm{\frac{1}{n}\widehat{\bX}^T \beps}_{\infty} = \max_{a = 1, 2, \dots, p} \left\lvert \frac{1}{n}\sum_{i=1}^{n}\widehat{X}_{ia}\epsilon_{i}\right\rvert$.  By lemma~\ref{lem:productsubgaussian}, the random variable $\widehat{X}_{ia}\epsilon_{i}$ is sub-exponential: for all $\theta \leq (8e\sigma_X\sigma)^{-1}$, $\E e^{\theta \widehat{X}_{ia} \epsilon_i} \leq e^{32e^2\theta^2\sigma_X^2 \sigma^2}$.  We are thus interested in:
\[\pr\left\{\norm{\frac{1}{n}\widehat{\bX}^T \beps}_{\infty} \geq \frac{\lambda}{4}\right\} \leq \pr\left\{\max_{a = 1, 2, \dots, p} \left\lvert \frac{1}{n}\sum_{i=1}^{n}\widehat{X}_{ia}\epsilon_{i}\right\rvert \geq 8eA\sigma_X\sigma\sqrt{\frac{\log{p}}{n}}\right\}\]

Thus, under the assumption that $\sqrt{\frac{\log{p}}{n}} \leq (A)^{-1}$, we conclude by invoking a union bound over $a \in [p]$ and lemma~\ref{lem:subexpconcentration} with $Z_i = \widehat{X}_{ia}\epsilon_i$, $\widebar{\theta} = (8e\sigma_X\sigma)^{-1}$, $\sigma_Z^2 = 64e^2\sigma_X^2\sigma^2$, and $t = 8eA\sigma_X\sigma\sqrt{\frac{\log{p}}{n}}$.
\end{proof}

\subsection{Proof of lemma~\ref{lem:term1cor1}}
\label{subsec:term1cor}
Let us first copy the regularization parameter $\lambda$ given in \eqref{eq:lambdacor}:
\[\lambda = 4A\sqrt{\alpha}\left(8\sqrt{2}e\sigma + 32e\sqrt{1 - \alpha}R\right)\sqrt{\frac{\log{p}}{n}}.\]
We repeat the lemma for the reader's convenience: 
\begin{customlemma}{A.6}
Assuming $\lambda$ as defined in equation~\eqref{eq:lambdacor}, and the assumptions of corollary~\ref{cor:IdentityCovariance}, we have:
\[
\pr\left\{\norm{\frac{1}{n}\widehat{\bX}^T\left(\widehat{\bX} - \bX\right)\bbe_0}_{\infty} \geq \frac{\lambda}{4}\right\} \leq 2p^{1 - \frac{A^2}{2}}.
\]
\end{customlemma}
\begin{proof}
As in the proof of lemma~\ref{lem:term1thm1}, we take
\begin{align*}
\norm{ \frac{1}{n}\widehat{\bX}^T\left(\widehat{\bX} -
                \bX\right)\bbe_0 }_{\infty} &= \norm{\bbe_0}_2 \max_{a=1, 2, \dots, p} \left\lvert
                \frac{1}{n}\sum_{i=1}^{n}\widehat{X}_{ia}\left\langle
                    \bX_i - \widehat{\bX}_i, \widetilde{\bbe}_0\right\rangle\right \rvert.
\end{align*}
We again would like to show that $\widehat{X}_{ia}\left\langle
                    \bX_i - \widehat{\bX}_i, \widetilde{\bbe}_0\right\rangle$ is sub-exponential.  Of course, we have done this already in the proof of lemma~\ref{lem:term1thm1}; however, in this simpler case we will be able to get tighter control on the sub-exponential constant.  We have:
                    \begin{align*}
                    \E \exp\left\{\theta\widehat{X}_{ia}\left\langle
                    \bX_i - \widehat{\bX}_i, \widetilde{\bbe}_0\right\rangle\right\} &= (1 - \alpha) + \alpha \E \exp\left\{\theta X_{ia}\sum_{b \neq a} (X_{ib} - \widehat{X}_{ib})\widetilde{\bbe}_{0,b}\right\}.
                    \end{align*}
Now, by inequality~\eqref{eq:subexpterm}, for $\lvert \theta \rvert \leq (256e^2\sqrt{2})^{-1}$, 
\[\E \exp\left\{\theta X_{ia}\sum_{b \neq a} (X_{ib} - \widehat{X}_{ib})\widetilde{\bbe}_{0,b}\right\} \leq \exp\left\{256e^2\theta^2(1 - \alpha)\right\}.\]
and we are interested in $(1 - \alpha) + \alpha \exp\left\{256e^2\theta^2(1 - \alpha)\right\}$.  Note that for any $\lvert \theta \rvert \leq (256\sqrt{2}e^2)^{-1}$, $256e^2\theta^2(1 - \alpha) \leq 1$.  We thus employ the numeric inequality $1 + \alpha(e^x - 1) \leq 1 + 2\alpha x \leq e^{2\alpha x}$ for $x \leq 1$ to see that:
\begin{align*}
(1 - \alpha) + \alpha \E \exp\left\{\theta X_{ia}\sum_{b \neq a} (X_{ib} - \widehat{X}_{ib})\widetilde{\bbe}_{0,b}\right\} &\leq (1 - \alpha) + \alpha \exp\left\{256e^2\theta^2(1 - \alpha)\right\} \\
&\leq \exp\left\{512e^2 \theta^2 \alpha(1-\alpha)\right\}.
\end{align*}
We are ready to conclude.  Notice that:
\begin{align*}
\pr\left\{\frac{1}{n}\norm{\widehat{\bX}^T\left(\widehat{\bX} - \bX\right)\bbe_0}_{\infty} \geq \frac{\lambda}{4}\right\} &\leq  \pr\left\{\frac{1}{n}\norm{\widehat{\bX}^T\left(\widehat{\bX} - \bX\right)\widetilde{\bbe}_0}_{\infty} \geq 32\sqrt{2}eA\sqrt{\alpha(1 - \alpha)} \sqrt{\frac{\log{p}}{n}}\right\} \\
&= \pr\left\{\max_{a=1, 2, \dots, p} \left\lvert
                \frac{1}{n}\sum_{i=1}^{n}\widehat{X}_{ia}\left\langle
                    \bX_i - \widehat{\bX}_i, \widetilde{\bbe}\right\rangle\right \rvert \geq 32\sqrt{2}eA\sqrt{\alpha(1 - \alpha)} \sqrt{\frac{\log{p}}{n}}\right\}.
\end{align*}
Under the assumption that $\sqrt{\frac{\log{p}}{n}} \leq \frac{\sqrt{\alpha(1 - \alpha)}}{8\sqrt{2}eA}$, the desired result follows by a union bound and using lemma~\ref{lem:subexpconcentration} with $Z_i = \widehat{X}_{ia}\left\langle
                    \bX_i - \widehat{\bX}_i, \widetilde{\bbe}\right\rangle$, $\widebar{\theta} = \frac{1}{256e^2\sqrt{2}}$, $\sigma_Z^2 = 1024e^2\alpha(1 - \alpha)$, and $t = 32\sqrt{2}eA\sqrt{\alpha(1 - \alpha)} \sqrt{\frac{\log{p}}{n}}$.
\end{proof}
\subsection{Proof of lemma~\ref{lem:term2cor1}}
\label{subsec:term2cor1}
We repeat this lemma here:
\begin{customlemma}{A.7}
Assuming $\lambda$ as defined in equation~\eqref{eq:lambdacor}, and the assumptions of corollary~\ref{cor:IdentityCovariance}, we have:
\[
\pr\left\{\norm{\frac{1}{n}\widehat{\bX}^T \beps}_{\infty} \geq \frac{\lambda}{4}\right\} \geq 2p^{1 - \frac{A^2}{2}}.
\]
\end{customlemma}
\begin{proof}
We write $\norm{\frac{1}{n}\widehat{\bX}^T \beps}_{\infty} = \max_{a = 1, 2, \dots, p} \left\lvert \frac{1}{n}\sum_{i=1}^{n}\widehat{X}_{ia}\epsilon_{i}\right\rvert$.  Notice that 
\[\E e^{\theta \widehat{X}_{ia} \epsilon_i} = (1 - \alpha) + \alpha \E e^{\theta X_{ia} \epsilon_i} \leq (1 - \alpha) + \alpha e^{32e^2\theta^2\sigma^2}.\]
Where the last inequality holds for all $\theta \leq (8e\sigma)^{-1}$ by lemma~\ref{lem:productsubgaussian}.  Thus by the numeric inequality $1 + \alpha(e^x - 1) \leq 1 + 2\alpha x \leq e^{2\alpha x}$ for $x \leq 1$, we have $\E e^{\theta \widehat{X}_{ia}\epsilon_i} \leq e^{64e^2\theta^2 \alpha \sigma^2}$.  We are thus interested in:
\[\pr\left\{\norm{\frac{1}{n}\widehat{\bX}^T \beps}_{\infty} \geq \frac{\lambda}{4}\right\} \leq \pr\left\{\max_{a = 1, 2, \dots, p} \left\lvert \frac{1}{n}\sum_{i=1}^{n}\widehat{X}_{ia}\epsilon_{i}\right\rvert \geq 8\sqrt{2}eA\sqrt{\alpha}\sigma\sqrt{\frac{\log{p}}{n}}\right\}\]
Thus, under the assumption that $\sqrt{\frac{\log{p}}{n}} \leq \frac{2\sqrt{\alpha}}{\sqrt{2}A}$, we conclude by invoking a union bound over $a \in [p]$ and lemma~\ref{lem:subexpconcentration} with $Z_i = \widehat{X}_{ia}\epsilon_i$, $\widebar{\theta} = (8e\sigma)^{-1}$, $\sigma_Z^2 = 128e^2\alpha\sigma^2$, and $t = 8\sqrt{2}eA\sqrt{\alpha}\sigma\sqrt{\frac{\log{p}}{n}}$.
\end{proof}

\subsection{Proof of lemma~\ref{lem:term1mnar}}
\label{subsec:term1mnar}
We repeat the regularization parameter $\lambda$ from ~\eqref{eq:lambdamnar}:
\[
\lambda = 4A\left(8e\sigma_X \sigma + 480e\sigma_X^2 R\right)\sqrt{\frac{\log{p}}{n}}\]
We now repeat the lemma for convenience:
\begin{customlemma}{A.8}
Assuming $\lambda$ as defined in equation~\eqref{eq:lambdamnar}, and the assumptions of corollary~\ref{thm:mnar}, we have:
\[
\pr\left\{\norm{\frac{1}{n}\widehat{\bX}^T\left(\widehat{\bX} - \bX\right)\bbe_0}_{\infty} \geq \frac{\lambda}{4}\right\} \leq 6p^{1 - \frac{A^2}{2}}.
\]
\end{customlemma}
\begin{proof}
We first re-write 
            \begin{align*}\norm{ \frac{1}{n}\widehat{\bX}^T\left(\widehat{\bX} -
                \bX\right)\bbe_0 }_{\infty} 
                &= \norm{\bbe_0}_2 \max_{a=1, 2, \dots, p} \left\lvert \frac{1}{n}
                \left\langle \widehat{\bX}\be_a, \left(\bX - \widehat{\bX
                }\right)\widetilde{\bbe}_0\right\rangle\right\rvert.
            \end{align*}
            Note that $\widehat{\bX}\be_a$ and $\left(\bX - \widehat{\bX}\right)\widetilde{\bbe}_0$ are both sub-gaussian random vectors and that $\E\left\langle \widehat{\bX}\be_i, \left(\bX - \widehat{\bX
                }\right)\widetilde{\bbe}_0\right\rangle = 0$ by the orthogonality
                property of the conditional expectation.  We will thus re-write the inner product as:
                \begin{align*}
                &\frac{\norm{\bbe_0}_2}{2} \max_{a=1, 2, \dots, p}\Biggl
                \lvert \underbrace{\left(\frac{1}{n}\norm{\widehat{\bX}\be_a + \left(\bX -
                \widehat{\bX}\right)\widetilde{\bbe}_0}_2^2 - \frac{1}{n}\E\norm{\widehat{\bX}\be_a + \left(\bX -
                \widehat{\bX}\right)\widetilde{\bbe}_0}_2^2\right)}_{I.} \\
                & - \underbrace{\left(\frac{1}{n}\norm{\widehat{\bX}\be_a}_2^2 -
                \frac{1}{n}\E\norm{\widehat{\bX}\be_i}_2^2\right)}_{II.} - \underbrace{\left(\frac{1}{n}\norm{\left(\bX -
                \widehat{\bX}\right)\widetilde{\bbe}_0}_2^2 - \E\norm{\left(\bX -
                \widehat{\bX}\right)\widetilde{\bbe}_0}_2^2\right)}_{III.}\Biggr\rvert.
                \end{align*}
                We thus have:
                A union bound as well as noting that $8e\sigma_X\sigma \geq 0$ gives:
                \begin{align*}
                \pr\left\{\norm{\frac{1}{n}\widehat{\bX}^T\left(\widehat{\bX} -
                \bX\right)\bbe_0}_{\infty} \geq \frac{\lambda}{4}\right\} &\leq
                p\pr\left\{\left\lvert I. - II. - III. \right\rvert \geq 432\sqrt{2}eA\sigma_X^2\sqrt{\frac{\log{p}}{n}}\right\}.
                \end{align*}
                We thus reduce ourselves to examining $\pr\left\{\lvert I. - II. -
                III.\rvert \geq t\right\}$.  Notice then that $\pr\left\{\lvert I. -
            II. - III.\rvert \geq t\right\} \leq \sum_{I \in
        \{I., II., III.\}}\pr\left\{\lvert I \rvert \geq \frac{t}{3}\right\}$.  We go
        through this sum term by term.  First, we examine $\pr\{\lvert I. \rvert
            \geq \frac{t}{3}\}$.  We re-write:
            \[\frac{1}{n}\norm{\widehat{\bX}\be_a + \left(\bX -
            \widehat{\bX}\right)\widetilde{\bbe}_0}_2^2 =
            \frac{1}{n}\sum_{i=1}^{n}\left(\left\langle
                        \widehat{\bX}_i, \be_a\right\rangle +
        \left\langle \bX_i - \widehat{\bX}_i, \widetilde{\bbe}_0\right\rangle
\right)^2.\]
Note that by two applications of fact~\ref{fact:sumsubgaussian}, $\left\langle
                        \widehat{\bX}_i, \be_a\right\rangle +
        \left\langle \bX_i - \widehat{\bX}_i, \widetilde{\bbe}_0\right\rangle$ is
        $10\sigma_X^2$ sub-gaussian.  Thus, by
        lemma~\ref{lem:subGaussianSquared}, we have that for $\lvert \theta
        \rvert \leq (160e\sigma_X^2)^{-1}$:
        \begin{align*}
        \E \exp\left\{\left(\left\langle
                        \widehat{\bX}_i, \be_a\right\rangle +
        \left\langle \bX_i - \widehat{\bX}_i, \widetilde{\bbe}_0\right\rangle\right)^2 - \E \left(\left\langle
                        \widehat{\bX}_i, \be_a\right\rangle +
        \left\langle \bX_i - \widehat{\bX}_i,
\widetilde{\bbe}_0\right\rangle\right)^2\right\} \leq e^{128(10e)^2\theta^2\sigma_X^4}.
\end{align*}
This implies that by lemma~\ref{lem:subexpconcentration}, taking $\widebar{\theta} = (160e\sigma_X^2)^{-1}, \sigma_Z^2 = (160e)^2\sigma_X^4$ and $t = 160\sqrt{2}eA\sigma_X^2\sqrt{\frac{\log{p}}{n}}$, that for $\sqrt{\frac{\log{p}}{n}} \leq \frac{1}{A}$, $\pr\left\{\lvert I. \rvert
            \geq 160eA\sigma_X^2\sqrt{\frac{\log{p}}{n}}\right\} \leq 2\exp\{-\frac{A^2}{2}\log{p}\}$.  Similarly, $\left \langle \widehat{\bX}_i, \be_a \right\rangle$ is $\sigma_X^2$ sub-gaussian, and lemma~\ref{lem:subGaussianSquared} implies that for $\lvert \theta \rvert \leq (8e\sigma_X^2)^{-1}$:
            \begin{align*}
        \E \exp\left\{\left(\left\langle
                        \widehat{\bX}_i, \be_a\right\rangle\right)^2 - \E \left(\left\langle
                        \widehat{\bX}_i, \be_a\right\rangle\right)^2\right\} \leq e^{128e^2\theta^2\sigma_X^4}.
\end{align*}
Thus, taking $\widebar{\theta} = (16e\sigma_X^2)^{-1}, \sigma_Z^2 = 256e^2\sigma_X^4$, and $t = 16eA\sigma_X^2\sqrt{\frac{\log{p}}{n}}$, and assuming $\sqrt{\frac{\log{p}}{n}} \leq \frac{1}{A}$, lemma~\ref{lem:subexpconcentration} yields $\pr\left\{\left\lvert II. \right\rvert \geq 160eA\sigma_X^2\sqrt{\frac{\log{p}}{n}}\right\} \leq \pr\left\{\left\lvert II. \right\rvert \geq 16eA\sigma_X^2\sqrt{\frac{\log{p}}{n}}\right\} \leq 2e^{-\frac{A^2}{2}\log{p}}$.  We additionally see that by fact~\ref{fact:sumsubgaussian}, $\left\langle \bX_i - \widehat{\bX}_i, \widetilde{\bbe}_0\right \rangle$ is $4\sigma_X^2$ sub-gaussian and thus lemma~\ref{lem:subGaussianSquared} implies that for $\lvert \theta \rvert \leq (64e\sigma_X^2)^{-1}$,
\begin{align*}
        \E \exp\left\{\left(
        \left\langle \bX_i - \widehat{\bX}_i, \widetilde{\bbe}_0\right\rangle\right)^2 - \E \left(
        \left\langle \bX_i - \widehat{\bX}_i,
\widetilde{\bbe}_0\right\rangle\right)^2\right\} \leq e^{128(4e)^2\theta^2\sigma_X^4}.
\end{align*}
Thus, taking $\widebar{\theta} = (64e\sigma_X^2)^{-1}, \sigma_Z^2 = (64e)^2\sigma_X^4$ and $t = 64eA\sigma_X^2\sqrt{\frac{\log{p}}{n}}$, and assuming $\sqrt{\frac{\log{p}}{n}} \leq \frac{1}{A}$, lemma~\ref{lem:subexpconcentration} implies $\pr\left\{\left\lvert III. \right \rvert \geq 160eA\sigma_X^2\sqrt{\frac{\log{p}}{n}}\right\} \leq \pr\left\{\left\lvert III. \right \rvert \geq 64eA\sigma_X^2\sqrt{\frac{\log{p}}{n}}\right\} \leq 2e^{-\frac{A^2}{2} \log{p}}$.  We are ready to conclude.  Under the most stringent assumption $\sqrt{\frac{\log{p}}{n}} \leq \frac{1}{A}$, 
\[\sum_{I \in
        \{I., II., III.\}}\pr\left\{\lvert I \rvert \geq 160eA\sigma_X^2\sqrt{\frac{\log{p}}{n}}\right\} \leq 6p^{-\frac{A^2}{2}}.\]
We thus have:
\begin{align*}
\pr\left\{\norm{\frac{1}{n}\widehat{\bX}^T\left(\widehat{\bX} -
                \bX\right)\bbe_0}_{\infty} \geq \frac{\lambda}{4}\right\} &\leq 6p^{1 - \frac{A^2}{2}},
\end{align*}
and our assumption that $A > \sqrt{2}$ yields the result.
\end{proof}

\subsection{Proof of lemma~\ref{lem:term2mnar}}
\label{subsec:term2mnar}
We repeat the lemma here for the reader's convenience:
\begin{customlemma}{A.9}
Assuming $\lambda$ as defined in equation~\eqref{eq:lambdamnar}, and the assumptions of corollary~\ref{thm:mnar}, we have:
\[
\pr\left\{\norm{\frac{1}{n}\widehat{\bX}^T \beps}_{\infty} \geq \frac{\lambda}{4}\right\} \leq 2p^{1 - \frac{A^2}{2}}.
\]
\end{customlemma}
\begin{proof}
This is exactly the same as lemma~\ref{lem:term2thm1}.
\end{proof}

\section{Proofs of technical lemmas from appendix~\ref{subsec:proofsAR}}
\label{sec:appendixARtechnicallemmas}
This appendix is dedicated to the proofs needed for theorem~\ref{thm:ARres}. It is organized as follows: subsection~\ref{subsec:usefulAR} collects facts that will be used in each proof in lemma~\ref{lem:usefulARlemmas}.  Then, subsection~\ref{subsec:proofARapprox} contains the proof of lemma~\ref{lem:ARapprox} which controls the term $\lvert \widehat{\phi} - \phi \rvert$.  Subsections~\ref{subsec:term1thmAR} - ~\ref{subsec:term6thmAR} contain the auxiliary lemmas necessary for each of the terms in theorem~\ref{thm:ARres}.  Finally, the assumptions of the theorem are addressed in subsection~\ref{subsec:verifylambdamin}.

\subsection{Useful facts for AR proofs}
\label{subsec:usefulAR}
Throughout this section, we will write the explicit forms $\widehat{X}_{ik}, \widetilde{X}_{ik}$ more succinctly, taking $f_L(\phi) = \frac{\phi^{d_1 + d_2}}{1 - \phi^{2(d_1 + d_2)}}\left(\phi^{-d_2} - \phi^{d_2}\right)$ and $f_R(\phi) = \frac{\phi^{d_1 + d_2}}{1 - \phi^{2(d_1 + d_2)}}\left(\phi^{-d_1} - \phi^{d_1}\right)$.  By the mean value theorem, we write 
\begin{align}
\label{eq:ARwritedifference}
\widetilde{X}_{ik} - \widehat{X}_{ik} = \left(\widehat{\phi} - \phi\right)\left(X_{i, L(k)}f_{L(k)}'(\xi_{L(k)}) + X_{i, R(k)}f_{R(k)}'(\xi_{R(k)})\right),
\end{align} 
for some $\xi_{L(k)}, \xi_{R(k)} \in (\min(\phi, \widehat{\phi}), \max(\phi, \widehat{\phi}))$.  We will make use many times of the following lemma, which collects various facts which will be useful for the proof:
\begin{lemma} Under the assumptions of theorem~\ref{thm:ARres}:
\label{lem:usefulARlemmas}
\begin{itemize}
\item[i.] For all $b \in [p]$, $\lvert f_{L(b)}'(\xi_{L(b))})\rvert, \lvert f_{R(b)}'(\xi_{R(b)})\rvert \leq C$, where $C$ is a universal constant.
\item[ii.] $\widehat{X}_{ia}, X_{i, L(a)}, X_{i, R(a)}$ are all $\sigma_X^2$ sub-gaussian.
\item[iii.] $\pr\left\{\max_{a,b \in [p] \times [p]} \frac{1}{n}\sum_{i=1}^{n}\lvert \widehat{X}_{ia}X_{i, L(b)}\rvert \geq 18e\sigma_X^2\right\} \leq 2p^{2 - \frac{A^2}{2}}$.
\end{itemize}
\end{lemma}
\begin{proof}

\begin{itemize}
\item[i.]  This follows from a straightforward, but tedious calculation and
    using the fact that $\lvert \phi \rvert < 1$.
\item[ii.] We see that $\widehat{X}_{ia}$ is sub-gaussian by using fact~\ref{fact:subGaussianConditionalExpectation} for vector $\widehat{\bX}_i$ and then using the fact that $\widehat{X}_{ia} = \left \langle \widehat{\bX}_i, \be_a\right\rangle$.  The second two statements follow by noticing that $\E e^{\theta X_{i, L(a)}} = \E_{L(a)}\E\left\{e^{\theta X_{i, L(a)}} \mid L(a)\right\}$ and using sub-gaussianity of each entry of $\bX$.  
\item[iii.]  Part ii. in combination with lemma~\ref{lem:absvalprodsubgaussian} imply that for any $\lvert \theta \rvert \leq (16e\sigma_X^2)^{-1}$, $\lvert \widehat{X}_{ia}X_{i, L(b)}\rvert$ satisfies $\E \exp\left\{\theta \left(\lvert \widehat{X}_{ia}X_{i, L(b)}\rvert - \E \lvert \widehat{X}_{ia}X_{i, L(b)}\rvert\right)\right\} \leq \E\exp\left\{128e^2\theta^2\sigma_X^4\right\}$.  Thus, Bernstein's inequality~\ref{lem:bernstein} with $\widebar{\theta} = (16e\sigma_X^2)^{-1}$ and $\sigma_Z^2 = 256e^2\sigma_X^4$ implies:
\[
\pr\left\{\left \lvert \frac{1}{n}\sum_{i=1}^{n}\left(\lvert \widehat{X}_{ia}X_{i, L(b)}\rvert - \E \lvert \widehat{X}_{ia}X_{i, L(b)}\rvert\right)\right \rvert \geq t\right\} \leq 2\exp\left\{\frac{n}{2}\min\left(\frac{t^2}{(16e\sigma_X^2)^2}, \frac{t}{16e\sigma_X^2}\right)\right\}.
\]
Now, a union bound over $[p] \times [p]$ as well as taking $t = 16eA\sigma_X^2\sqrt{\frac{\log{p}}{n}}$ and recalling the assumption that $A\sqrt{\frac{\log{p}}{n}} < 1$ (which is satisfied for $c_{\ell}$ in assumption \textbf{B2.} large enough) yields:
\[
\pr\left\{\max_{a,b \in [p] \times [p]}\left \lvert \frac{1}{n}\sum_{i=1}^{n}\left(\lvert \widehat{X}_{ia}X_{i, L(b)}\rvert - \E \lvert \widehat{X}_{ia}X_{i, L(b)}\rvert\right)\right \rvert \geq 16eA\sigma_X^2\sqrt{\frac{\log{p}}{n}}\right\} \leq 2p^{2 - \frac{A^2}{2}}.
\]
Noting that by fact~\ref{fact:expectationsubgaussianproduct}, $\E\lvert \widehat{X}_{ia}X_{i, L(b)}\rvert \leq \sigma_X^2$, the result follows.
\end{itemize}
\end{proof}

\subsection{Proof of lemma~\ref{lem:ARapprox}}
\label{subsec:proofARapprox}
We re-state the lemma here:
\begin{customlemma}{B.1}
Define $\hat{\phi}$ as in ~\eqref{eq:ARParamApprox}.  Then under the assumptions of theorem~\ref{thm:ARres},
\[\pr\left\{\left\lvert \hat{\phi} - \phi\right\rvert \geq \frac{4}{\alpha^2}\sqrt{\frac{\log{p}}{np}}\right\} \leq c p^{-C}.\]
where $c, C > 0$ are universal constants.
\end{customlemma}
\begin{proof}
It is helpful to recall the definition of $\widehat{\phi}$ from ~\eqref{eq:ARParamApprox}:
\begin{align*}
\hat{\phi} =
\frac{\frac{1}{\alpha^2 np}\sum_{i=1}^{n}\sum_{a=1}^{p-1}X_{ia}X_{i(a+1)}M_{ia}M_{i(a+1)}}{\frac{1}{\alpha n
p}\sum_{i=1}^{n}\sum_{a=1}^{p-1}X_{ia}^2 M_{ia}}.
\end{align*}
 Let us first examine the numerator: $\frac{1}{\alpha^2 np}\sum_{i=1}^{n}\sum_{a=1}^{p-1}X_{ia}X_{i(a+1)}M_{ia}M_{i(a+1)}$.  We
    have:
    \begin{align*}
        &\pr\left\{\left\lvert \frac{1}{n(p-1)}\sum_{i=1}^{n}\sum_{a=1}^{p-1} X_{ia}X_{i(a+1)}M_{ia} M_{i(a-1)} -
        \frac{\alpha^2 \phi}{1 - \phi^2} \right\rvert \geq t\right\} \\
        &= \sum_{\bM}\pr\left\{\left\lvert \frac{1}{n(p-1)}\sum_{i=1}^{n}\sum_{a=1}^{p-1} X_{ia}X_{i(a+1)}M_{ia} M_{i(a+1)} -
        \frac{\alpha^2 \phi}{1 - \phi^2} \right\rvert \geq t \mathrel{\Bigg|}\bM\right\}\pr\left\{\bM\right\}
    \end{align*}
Let $N_M \equiv \sum_{i,a}M_{ia} M_{i(a+1)}$ and notice that $\E
    \left\{\frac{1}{n(p-1)}\sum_{i=1}^{n}\sum_{a=1}^{p-1}
    X_{ia}X_{i(a+1)}M_{ia} M_{i(a+1)}\mathrel{\Bigg|}\bM\right\} =
    \frac{N_M \phi}{n(p-1)(1 - \phi^2)}$. By an application of the triangle
    inequality, we see that:
    \begin{align*}
        &\pr\left\{\left\lvert \frac{1}{n(p+1)}\sum_{i=1}^{n}\sum_{a=1}^{p-1} X_{ia}X_{i(a+1)}M_{ia} M_{i(a+1)} -
        \frac{\alpha^2 \phi}{1 - \phi^2} \right\rvert \geq
        t\mathrel{\Bigg|}\bM\right\} \\
        &\leq \pr\left\{\left\lvert \frac{1}{n(p-1)}\sum_{i=1}^{n}\sum_{a=1}^{p-1} X_{ia}X_{i(a+1)}M_{ia} M_{i(a+1)} -
        \frac{N_M \phi}{n(p-1)(1 - \phi^2)} \right\rvert \geq
        t - \left\lvert\frac{\phi}{1-\phi^2} \left(\alpha^2 -
        \frac{N_M}{n(p-1)}\right)\right\rvert\mathrel{\Bigg|}\bM\right\} 
    \end{align*}
 Now, consider flattening $\bX$ into the vector $\underline{\bX}$ which has
    distribution $N\left(0, \underline{\bSig}\right)$, where we define:
    \[\underline{\bSig} \equiv 
    \begin{bmatrix}
        \bSig & 0 & \dots & 0 \\
        0 & \bSig & \dots & 0 \\
        \vdots & \vdots & \dots & \vdots \\
        0 & 0 & \dots & \bSig
    \end{bmatrix}\]
    This implies that:
    \[
        \frac{1}{n(p-1)}\sum_{i,a} X_{ia}X_{i(a+1)} M_{ia}M_{i(a+1)}
        \stackrel{d}{=} \frac{1}{n(p-1)} \left\langle \bg,
        \underline{\bSig}^{\frac{1}{2}}\underline{\bA}\underline{\bSig}^{\frac{1}{2}}\bg\right\rangle 
    \]
    where $\bg \sim N(0, \bI_{np})$, $\bA \in \mathbb{R}^{np \times np}$ and
    \[
        \underline{\bA} =
        \begin{bmatrix}
        \bA_1 & 0 & \dots & 0 \\
        0 & \bA_2 & \dots & 0 \\
        \vdots & \vdots & \dots & \vdots \\
        0 & 0 & \dots & \bA_n
    \end{bmatrix}
    \]
    and $\left(\bA_{i}\right)_{jk} = \mathbbm{1}\left\{j = k - 1, \quad
    \bM_{ij}\bM_{ik} = 1\right\}$.  We will make use of the following
    Hanson-Wright inequality (see ~\cite{vershynin2018high} Theorem 6.2.1):
    \begin{align*}\pr\left\{\frac{1}{n(p-1)}\left\lvert \left\langle\bg,
        \underline{\bSig}^{\frac{1}{2}}\underline{\bA}\underline{\bSig}^{\frac{1}{2}}\bg\right\rangle
        - \E \left\langle\bg,
        \underline{\bSig}^{\frac{1}{2}}\underline{\bA}\underline{\bSig}^{\frac{1}{2}}\bg\right\rangle
        \right\rvert \geq t\right\} &\leq
        2\exp\left\{-c\min\left\{\frac{(tn(p-1))^2}{\norm{\underline{\bSig}^{\frac{1}{2}}\underline{\bA}\underline{\bSig}^{\frac{1}{2}}}_{F}^2},
        \frac{tn(p-1)}{\norm{\underline{\bSig}^{\frac{1}{2}}\underline{\bA}\underline{\bSig}^{\frac{1}{2}}}_{op}}\right\}\right\}
    \end{align*}
    We now upper bound $\norm{\underline{\bSig}^{\frac{1}{2}}\underline{\bA}\underline{\bSig}^{\frac{1}{2}}}_{F}^2$ and $\norm{\underline{\bSig}^{\frac{1}{2}}\underline{\bA}\underline{\bSig}^{\frac{1}{2}}}_{op}$.  We begin with the former:
    \[\norm{\underline{\bSig}^{\frac{1}{2}}\underline{\bA}\underline{\bSig}^{\frac{1}{2}}}_{F}^2
    =
    \sum_{i=1}^{n}\norm{\bSig^{\frac{1}{2}}\bA_i\bSig^{\frac{1}{2}}}_{F}^2 \leq
    n\norm{\bSig^{\frac{1}{2}}\bA\bSig^{\frac{1}{2}}}_{F}^2,\]
    where $\left(\bA\right)_{jk} = \mathbbm{1}\left\{j = k - 1\right\}$.  Noting
    the covariance structure of $\bSig$, we compute:
    \begin{align*}
        n\norm{\bSig^{\frac{1}{2}}\bA\bSig^{\frac{1}{2}}}_{F}^2 &=
        n\text{Tr}\left(\bSig^{\frac{1}{2}}\bA^{T}\bSig^{\frac{1}{2}}\bSig^{\frac{1}{2}}\bA\bSig^{\frac{1}{2}}\right)
        = n\text{Tr}\left(\bA\bSig\bA^T\bSig\right)
    \end{align*}
    Computing this trace gives:
    \begin{align*}
     n\text{Tr}\left(\bA\bSig\bA^T\bSig\right) &=n\sum_{a=2}^{p-1}\frac{\phi^2}{(1 - \phi^2)^2} \left(2 + \frac{2 -
    \phi^{2(a-1)} - \phi^{2(p-a)}}{1 - \phi^2}\right) \leq Cn(p-1)
    \end{align*}
    where the inequality is by using the assumption $\phi < 1$.  Now we tackle
    $\norm{\underline{\bSig}^{\frac{1}{2}}\underline{\bA}\underline{\bSig}^{\frac{1}{2}}}_{op}$.
    Sub-multiplicativity of the operator norm yields
    $\norm{\underline{\bSig}^{\frac{1}{2}}\underline{\bA}\underline{\bSig}^{\frac{1}{2}}}_{op}
    \leq \norm{\underline{\bSig}}_{op}\norm{\underline{\bA}}_{op}$.  Note that
    $\norm{\underline{\bA}}_{op} = 1$.  Additionally, since $\underline{\bSig}$
    is block diagonal with $\bSig$ as each of the $n$ blocks,
    $\norm{\underline{\bSig}}_{op} = \norm{\bSig}_{op}$.  Now, since $\bSig$ is
    Toeplitz,~\cite{gray2006toeplitz} Lemma 4.1 implies that $\norm{\bSig}_{op}
    \leq (1 - \phi^2)^{-1}(1 - \phi)^{-1}$ and the concentration inequality becomes:
    \begin{align*}\pr\left\{\frac{1}{n(p-1)}\left\lvert \left\langle\bg,
        \underline{\bSig}^{\frac{1}{2}}\underline{\bA}\underline{\bSig}^{\frac{1}{2}}\bg\right\rangle
        - \E \left\langle\bg,
        \underline{\bSig}^{\frac{1}{2}}\underline{\bA}\underline{\bSig}^{\frac{1}{2}}\bg\right\rangle
        \right\rvert \geq t\right\} &\leq 2\exp\left\{-cn(p-1)\min\left\{t^2, t\right\}\right\}.
   	\end{align*}
   	Now, let $\mathcal{M}_{\epsilon} = \left\{\left \lvert \alpha^2 -
    \frac{N_M}{n(p-1)} \right\rvert \leq \epsilon \right\}$;  that is,
    $\mathcal{M}_{\epsilon}$ denotes the event that $\frac{N_M}{n(p-1)}$
    estimates $\alpha^2$ with at most $\epsilon$ error.  Notice that on this
    event, $t - \left\lvert \frac{\phi}{1 - \phi^2}\left(\alpha^2 -
        \frac{N_M}{n(p-1)}\right)\right \rvert \geq t - \frac{\lvert \phi
    \rvert}{1 - \phi^2}\epsilon$.  Thus, by the Hanson-Wright inequality, we have:
   	\begin{align*}
   	&\sum_{\bM}\pr\left\{\left\lvert \frac{1}{n(p-1)}\sum_{i=1}^{n}\sum_{a=1}^{p-1}X_{ia}X_{i(a+1)}M_{ia}M_{i(a+1)} - \frac{\alpha^2\phi}{1 - \phi^2}\right \rvert \geq t \mathrel{\Bigg|}\bM\right\}\pr\left\{\bM\right\} \\
   	&\leq 2\exp\left\{-cn(p-1)\min\left\{\left(t - \frac{\lvert \phi \rvert}{1 -
            \phi^2}\epsilon\right)^2,
    t - \frac{\lvert \phi \rvert}{1 - \phi^2}\epsilon\right\}\right\} + \pr\left\{\mathcal{M}_{\epsilon}^c\right\}.
   	\end{align*}
   	We now tackle $\pr\left\{\mathcal{M}_{\epsilon}^{C}\right\} =
    \pr\left\{\left\lvert\frac{1}{n(p-1)}\sum_{i=1}^{n}\sum_{a=2}^{p}M_{ia}M_{i(a-1)}
    - \alpha^2\right\rvert \geq \epsilon\right\}$.  Notice that by the bounded differences inequality (e.g. Theorem 2.9.1 in
    ~\cite{vershynin2018high}), we have:
    \[\pr\left\{\left\lvert\frac{1}{n(p-1)}\sum_{i=1}^{n}\sum_{a=2}^{p}M_{ia}M_{i(a-1)}
    - \alpha^2\right\rvert \geq \epsilon\right\} \leq
    2\exp\left\{-\frac{\epsilon^2 n(p-1)}{2}\right\}.\]
    Thus,
    \begin{align*}
&\pr\left\{\left\lvert \frac{1}{n(p-1)}\sum_{i=1}^{n}\sum_{a=1}^{p-1} X_{ia}X_{i(a+1)}M_{ia} M_{i(a-1)} -
        \frac{\alpha^2 \phi}{1 - \phi^2} \right\rvert \geq t\right\} \\
        &\leq 2\exp\left\{-cn(p-1)\min\left\{\left(t - \frac{\lvert \phi
        \rvert}{1 - \phi^2}\epsilon\right)^2, t - \frac{\lvert \phi \rvert}{1 -
        \phi^2}\epsilon\right\}\right\}  + 2\exp\left\{-\frac{\epsilon^2 n(p-1)}{2}\right\}.
    \end{align*}
    Taking $t \asymp \sqrt{\frac{\log{p}}{np}}, \epsilon \asymp t$ gives:
    \begin{align*}
    \pr\left\{\left\lvert \frac{1}{n(p-1)}\sum_{i=1}^{n}\sum_{a=1}^{p-1} X_{ia}X_{i(a+1)}M_{ia} M_{i(a-1)} -
        \frac{\alpha^2 \phi}{1 - \phi^2} \right\rvert \geq \sqrt{\frac{\log{p}}{np}}\right\} \leq 4p^{-c_1}, 
    \end{align*}
    where $c_1 > 0$ is a universal constant.
    We now tackle the denominator $\frac{1}{ n
(p-1)}\sum_{i=1}^{n}\sum_{a=1}^{p-1}X_{ia}^2 M_{ia}$.  We are interested in:
\begin{align*}
    \pr\left\{\left\lvert \frac{1}{n
    (p-1)}\sum_{i=1}^{n}\sum_{a=1}^{p-1}X_{ia}^2 M_{ia} - \frac{\alpha}{1 -
    \phi^2}\right\rvert \geq t\right\}
\end{align*}
To this end, we define the matrix $\bA \in \mathbb{R}^{np \times np}$:
\[ \bA \equiv
\begin{bmatrix}
    \bSig^{\frac{1}{2}}\left(\sum_{a=1}^{p-1}\be_a \be_a^T
    \bM_{1a}\right)\bSig^{\frac{1}{2}} & 0 &\dots & 0 \\
    0 &  \bSig^{\frac{1}{2}}\left(\sum_{a=1}^{p-1}\be_a \be_a^T
    \bM_{2a}\right)\bSig^{\frac{1}{2}} & \dots & 0 \\
    \vdots & \vdots & \cdots & \vdots \\
    0 & 0 & \dots &  \bSig^{\frac{1}{2}}\left(\sum_{a=1}^{p-1}\be_a \be_a^T
    \bM_{na}\right)\bSig^{\frac{1}{2}}
\end{bmatrix}
\]
We thus see that we can write our estimator $\frac{1}{ n
(p-1)}\sum_{i=1}^{n}\sum_{a=1}^{p-1}X_{ia}^2 M_{ia} = \frac{1}{ n
(p-1)}\left\langle \bg, \bA \bg \right\rangle$. Where $\bg \sim N(0, \bI_{np})$.
We can calculate $\norm{\bA}_{F}^2 \leq n
\norm{\bSig^{\frac{1}{2}}\left(\sum_{a=1}^{p}\be_a \be_a^T
    \right)\bSig^{\frac{1}{2}}}_{F}^2 = np\left(\frac{1 - \phi^{2p}}{\left(1 -
\phi^2\right)^3}\right) \leq Cnp$ and $\norm{\bA}_{op} \leq (1 - \phi^2)^{-1}(1
- \phi)^{-1}$.  We thus combine again the Hanson-Wright inequality and bounded differences inequality in the same manner as above to see:
    \begin{align*}
&\pr\left\{\left\lvert \frac{1}{n
    (p-1)}\sum_{i=1}^{n}\sum_{a=1}^{p-1}X_{ia}^2 M_{ia} - \frac{\alpha}{1 -
    \phi^2}\right\rvert \geq t\right\} \\
&\leq 2\exp\left\{-cn(p-1)\min\left\{\left(t - C(\phi)\epsilon\right)^2, t -
C(\phi)\epsilon\right\}\right\} + 2\exp\left\{-2\epsilon^2n(p-1)\right\}. 
    \end{align*}
    Taking $t \asymp \sqrt{\frac{\log{p}}{np}}, \epsilon \asymp t$ gives:
    \begin{align*}
\pr\left\{\left\lvert \frac{1}{n
    (p-1)}\sum_{i=1}^{n}\sum_{a=1}^{p-1}X_{ia}^2 M_{ia} - \frac{\alpha}{1 -
    \phi^2}\right\rvert \geq \sqrt{\frac{\log{p}}{np}}\right\} 
    \leq 4p^{-c_2}. 
    \end{align*}
    where $c_2 > 0$ is a universal constant.  To conclude, we analyze $\left \lvert \frac{\frac{1}{\alpha^2 np}\sum_{i=1}^{n}\sum_{a=1}^{p-1}X_{ia}X_{i(a+1)}M_{ia}M_{i(a+1)}}{\frac{1}{\alpha n
p}\sum_{i=1}^{n}\sum_{a=1}^{p-1}X_{ia}^2 M_{ia}} - \phi \right \rvert$ noting that with probability at least $1 - c p^{-C}$,  $\left \lvert \frac{1}{\alpha^2 np}\sum_{i=1}^{n}\sum_{a=1}^{p-1}X_{ia}X_{i(a+1)}M_{ia}M_{i(a+1)} - \frac{\phi}{1 - \phi^2}\right \rvert \leq \frac{1}{\alpha^2}\sqrt{\frac{\log{p}}{np}}$ and $\left\lvert \frac{1}{\alpha n
p}\sum_{i=1}^{n}\sum_{a=1}^{p-1}X_{ia}^2 M_{ia} - \frac{1}{1 - \phi^2}\right\rvert \leq \frac{1}{\alpha}\sqrt{\frac{\log{p}}{np}}$.  Using these two facts as well as the assumption $\sqrt{\frac{\log{p}}{np}} \leq \frac{\alpha}{2}$ (which holds for $c_{\ell}$ large enough) implies the result.
\end{proof}

\subsection{Proof of lemma~\ref{lem:term1thmAR}}
\label{subsec:term1thmAR}
We re-state the lemma for the reader's convenience:
\begin{customlemma}{B.2}
Under the assumptions of theorem~\ref{thm:ARres}, for any $\bu \in \mathbb{R}^{p}$, we have:
\[
\pr\left\{\norm{\frac{1}{n}\widehat{\bX}^T\left(\widetilde{\bX} - \widehat{\bX}\right)\bu}_{\infty} \geq \frac{C\sigma_X^2}{\alpha^2}\norm{\bu}_1\sqrt{\frac{\log{p}}{np}}\right\} \leq c_0p^{-c_1},
\]
where $C, c_0, c_1$ denote universal constants.
\end{customlemma}

\begin{proof}
Let us begin by writing:
\begin{align*}
\norm{\frac{1}{n}\widehat{\bX}^T\left(\widetilde{\bX} - \widehat{\bX}\right)\bu}_{\infty} = \max_{a \in [p]}\left\lvert\frac{1}{n}
\sum_{i=1}^{n}\sum_{b=1}^{p}\widehat{X}_{ia}\left(\widetilde{X}_{ib}
-\widehat{X}_{ib}\right)u_b\right\rvert.
\end{align*}
Using~\eqref{eq:ARwritedifference}, we write:
\begin{align*}
\max_{a \in [p]}\left\lvert\frac{1}{n}
\sum_{i=1}^{n}\sum_{b=1}^{p}\widehat{X}_{ia}\left(\widetilde{X}_{ib}
-\widehat{X}_{ib}\right)u_b\right\rvert &= \max_{a \in [p]}\left\lvert \frac{1}{n}\sum_{i, b} \widehat{X}_{ia}(\widehat{\phi} - \phi)\left(X_{i, L(b)}f_{L(b)}'(\xi_{L(b)}) + X_{i, R(b)}f_{R(b)}'(\xi_{R(b)})\right)u_b\right \rvert.
\end{align*}
The triangular inequality then implies that the RHS is upper bounded by:
\begin{align}
\label{eq:ARterm1fundamentalineq}
\lvert \widehat{\phi} - \phi \rvert \max_{a \in [p]} \sum_{b=1}^{p}\lvert u_b\rvert \frac{1}{n}\sum_{i=1}^{n} \left(\lvert \widehat{X}_{i, a}X_{i, L(b)}f_{L(b)}'(\xi_{L(b)}) \rvert + \lvert \widehat{X}_{ia}X_{i, R(b)}f_{R(b)}'(\xi_{R(b)})\rvert\right).
\end{align}
Lemma~\ref{lem:usefulARlemmas} i. implies that~\eqref{eq:ARterm1fundamentalineq} is upper bounded by:
\[
 C \lvert \widehat{\phi} - \phi \rvert \norm{\bu}_1 \max_{a, b \in [p] \times [p]} \frac{1}{n}\sum_{i=1}^{n}\left(\lvert \widehat{X}_{ia}X_{i, L(b)}\rvert + \lvert \widehat{X}_{ia}X_{i, R(b)}\rvert\right).
\]
We are thus led to study:
\[
\pr\left\{\max_{a, b \in [p] \times [p]} \frac{1}{n}\sum_{i=1}^{n}\left(\lvert \widehat{X}_{ia}X_{i, L(b)}\rvert + \lvert \widehat{X}_{ia}X_{i, R(b)}\rvert\right) \geq t\right\} \leq \pr\left\{\max_{a,b \in [p] \times [p]}\frac{1}{n}\sum_{i=1}^{n} \lvert \widehat{X}_{ia}X_{i, L(b)}\rvert  \geq \frac{t}{2}\right\},
\]
where the inequality follows by an application of the union bound, and inequality $\pr\left\{A + B \geq t\right\} \leq 2 \pr\left\{A \geq \frac{t}{2}\right\}$ for identitically distributed (but not necessarily independent) non-negative random variables $A$ and $B$.  Let $\mathcal{A}_1$ denote the event $\left\{\max_{a,b \in [p] \times [p]}\frac{1}{n}\sum_{i=1}^{n} \lvert \widehat{X}_{ia}X_{i, L(b)}\rvert  \geq 18e\sigma_X^2\right\}$ and notice that lemma~\ref{lem:usefulARlemmas} iii. implies $\pr\{\mathcal{A}_1\} \geq 1 - 2p^{2 - \frac{A^2}{2}}$ (that is we take $t = 36e\sigma_X^2$).  Additionally, let $\mathcal{A}_2$ denote the event that $\lvert \widehat{\phi} - \phi\rvert \leq \frac{4}{\alpha^2}\sqrt{\frac{\log{p}}{np}}$ and notice that by lemma~\ref{lem:ARapprox}, $\pr\{\mathcal{A}_2\} \geq 1 - cp^{-C}$.  The result follows immediately.
\end{proof}

\subsection{Proof of lemma~\ref{lem:term2thmAR}}
\label{subsec:term2thmAR}
We re-state the lemma for the reader's convenience:
\begin{customlemma}{B.4}
Under the assumptions of theorem~\ref{thm:ARres}, we have:
\[
\pr\left\{\norm{\frac{1}{n}\left(\widetilde{\bX} - \widehat{\bX}\right)^T\left(\widetilde{\bX} - \widehat{\bX}\right)\bu}_{\infty} \geq C\frac{\sigma_X^2}{\alpha^4}\norm{\bu}_1\frac{\log{p}}{np}\right\} \leq c_0p^{-c_1},
\]
where $C, c_0, c_1$ denote universal constants.
\end{customlemma}
\begin{proof}
We begin by writing:
\[
\norm{\frac{1}{n}\left(\widetilde{\bX} - \widehat{\bX}\right)^T\left(\widetilde{\bX} - \widehat{\bX}\right)\bu}_{\infty} = \max_{a \in [p]}\left\lvert\frac{1}{n}\sum_{i=1}^{n}\sum_{b=1}^{p}\left(\widetilde{X}_{ia} - \widehat{X}_{ia}\right)\left(\widetilde{X}_{ib} - \widehat{X}_{ib}\right)u_b\right\rvert.
\]
Using ~\eqref{eq:ARwritedifference}, this can be written as:
\[
\max_{a \in [p]}\left\lvert \widehat{\phi} - \phi \right \rvert^2\left\lvert\frac{1}{n}\sum_{i=1}^{n}\sum_{b=1}^{p}\prod_{k \in \{a,b\}}\left(X_{i, L(k)}f_L'(\xi_{L(k)}) + X_{i, R(k)}f_R'(\xi_{R(k)})\right)u_b\right\rvert.
\]
The triangular inequality, expanding the product and using lemma~\ref{lem:usefulARlemmas} i., we can upper bound this quantity by:
\begin{align*}
&C\max_{a \in [p]}\left\lvert \widehat{\phi} - \phi \right \rvert^2\sum_{b=1}^{p}\lvert u_b \rvert \frac{1}{n}\sum_{i=1}^{n}\left(\lvert X_{i, L(a)}X_{i, L(b)} \rvert + \lvert X_{i, L(a)}X_{i, R(b)} \rvert + \lvert X_{i, R(a)} X_{i, L(b)} \rvert + \lvert X_{i, R(a)} X_{i, R(b)} \rvert\right) \\
&\leq C\left\lvert \widehat{\phi} - \phi \right \rvert^2\norm{\bu}_1\max_{a,b \in [p] \times [p]} \frac{1}{n}\sum_{i=1}^{n}\left(\lvert X_{i, L(a)}X_{i, L(b)} \rvert + \lvert X_{i, L(a)}X_{i, R(b)} \rvert + \lvert X_{i, R(a)} X_{i, L(b)} \rvert + \lvert X_{i, R(a)} X_{i, R(b)} \rvert\right).
\end{align*}
Just as in the proof of lemma~\ref{lem:term1thmAR}, we are led to study:
\[
\pr\left\{\max_{a,b \in [p] \times [p]}\frac{1}{n}\sum_{i=1}^{n} \lvert \widehat{X}_{ia}X_{i, L(b)}\rvert  \geq \frac{t}{4}\right\}.
\]
Lemma~\ref{lem:usefulARlemmas} iii. with $t=72e\sigma_X^2$ implies that :
\[\max_{a,b \in [p] \times [p]} \frac{1}{n}\sum_{i=1}^{n}\left(\lvert X_{i, L(a)}X_{i, L(b)} \rvert + \lvert X_{i, L(a)}X_{i, R(b)} \rvert + \lvert X_{i, R(a)} X_{i, L(b)} \rvert + \lvert X_{i, R(a)} X_{i, R(b)} \rvert\right) \leq 72e\sigma_X^2,\]
with probability at least $1 - p^{2 - \frac{A^2}{2}}$.  Additionally, lemma~\ref{lem:ARapprox} implies that $\left\lvert \widehat{\phi} - \phi \right \rvert ^2 \leq \frac{16}{\alpha^4}\frac{\log{p}}{np}$ with probability at least $1 - c p^{-C}$.  The result follows immediately.
\end{proof}

\subsection{Proof of lemma~\ref{lem:term3thmAR}}
\label{subsec:term3thmAR}
We re-state the lemma here for ease of reading:
\begin{customlemma}{B.5}
Under the assumptions of theorem~\ref{thm:ARres}, for any $\bu \in \mathbb{R}^{p}$, we have:
\[
\pr\left\{\norm{\frac{1}{n}\left(\widetilde{\bX} - \widehat{\bX}\right)^T\left(\widehat{\bX} - \bX\right)\bu}_{\infty} \geq C\frac{\sigma_X^2}{\alpha^2}\norm{\bu}_1 \sqrt{\frac{\log{p}}{np}}\right\} \leq c_0p^{-c_1},
\]
where $C, c_0, c_1$ denote universal constants.
\end{customlemma}
\begin{proof}
We begin by writing:
\[
\norm{\frac{1}{n}\left(\widetilde{\bX} - \widehat{\bX}\right)^T\left(\widehat{\bX} - \bX\right)\bu}_{\infty} = \max_{a \in [p]}\left\lvert \frac{1}{n}\sum_{i=1}^{n}\sum_{b=1}^{p}\left(\widetilde{X}_{ia} - \widehat{X}_{ia}\right)\left(\widehat{X}_{ib} - X_{ib}\right)u_b \right \rvert.
\]
The triangular inequality, ~\eqref{eq:ARwritedifference}, and lemma~\ref{lem:usefulARlemmas} i. imply this is upper bounded by:
\[
C\left\lvert\widehat{\phi} - \phi \right\rvert \norm{\bu}_1 \max_{a, b \in [p] \times [p]}\frac{1}{n}\sum_{i=1}^{n}\left(\lvert X_{i, L(a)}\widehat{X}_{ib} \rvert + \lvert X_{i, L(a)}X_{ib} \rvert + \lvert X_{i, R(a)} \widehat{X}_{ib} \rvert + \lvert X_{i, R(a)} X_{ib} \rvert \right).
\]
Lemma~\ref{lem:usefulARlemmas} iii. with $t=72e\sigma_X^2$ implies 
\[
\max_{a, b \in [p] \times [p]}\frac{1}{n}\sum_{i=1}^{n}\left(\lvert X_{i, L(a)}\widehat{X}_{ib} \rvert + \lvert X_{i, L(a)}X_{ib} \rvert + \lvert X_{i, R(a)} \widehat{X}_{ib} \rvert + \lvert X_{i, R(a)} X_{ib} \rvert \right) \leq 72e\sigma_X^2,
\]
with probability at least $1 - p^{2 - \frac{A^2}{2}}$ and lemma~\ref{lem:ARapprox} implies that $\lvert \widehat{\phi} - \phi \rvert \leq \frac{4}{\alpha^2}\sqrt{\frac{\log{p}}{np}}$ with probability at least $1 - cp^{-C}$.  The result follows immediately.
\end{proof}

\subsection{Proof of lemma~\ref{lem:term4thmAR}}
\label{subsec:term4thmAR}
We re-state the lemma here for ease of reading:
\begin{customlemma}{B.5}
Assuming $\lambda$ as defined in equation~\eqref{eq:lambdaAR}, and the assumptions of theorem~\ref{thm:ARres}, we have, we have:
\[
\pr\left\{\norm{\frac{1}{n}\widehat{\bX}^T\left(\widehat{\bX} - \bX\right)\bbe_0}_{\infty} \leq \frac{\lambda}{16}\right\} \geq 1 - c_0p^{-c_1},
\]
where $c_0, c_1$ denote universal constants.
\end{customlemma}
\begin{proof}
This follows immediately from lemma~\ref{lem:term1thm1}.
\end{proof}

\subsection{Proof of lemma~\ref{lem:term5thmAR}}
\label{subsec:term5thmAR}
We repeat the lemma here:
\begin{customlemma}{B.7}
Assuming $\lambda$ as defined in equation~\eqref{eq:lambdaAR}, and the assumptions of theorem~\ref{thm:ARres}, we have:
\[
\pr\left\{\norm{\frac{1}{n}\left(\widetilde{\bX} - \widehat{\bX}\right)^T\beps}_{\infty} \geq C\frac{\sigma_X\sigma}{\alpha^2}\sqrt{\frac{\log{p}}{n}}\right\} \leq c_0p^{-c_1},
\]
where $C, c_0, c_1$ are universal constants.
\end{customlemma}
\begin{proof}
We will write:
\[
\norm{\frac{1}{n}\left(\widetilde{\bX} - \widehat{\bX}\right)^T\beps}_{\infty} = \max_{a \in [p]}\left\lvert \frac{1}{n} \sum_{i=1}^{n} \left(\widetilde{X}_{ia} - \widehat{X}_{ia}\right)\epsilon_i\right \rvert.
\]
Using equation~\eqref{eq:ARwritedifference}, the triangular inequality, and lemma~\ref{lem:usefulARlemmas} i., this quantity is upper bounded by:
\[
C \left\lvert \widehat{\phi} - \phi \right \rvert \max_{a \in
[p]}\frac{1}{n}\sum_{i=1}^{n}\left(\lvert X_{i, L(a)}\epsilon_i \rvert + \lvert
X_{i, R(a)}\epsilon_i\rvert\right).
\]
We are interested in 
\[
\label{eq:ARepsilonmiddlestep}
\pr\left\{\max_{a \in
[p]}\frac{1}{n}\sum_{i=1}^{n}\left(\lvert X_{i, L(a)}\epsilon_i \rvert + \lvert
X_{i, R(a)}\epsilon_i\rvert\right) \geq t\right\} \leq 2\pr\left\{\max_{a \in
[p]}\frac{1}{n}\sum_{i=1}^{n}\lvert X_{i, L(a)}\epsilon_i \rvert \geq
\frac{t}{2}\right\}.
\]
We will require a concentration inequality similar to
lemma~\ref{lem:usefulARlemmas} iii. Lemma~\ref{lem:usefulARlemmas} ii. implies that $X_{i, L(a)}$ is $\sigma_X^2$
sub-gaussian and by assumption $\epsilon_i$ is $\sigma^2$ sub-gaussian.  Thus,
lemma~\ref{lem:absvalprodsubgaussian} implies that for any $\lvert \theta \rvert
\leq (16e\sigma_X \sigma)^{-1}$, $\lvert X_{i, L(a)} \epsilon_i\rvert$ satisfies
$\E\exp\left\{\theta\left(\lvert X_{i, L(a)} \epsilon_i\rvert - \E \lvert X_{i,
    L(a)} \epsilon_i\rvert\right)\right\} \leq \exp\left\{128e\theta^2
    \sigma_X^2 \sigma^2\right\}$.  Thus, Bernstein's
    inequality~\ref{lem:bernstein} with $\widebar{\theta} =
    (16e\sigma_X\sigma)^{-1}$ and $\sigma_Z^2 = 256e^2\sigma_X^2 \sigma^2$
    implies:
    \[
    \pr\left\{\left\lvert \frac{1}{n}\sum_{i=1}^{n} \left( \lvert X_{i,
                L(a)}\epsilon_i \rvert - \E \lvert X_{i, L(a)} \epsilon_i \rvert
        \right) \right \rvert \geq t \right\} \leq
        2\exp\left\{-\frac{n}{2}\min\left(\frac{t^2}{(16e\sigma_X\sigma)^2},
            \frac{t}{16e\sigma_X\sigma}\right)\right\}.
    \]
    Taking $t = 16eA\sigma_X\sigma \sqrt{\frac{\log{p}}{n}}$, a union bound over
    $[p]$ and recalling the assumption that $A\sqrt{\frac{\log{p}}{n}} < 1$
    yields:
    \[
        \pr\left\{\max_{a \in [p]}\left\lvert \frac{1}{n}\sum_{i=1}^{n} \left( \lvert X_{i,
                L(a)}\epsilon_i \rvert - \E \lvert X_{i, L(a)} \epsilon_i \rvert
    \right) \right \rvert \geq 16eA\sigma_X\sigma\sqrt{\frac{\log{p}}{n}}
\right\} \leq 2p^{1 - \frac{A^2}{2}}.
    \]
    Noting the assumption that $A\sqrt{\frac{\log{p}}{n}} < 1$ and fact~\ref{fact:expectationsubgaussianproduct}, this implies:
    \[
    \pr\left\{\max_{a \in [p]} \frac{1}{n}\sum_{i=1}^{n}\lvert X_{i, L(a)}\epsilon_i \rvert \geq 18eA\sigma_X\sigma\right\} \leq 2p^{1 - \frac{A^2}{2}}.
    \] 
    Therefore, we take $t = 36eA\sigma_X\sigma$ in ~\eqref{eq:ARepsilonmiddlestep}.  Additionally, we recall that lemma~\ref{lem:ARapprox} implies $\lvert \widehat{\phi} - \phi \rvert \leq \frac{4}{\alpha^2}\sqrt{\frac{\log{p}}{np}}$ with probability at least $1 - cp^{-C}$.  Putting these together implies that with probability at least $1 - cp^{-C}$:
    \[
    C \left \lvert \widehat{\phi} - \phi \right \rvert \max_{a \in [p]}\frac{1}{n}\sum_{i=1}^{n} \left(\lvert X_{i, L(a)}\epsilon_i \rvert + \lvert X_{i, R(a)}\epsilon_i\right) \leq C\frac{\sigma_X\sigma}{\alpha^2}\sqrt{\frac{\log{p}}{np}}.
    \]
    The result is immediate.
\end{proof}

\subsection{Proof of lemma~\ref{lem:term6thmAR}}
\label{subsec:term6thmAR}
\begin{customlemma}{B.6}
Assuming $\lambda$ as defined in equation~\eqref{eq:lambdaAR}, and the assumptions of theorem~\ref{thm:ARres}, we have:
\[
\pr\left\{\norm{\frac{1}{n}\widehat{\bX}^T\beps}_{\infty} \leq C\sigma_X\sigma\sqrt{\frac{\log{p}}{n}}\right\} \geq c_0 p^{-c_1},
\]
where $C, c_0, c_1$ are universal constants.
\end{customlemma}
\begin{proof}
This is immediate from lemma~\ref{lem:term2thm1}.
\end{proof}

\subsection{Proof of $\lambda_{\min}(\Sigma_{\widehat{\bX}}) > 0$}
\label{subsec:verifylambdamin}
As mentioned after the assumptions of the main theorem, we need to verify a lower bound on the minimum eigenvalue of the covariance.  We remark that this is included to show a specific case of when the minimum eigenvalue can be lower bounded by a constant.  For this reason, we have carried out the analysis in far from the tightest way possible and with more work, we believe the assumption $\alpha > 0.844$ need not be so restrictive.  We have the following lemma:
\begin{lemma}
\label{lem:verifylambdamin}
Under the assumptions of theorem~\ref{thm:ARres}, $\Sigma_{\widehat{\bX}}$ satisfies $\lambda_{\min}(\Sigma_{\widehat{\bX}}) > 0$.
\end{lemma}
\begin{proof}
Consider a row of $\bX$, let it be $\bx$ and a row of $\widehat{\bX}$ which we call $\widehat{\bx}$.  Note that $\Sigma_{\bX} = \E \bx \bx^T$ and $\Sigma_{\widehat{\bX}} = \E \widehat{\bx} \widehat{\bx}^T$.  Notice:
\[
\E \bx \bx^T = \E \left(\bx - \widehat{\bx} + \widehat{\bx}\right)\left(\bx -\widehat{\bx} + \widehat{\bx}\right)^T = \E \widehat{\bx} \widehat{\bx}^T - \E \left(\bx - \widehat{\bx}\right)\left(\bx - \widehat{\bx}\right)^T.
\]
This implies that $\lambda_{\min}\left(\E \widehat{\bx} \widehat{\bx}^T\right) \geq \lambda_{\min}\left(\E \bx \bx^T \right) - \lambda_{\max}\left(\E \left(\bx - \widehat{\bx}\right)\left(\bx - \widehat{\bx}\right)^T\right)$.  Let $\bx_0$ denote the zero-imputed estimator for the same observation $\bz$ for which $\widehat{\bx} = \E \bx \mid \bz$.  Then by the orthogonality property of the condtional expectation, we have $\lambda_{\min}\left(\E \widehat{\bx} \widehat{\bx}^T\right) \geq \lambda_{\min}\left(\E \bx \bx^T \right) - \lambda_{\max}\left(\E \left(\bx - \bx_0\right)\left(\bx - \bx_0\right)^T\right)$.  Notice now that $\E \left(\bx - \bx_0\right)\left(\bx - \bx_0\right)^T = (1 - \alpha)^2 \Sigma_{\bX} - \alpha(1 - \alpha)\diag\left(\Sigma_{\bX}\right)$, which has maximum eigenvalue $(1 - \alpha)^2\lambda_{\max}(\Sigma_{\bX}) - \frac{\alpha(1 - \alpha)}{1 - \phi^2}$.  Now, Gershgorin's circle theorem implies that $\lambda_{\min}(\Sigma_{\bX}) \geq \frac{1}{4}$ and we know $\lambda_{\max}(\Sigma_{\bX}) \leq 3$.  Thus, assuming $\alpha > 0.844$, $\lambda_{\min}(\Sigma_{\widehat{\bX}}) > 0$.
\end{proof}

\section{Proofs of technical lemmas from appendix~\ref{subsec:proofGaussianRes}}
\label{sec:appendixboundedtechincallemmas}
This appendix is dedicated to the proofs needed for theorem~\ref{thm:GaussianRes}. It is organized as follows: sub-section~\ref{subsec:operatornorm} provides a useful lemma which allows us to control the empirical covariance, sub-section~\ref{subsec:markovblankets} provides a series of lemmas controlling the Markov blankets, subsection~\ref{subsec:usefulGaussianlemmas} provides a series of technical lemmas which will be used multiple times, and subsections~\ref{subsec:term1thmGaussian}-~\ref{subsec:term5thmGaussian} provide proofs for lemmas ~\ref{lem:term1thmGaussian} - ~\ref{lem:term5thmGaussian}.

\subsection{Control of the empirical covariance}
\label{subsec:operatornorm}
We begin with a lemma controlling the empirical covariance.
\begin{lemma}
    \label{lem:operatorMissing}
Consider the empirical covariance matrix  
\[\widetilde{\bSig} = \frac{1}{\alpha^2 n}\sum_{i=1}^{n}\bX_{0,i} \bX_{0,i}^\top -
        \frac{1 - \alpha}{\alpha^2 n}\sum_{i=1}^{n}\diag\left(\bX_{0,i}
        \bX_{0,i}^\top\right).\]  

We have the following two results:        
    \begin{enumerate}[align=left]
    \item[(i.)] For any submatrices $\bSig_{SS}, \widetilde{\bSig}_{SS}$
        and $u \geq 0$, if $\sqrt{\frac{|S| + u}{n}} \leq \frac{1}{2}$, we have:
        \begin{align*}
            \norm{\widetilde{\bSig}_{SS} - \bSig_{SS}}_{op} \leq
            \frac{C(1-\alpha)\norm{\bSig}_{op}}{\alpha^2}\sqrt{\frac{|S| + u}{n}}
        \end{align*}
        with probability at least $1 - e^{-u}$.
    \item[(ii.)]  Assume that for some $A \geq \sqrt{3}$,
        $\frac{A}{\alpha}\sqrt{\frac{\log{p}}{n}} \leq 1$.  Then, 
            \[\norm{\widetilde{\bSig} - \bSig}_{\ell_1 \rightarrow
        \ell_{\infty}} \leq \frac{CA\norm{\bSig}_{op}}{\alpha^2}\sqrt{\frac{\log{p}}{n}}\]
            with probability at least $1 - p^{3 - A^2}$.
    \end{enumerate}
\end{lemma}

\begin{proof}
    \vspace{1mm}
    \noindent i.
The proof follows by an $\epsilon$-net argument and is a straightforward
    extension of ~\cite{vershynin2018high}, Exercise 4.7.3.  First, let
    $\mathcal{N}_{1/4}$ be a $\frac{1}{4}$-net of the unit sphere $S^{|S| -
    1}$.  For convenience, we will drop the $S$ dependence of $\tilde{\bSig}$
    and $\bSig$ for the remainder of the proof, noting that both matrices are
    square with dimensions $|S| \times |S|$. Note now that $\norm{\tilde{\bSig} - \bSig}_{op} \leq 2\sup_{\bu \in
    \mathcal{N}_{1/4}}\left\lvert \left\langle\left(\tilde{\bSig} -
    \bSig\right)\bu, \bu\right\rangle\right\rvert$. Thus, we have:
    \begin{align*}
        \pr\left\{\norm{\tilde{\bSig} - \bSig}_{op} \geq t\right\} &\leq
        \sum_{\bu \in \mathcal{N}_{1/4}}\pr\left\{\left\lvert \left\langle\left(\tilde{\bSig} -
        \bSig\right)\bu, \bu\right\rangle\right\rvert \geq \frac{t}{2}\right\}
        \\
        &\leq \sum_{\bu \in \mathcal{N}_{1/4}}\pr\left\{\left\lvert U_1(\bu) \right\rvert \geq \frac{t}{4}\right\} +
        \pr\left\{\left\lvert U_2(\bu) \right\rvert \geq \frac{t}{4}\right\}
    \end{align*}
    where 
    \begin{align*}
        &U_1(\bu) = \left\langle \bu,
        \left(\frac{1}{n}\sum_{i=1}^{n}\frac{1}{\alpha^2}\bZ_{i}\bZ_{i}^T -
        \E\frac{1}{\alpha^2}\bZ_{1}\bZ_{1}^T\right)\bu\right\rangle\\
        &U_2(\bu) =
        \left\langle \bu, \left(\frac{1}{n}\sum_{i=1}^{n}\frac{1 -
        \alpha}{\alpha^2}\diag\left(\bZ_i \bZ_i^T\right) - \E\frac{1 -
        \alpha}{\alpha^2}\diag\left(\bZ_1 \bZ_1^T\right)\right)\bu\right\rangle
    \end{align*} 
    Let us first tackle the term $U_1(\bu)$.  To this end, we re-write:
    \begin{align*}
        \pr\left\{\left\lvert U_1(\bu) \right\rvert \geq \frac{t}{4}\right\} &=
        \pr\left\{\frac{1}{\alpha^2 n}\left\lvert\sum_{i=1}^{n}\left\langle \bZ_i, \bu \right\rangle^2 -
    \E \left\langle \bZ_i, \bu \right\rangle^2 \right\rvert\geq \frac{t}{4}\right\}
    \end{align*}
    Notice that $\E e^{\lambda
    \left\langle \bZ_i, \bu \right\rangle} \leq
    e^{\frac{1}{2}\lambda^2\norm{\bSig}_{op}\norm{\bu}_2^2}$.
    Thus, by Lemma~\ref{lem:subGaussianSquared}, we have $\forall \lambda \leq
    \frac{1}{8e\norm{\bSig}_{op}}$:
    \[\E\text{exp}\left\{\lambda\left(\left\langle \bZ_i, \bu \right\rangle^2 -
    \E \left\langle \bZ_i, \bu \right\rangle^2\right)\right\} \leq e^{\lambda^2
    32 e^2 \norm{\bSig}_{op}^2}\]
    Thus, if $t \leq \frac{32e \norm{\bSig}_{op}}{\alpha^2}$, we have:
    \[\pr\left\{\frac{1}{\alpha^2 n}\left\lvert\sum_{i=1}^{n}\left\langle \bZ_i, \bu \right\rangle^2 -
    \E \left\langle \bZ_i, \bu \right\rangle^2 \right\rvert\geq
    \frac{t}{4}\right\} \leq 2\text{exp}\left\{-\frac{n \alpha^4
    t^2}{1024e^2\norm{\bSig}_{op}^2}\right\}\]
    which implies by the union bound:
    \[\sum_{\bu \in \mathcal{N}_{1/4}}\pr\left\{\left\lvert U_1(\bu)\right\rvert
        \geq \frac{t}{4}\right\} \leq 2\left\lvert \mathcal{N}_{1/4}
        \right\rvert \text{exp}\left\{-\frac{n \alpha^4
    t^2}{1024e^2\norm{\bSig}_{op}^2}\right\}\]
    We now tackle $U_2(\bu)$.  First, we see that:
    \begin{align*}
        U_2(\bu) = \frac{1 - \alpha}{\alpha^2 n} \sum_{i=1}^{n} \left\langle
        \bu, \left(\text{diag}\left(\bZ_i \bZ_i^T\right) - \E \text{diag}\left(\bZ_i
        \bZ_i^T\right)\right)\bu\right\rangle = \frac{1 - \alpha}{\alpha^2 n}
        \sum_{i=1}^{n} \sum_{j=1}^{|S|} u_j^2 \left(\bZ_{ij}^2 - \E
        \bZ_{ij}^2\right)
    \end{align*}
    We are interested in 
    \begin{align*}
        \E \text{exp}\left\{\lambda\sum_{j=1}^{|S|} u_j^2 \left(\bZ_{ij}^2 - \E
        \bZ_{ij}^2\right)\right\} &= 1 + \sum_{k=2}^{\infty} \frac{\lambda^k \E
        \left\{\left(\sum_{j=1}^{|S|}u_j^2 \left(\bZ_{ij}^2 - \E
        \bZ_{ij}^2\right) \right)^k\right\}}{k!} \\
        &\leq 1 + \sum_{k=2}^{\infty} \frac{\lambda^k \E
        \left\{\left(\sum_{j=1}^{|S|}u_j^2 \bZ_{ij}^{2}\right)^k\right\}}{k!}
    \end{align*}
    Where the inequality follows by the non-negativity of $\bZ_{ij}^2$.  Notice
    now that since $\bu$ lies on the unit sphere, the vector $\left(\bu_1^2,
    \bu_2^2, \dots, \bu_{|S|}^2\right)$ lies on the simplex, and thus Jensen's
    inequality (with respect to the discrete distribution formed by the squared
    elements of $\bu$) gives $\left(\sum_{j=1}^{|S|}u_j^2 \bZ_{ij}^{2}\right)^k \leq
    \sum_{j=1}^{|S|}u_j^2 \bZ_{ij}^{2k}$.  Thus, we have:
    \begin{align*}
    1 + \sum_{k=2}^{\infty} \frac{\lambda^k \E
        \left\{\left(\sum_{j=1}^{|S|}u_j^2 \bZ_{ij}^{2}\right)^k\right\}}{k!}
        &\leq 1 + \sum_{k=2}^{\infty} \frac{\lambda^k 
        \sum_{j=1}^{|S|}u_j^2 \E \bZ_{ij}^{2k}}{k!} \\
        &\leq 1 + \sum_{k=2}^{\infty}\frac{\lambda^k 2k \left(2
        \norm{\bSig}_{op}\right)^k \Gamma(k) \norm{\bu}_2^2}{k!} 
    \end{align*}
    where the first inequality follows by the calculations of
    Equation~\ref{eq:2k_moment} and noticing that each random variable
    $\bZ_{ij}$ has sub-Gaussian constant at most $\norm{\bSig}_{op}$.  Now,
    under the assumption $\lambda < \frac{1}{8e\norm{\bSig}_{op}}$, we have:
    \begin{align*}
    1 + \sum_{k=2}^{\infty}\frac{\lambda^k 2k \left(2
        \norm{\bSig}_{op}\right)^k \Gamma(k) \norm{\bu}_2^2}{k!} 
        \leq \text{exp}\left\{32e^2\lambda^2 \norm{\bSig}_{op}^2\right\}
    \end{align*}
    Thus, if $t \leq \frac{32(1 - \alpha)e \norm{\bSig}_{op}}{\alpha^2}$, we
    have 
    \[\sum_{\bu \in \mathcal{N}_{1/4}}\pr\left\{\left\lvert U_1(\bu)\right\rvert
        \geq \frac{t}{4}\right\} \leq 2\left\lvert \mathcal{N}_{1/4}
        \right\rvert \text{exp}\left\{-\frac{n \alpha^4
    t^2}{1024(1 - \alpha)^2e^2\norm{\bSig}_{op}^2}\right\}\]
    Noting that $\left\lvert \mathcal{N}_{1/4} \right\rvert \leq 9^{|S|}$ and
    taking $t = \frac{64e(1 - \alpha)\norm{\bSig}_{op}\sqrt{|S| + u}}{\alpha^2
    \sqrt{n}}$ gives the result as long as $\sqrt{\frac{|S| + u}{n}} \leq
    \frac{1}{2}$.

    \vspace{1mm}
    \noindent ii. Notice that the off-diagonal entries of
    $\widetilde{\bSig}_{SS}$ are given by
    $\frac{1}{\alpha^2 n} \sum_{i=1}^{n}\bZ_{i\ell}\bZ_{ik}$ and the diagonal
    entries are given by $\frac{1}{\alpha n} \sum_{i=1}^{n} \bZ_{i\ell}^2$.  We
    notice that repeating the argument of Lemma~\ref{lem:productsubGaussian}
    using a symmetrization argument to account for products with non-zero mean
    implies $\E e^{\lambda \left(\bZ_{i\ell} \bZ_{ik} - \E \bZ_{i\ell}
    \bZ_{ik}\right)} \leq e^{128 e^2 \lambda^2
    \norm{\bSig}_{op}^2}$ for all $\lambda \leq \frac{1}{16e\norm{\bSig}_{op}}$.
    Similarly, by Lemma~\ref{lem:subGaussianSquared}, $\E e^{\lambda \left(
    \bZ_{i\ell}^2 - \E \bZ_{i\ell}^2\right)} \leq e^{32e^2 \lambda^2 \norm{\bSig}_{op}^2}$ for all
    $\lambda \leq \frac{1}{8e\norm{\bSig}_{op}}$.  This implies that for $t \leq
    \frac{16e\norm{\bSig}_{op}}{\alpha^2} \wedge
    \frac{8e\norm{\bSig}_{op}}{\alpha} $:
    \begin{align*}
        \pr\left\{\left\lvert \frac{1}{\alpha^2 n} \sum_{i=1}^{n} \bZ_{i\ell}
        \bZ_{ik} - \bSig_{\ell k} \right\rvert \geq t\right\} &\leq
        2\text{exp}\left\{- \frac{\alpha^4 n t^2}{256 e^2
        \norm{\bSig}_{op}^2}\right\}\\ 
        \pr\left\{\left\lvert \frac{1}{\alpha n} \sum_{i=1}^{n} \bZ_{i\ell}^2
         - \bSig_{\ell \ell} \right\rvert \geq t\right\} &\leq
        2\text{exp}\left\{- \frac{\alpha^2 n t^2}{64 e^2
        \norm{\bSig}_{op}^2}\right\}
    \end{align*}
    Thus for $p \geq 2$ a union bound gives: 
    \begin{align*}
        \pr\left\{\norm{\widetilde{\bSig} - \bSig}_{\ell_1 \rightarrow
        \ell_{\infty}} \geq t\right\} \leq \text{exp}\left\{- \frac{\alpha^4 n t^2}{256 e^2
        \norm{\bSig}_{op}^2} + 3 \log{p}\right\}
    \end{align*}
    Thus, taking $t = \frac{A \cdot 16e
    \norm{\bSig}_{op}}{\alpha^2}\sqrt{\frac{\log{p}}{n}}$ and assuming
    $\frac{A}{\alpha}\sqrt{\frac{\log{p}}{n}} \leq 1$ yields the result with
    probability at least $1 - p^{3-A^2}$
\end{proof}

\subsection{Control of Markov blankets}
\label{subsec:markovblankets}
\begin{lemma}
\label{lem:markovblanketmoments}
    \begin{itemize}
        Under the assumptions of theorem~\ref{thm:GaussianRes}, for any $\bu \in \mathbb{R}^{p}$, we have the following:
        \item[i.] $\E\left\{\left(\sum_{b=1}^{p}\left\lvert u_b \right\rvert\sqrt{S_{(i,b)}}\sum_{j
            \in S_{(i,b)}}\left\lvert X_{ij}\right\rvert\right)^2\right\} \leq
            \norm{\bu}_1^2 \norm{\bSig}_{op} \E S_{(i,b)}^3$, and
    \item[ii.] $\E\left\{\left(\sum_{b=1}^{p}\left\lvert u_b\right\rvert\sqrt{S_{(i,b)}}\sum_{j
            \in S_{(i,b)}}\left\lvert X_{ij}\right\rvert\right)^4\right\} \leq
            3\norm{\bu}_1^4 \norm{\bSig}_{op}^2 \E S_{(i,b)}^6$.
    \end{itemize}
    \end{lemma}
    \begin{proof}
    \begin{itemize}
    \item[i.] We first expand:
        \begin{align*}
        \E\left\{\left(\sum_{b=1}^{p}\left\lvert u_b\right\rvert\sqrt{S_{(i,b)}}\sum_{j
            \in S_{(i,b)}}\left\lvert X_{ij}\right\rvert\right)^2\right\} =
            \E\left\{\sum_{b, c \in [p] \times [p]} \prod_{I \in \{b, c\}} \left\lvert u_I\right\rvert\sqrt{S_{(i,I)}}\sum_{j
            \in S_{(i,I)}}\left\lvert\bX_{ij}\right\rvert\right\}
        \end{align*}
        We first look at the terms $b = c$ and the quantity of interest is thus
            $\E\left\{S_{(i,b)} \left(\sum_{j \in S_{(i,b)}}
            |X_{ij}|\right)^2\right\}$.  Noticing that $\E X_{ij}^2 \leq
            \norm{\bSig}_{op}$ and $\E \bX_{ij} X_{ik} \leq \norm{\bSig}_{op}$
            shows $\E\left\{S_{(i,b)} \left(\sum_{j \in S_{(i,b)}}
            |X_{ij}|\right)^2\right\} \leq \norm{\bSig}_{op} \E S_{(i,b)}^3$.
            Similarly noting that by Cauchy-Schwarz, $\E \sqrt{S_{(i,b)}
            S_{(i,c)}} \leq \E S_{i,b}$, we see that:
            \begin{align*}
            \E\left\{\sum_{b \neq c \in [p] \times [p]} \prod_{I \in \{b, c\}} \left\lvert u_I \right\rvert\sqrt{S_{(i,I)}}\sum_{j
                \in S_{(i,I)}}\left\lvert\bX_{ij}\right\rvert\right\} &= \sum_{b
                \neq c \in [p] \times [p]} \prod_{I \in \{b, c\}} \left\lvert u_I \right\rvert \E \left\{\sqrt{S_{(i,I)}}\sum_{j
                \in S_{(i,I)}}\left\lvert\bX_{ij}\right\rvert\right\} \\
                &\leq \sum_{b
                \neq c \in [p] \times [p]} \prod_{I \in \{b, c\}}
                \left\lvert u_I\right\rvert \norm{\bSig}_{op} \E
                S_{(i,b)}^3
            \end{align*}
            Summing the two pieces gives the desired result.
        \item[ii.]  First notice that $4$ identically distributed, positive
            random variables $Y_1, Y_2, Y_3, Y_4$ are such that $\E Y_1 Y_2 Y_3 Y_4
            \leq E Y_1^4$.  One can see this by noticing that $Y_1 Y_2 Y_3 Y_4 \leq
            \frac{1}{2}\left((Y_1 Y_2)^2 + (Y_3Y_4)^2\right)$ and that $\E
            (Y_1Y_2)^2 \leq \E Y_1^4$ where the second inequality is by
            Cauchy-Schwarz.  Now, expanding and noting that $\E |X_{ij}|^4
            \leq 3 \norm{\bSig}_{op}^2$ gives the result (following the same
            recipe as above).
    \end{itemize}
\end{proof}

In the remaining of this subsection, we bound the moment generating function (MGF) of the Markov blanket.

Let $G= (V=[d],E)$ be the sparsity graph of the precision matrix $\bOmega$; i.e., an edge is present if and only if the corresponding entry in $\bOmega$ is not zero. 
Each vertex is independently declared `open' with probability $\alpha$ and `closed' otherwise. 
We let $\omega_v \in \{0,1\}$ indicate the state of a vertex $v$ (open if $\omega_v =1$ and closed otherwise).   
The Markov blanket of the random variable $X_{i,u}$ (at vertex $u \in V$) is the set of first open vertices encountered by all walks starting at vertex $u$ on the graph.         
We denote its size by $S_{G}(v)$. Let $S_{T}(v)$ denote the size of the Markov blanket of $v$ in the $d_{\max}$-regular tree $T$ rooted at $v$, which we define in the same way.  
We first compare $S_{G}(v)$ with $S_{T}(v)$, and then perform a recursive argument to bound the MGF of latter.    
\begin{lemma}\label{lem:Comparison}
Let $v \in V$. There exists a coupling of $S_{G}(v)$ and $S_{T}(v)$ such that $\Prob(S_{G}(v) \le S_{T}(v)) = 1$.
\end{lemma}
\begin{proof}
The construction of the coupling relies on the notion of the \emph{path tree} of a graph~\cite{godsil1981matchings}. \emph{The path tree of $G$ rooted at $v$} is a finite tree rooted at $v$ whose vertices are simple paths $(v,v_1,\cdots,v_l)$ in $G$ starting at $v$ (a path is simple if no vertex appears more than once in it). Two paths are adjacent if one can be obtained by appending one new vertex to the other, i.e., edges of the path tree are of the form $(v,v_1,\cdots,v_l) \sim (v,v_1,\cdots,v_{l+1})$. 

Let $\tilde{T}$ be the path tree rooted at $v$ obtained from $G$. We first associate to every path $(v,v_1,\cdots,v_l)$ the state $\omega_{v_l}$ of its endpoint in $G$. We then consider all paths staring at $v$ and having the same endpoint $u$. If there is more than one such path (and this will be the case if the graph $G$ contains cycles) then we independently resample the state (open/closed) of all but one path. (We do this for all $u \neq v$.) The path whose state is not resampled is chosen arbitrarily, e.g., uniformly at random. 
In this way we have constructed a \emph{Bernoulli site percolation process} on the tree $\tilde{T}$ such that the size of the Markov blanket of $v$ in $G$ is almost surely upper-bounded by that of $v$ in $\tilde{T}$. 

Lastly, the tree $\tilde{T}$ has maximal degree $d_{\max}$ but may not be regular (it is not if $G$ is not regular). We extend it to an infinite $d_{\max}$-regular tree $T$ rooted at $v$ and associate to the extra vertices independent states sampled with probability $\alpha$. The size of the Markov blanket of $v$ can only grow with this operation.           
\end{proof}
Now we bound the moment generating function of $S_{T}(v)$. Let $u$ be an offspring of $v$. Let $S^{\downarrow}_{T}(u)$ be the number of first open vertices which are descendants of $u$, and $S^{\downarrow \le l}_{T}(u)$ the number of first open vertices up to distance $l$ from $u$. 
 Define $\chi(\theta) \equiv \E e^{\theta S^{\downarrow}_{T}(u)}$ and $\chi^{\le l}(\theta) \equiv \E e^{\theta S^{\downarrow \le l}_{T}(u)}$ , $\theta \in \mathbb{R}$. These functions do not depend on the specific choice of $u$. It is clear by monotone convergence that $\chi^{\le l}$ converges pointwise to $\chi$ as $l \to \infty$. Also note that $\chi$ is so far only defined formally: contrarily to $S^{\downarrow \le l}_{T}(u)$ which is finite almost surely, $S^{\downarrow}_{T}(u)$ may be infinite. 
 \begin{lemma}
\label{lem:BlanketMGF}
We have $\chi^{\le 0} = \alpha e^{\theta} + 1-\alpha$, and the functions $\chi^{\le l}, l \ge 1$ satisfy the recursion: $\chi^{\le l}(\theta) = \alpha e^{\theta} + (1-\alpha)(\chi^{\le l-1}(\theta))^{d_{\max}-1}$.
\end{lemma}
\noindent\begin{proof}
Since $S^{\downarrow \le 0}_{T}(u) = 1$ with probability $\alpha$ and $0$ otherwise, the first statement follows. 
Next, we exploit the recursive structure of the tree. The vertex $u$ has $d_{\max}-1$ offsprings and 
\begin{align*} 
S^{\downarrow \le l}_{T}(u) = 
\begin{cases}
1 &~\mbox{with probability}~ \alpha\\
\sum_{w \sim u} Z(w) &~\mbox{with probability}~ 1-\alpha,
\end{cases} 
\end{align*}
where $Z_w$ is the number of first open vertices which are descendants of $w$ and are within distance $l-1$ from $w$.
Since $(Z(w))_{w \sim u}$ are independent and have the same distribution as $S^{\downarrow \le l-1}_{T}(u)$, we have
\begin{align*}
\E e^{\theta S^{\downarrow \le l}_{T}(u)} &= \alpha e^{\theta} + (1-\alpha) \E e^{\theta \sum_{w \sim u} S^{\downarrow \le l-1}_{T}(w)}\\
&=\alpha e^{\theta} + (1-\alpha) \Big(\E e^{\theta S^{\downarrow \le l-1}_{T}(u)}\Big)^{d_{\max}-1}.
\end{align*}
\end{proof}

\noindent Next, Lemma~\ref{lem:generatingFunctionBoundary} follows from Lemma~\ref{lem:Comparison} and the following result.
\begin{lemma}
    If $(1-\alpha)(d_{\max}-1)<1$ then the random variable $S_{T}(v)$ has sub-exponential tail.  That is, there
    exists a constant $c = c(\alpha, d_{\text{max}})>0$ depending only on $\alpha,
    d_{\text{max}}$ such that:
    \begin{align*}
        \pr\left\{S_T(v) \geq t\right\} \leq 2e^{-ct}.
    \end{align*}
\end{lemma}
\noindent\begin{proof}
It suffices to prove that there exists $\theta>0$ such that $\chi(\theta)<+\infty$ when $(1-\alpha)(d_{\max}-1)<1$. Indeed, by the argument used in Lemma~\ref{lem:BlanketMGF}, the MGF of $v$ (the parent of $u$ and root of the tree) is given by $\E e^{\theta S_{T}(v)} = \alpha e^{\theta} + (1-\alpha)\chi(\theta)^{d_{\max}}$.
Therefore the argument boils down to analyzing the convergence of $\chi^{\le l}$ to a fixed point of $f_{\theta}(x) \equiv \alpha e^{\theta}+(1-\alpha)x^{d_{\max}-1}$. We have $f_{\theta}'(x) = (d_{\max}-1)(1-\alpha)x^{d_{\max}-2}$. The point at which $f_{\theta}'(x)=1$ is $x_* = ((d_{\max}-1)(1-\alpha))^{-1/(d_{\max}-2)}$. The equation $f_{\theta}(x)=x$ has two roots if $f_{\theta}(x_*) \le x_*$ (which merge when $f_{\theta}(x_*) = x_*$) and no roots otherwise. After rearranging the inequality $f_{\theta}(x_*) \le x_*$, this provides a bound on $\theta$:
\[\alpha e^{\theta} \le \frac{d_{\max}-2}{d_{\max}-1}\frac{1}{((d_{\max}-1)(1-\alpha))^{-1/(d_{\max}-2)}}.\]
Let $\bar{\theta} = \bar{\theta}(\alpha,d_{\max})$ the maximal value of $\theta$ such that the above bound holds, and let $\theta < \bar{\theta}$. It remains to show that the fixed point iteration $x_{l} = f_{\theta}(x_{l-1})$, $x_0 = 1$ converges to one of the two roots when $(d_{\max}-1)(1-\alpha) <1$. This will be true if $x_0$ is smaller than the largest fixed point, call it $x_+$. We observe that $f_{\theta}$ has derivative larger than 1 the second time it crosses the diagonal: $f_{\theta}'(x_+) >1$. Since $f_{\theta}'$ is increasing it suffices to check that $f_{\theta}'(x_0) < 1$ to ensure that $x_0 < x_+$. And we have $f_{\theta}'(x_0) = (1-\alpha)(d_{\max}-1)$.    
To sum up, if $(1-\alpha)(d_{\max}-1)<1$ and $\theta <\bar{\theta}$ then $(\chi^{\le l}(\theta))_{l\ge 0}$ converges to a finite limit, hence $S^{\downarrow}_{T}(u)$ has a finite MGF on $[0,\bar{\theta}]$ (and is the smallest of the two fixed points of $f_{\theta}$.)
\end{proof}

The next lemma follows immediately from Lemma~\ref{lem:Comparison} and the preceding result.
\begin{lemma}
    \label{lem:generatingFunctionBoundary}
    The random variables $S(i,b)$ have sub-exponential tails.  That is, there
    exists a constant $c =c(\alpha, d_{\text{max}})>0$ depending only on $\alpha,
    d_{\text{max}}$ such that for all $t \ge 0$,
    \begin{align*}
        \pr\left\{S(i,b) \geq t\right\} \leq c_0 e^{-ct},
    \end{align*}
    where $c_0$ is a universal constant.
\end{lemma}

The same strategy implies the following result:
\begin{lemma}
\label{lem:moments}
Suppose that $(1 - \alpha)(d_{\max} - 1)^k < 1$.  Then, the Markov blankets
    $S_{(i,b)}$ satisfy:
    \begin{align*}
    \E S_{(i,b)}^k \leq \alpha + (1 - \alpha) d_{\max}^k \cdot \frac{\alpha}{1 -
        (1 - \alpha)(d_{\max} - 1)^k}
    \end{align*}
\end{lemma}

\subsection{Useful lemmas}
\label{subsec:usefulGaussianlemmas}
We will use many times the lemmas in this sub-section.
\begin{lemma}
\label{lem:differencecontrol}
For all $b \in [p]$ and $i \in [n]$, we have:
\[
\norm{\left(\bSig_{b, S_{(i,b)}} -
        \widetilde{\bSig}_{b, S_{(i,b)}}\right)\bSig_{S_{(i,b)},
        S_{(i,b)}}^{-1}}_{\infty} \leq C(\alpha, d_{\max}) \sqrt{\left \lvert S_{(i,b)}\right\rvert}\sqrt{\frac{\log{2np}}{n}},
\]
and 
\[
\norm{\widetilde{\bSig}_{b, S_{(i,b)}}\left(\bSig_{S_{(i,b)},
        S_{(i,b)}}^{-1} - \widetilde{\bSig}_{S_{(i,b)},
        S_{(i,b)}}^{-1}\right)}_{\infty} \leq C(\alpha)\sqrt{\left \lvert S_{(i,b)} \right \rvert} \sqrt{\frac{\log{p}}{n}},
\]
with probability at least $1 - \frac{c}{p}$, where $c$ is a universal constant
and $C(\alpha)$ is a constant depending only on $\alpha$.
\end{lemma}
\begin{proof}
We examine the second term first.  Note that for some vector $\bx \in
\mathbb{R}^{d}$ and square matrix $\bA \in \mathbb{R}^{d \times d}$, we have the
following inequality (letting $\bA^{(i)}$ denote the $i$th column of $\bA$):
\begin{align*}
    \norm{\bx^T \bA}_{\infty} = \max_{i = 1, 2, \dots, d}\left\lvert\left\langle
    \bx, \bA^{(i)}\right\rangle\right\rvert \leq \max_{i=1,2,\dots, d}
    \norm{\bx}_{\infty} \norm{\bA^{(i)}}_{1} \leq \sqrt{d}\max_{i = 1, \dots, d}
    \norm{\bx}_{\infty} \norm{\bA^{(i)}}_{2} \leq
    \sqrt{d}\norm{\bx}_{\infty}\norm{\bA}_{op},
\end{align*}
which implies immediately:
    \begin{align*}
\norm{\widetilde{\bSig}_{b, S_{(i,b)}}\left(\bSig_{S_{(i,b)},
        S_{(i,b)}}^{-1} - \widetilde{\bSig}_{S_{(i,b)},
        S_{(i,b)}}^{-1}\right)}_{\infty} &\leq \norm{\widetilde{\bSig}_{b,
        S_{(i,b)}}}_{\infty}\sqrt{S_{(i,b)}}\norm{\bSig_{S_{(i,b)},
        S_{(i,b)}}^{-1} - \widetilde{\bSig}_{S_{(i,b)},
        S_{(i,b)}}^{-1}}_{op}.
\end{align*}
Our strategy will be to take the power series
    expansion of $\widetilde{\bSig}_{S_{(i,b)},
        S_{(i,b)}}^{-1} = \left(\bSig_{S_{(i,b)},
        S_{(i,b)}} + \bW_{S_{(i,b)},
        S_{(i,b)}}\right)^{-1}$, where for convenience we have used $\bW_{S_{(i,b)},
        S_{(i,b)}} = \widetilde{\bSig}_{S_{(i,b)},
        S_{(i,b)}} - \bSig_{S_{(i,b)},
        S_{(i,b)}}$.  Recall that by Lemma ~\ref{lem:operatorMissing} i., with
        probability at least $1 - e^{-u}$, $\norm{\widetilde{\bSig}_{SS} - \bSig_{SS}}_{op} \leq
            \frac{64e(1-\alpha)\norm{\bSig}_{op}}{\alpha^2}\sqrt{\frac{|S| +
            u}{n}}$ as long as $\sqrt{\frac{|S| + u}{n}} \leq \frac{1}{2}$.  Let
            $\mathcal{A}_O = \left\{\norm{\bW_{S_{(i,b)},
        S_{(i,b)}}}_{op} \leq
            \frac{64e(1-\alpha)\norm{\bSig}_{op}}{\alpha^2}\sqrt{\frac{|S_{(i,b)}| +
            2\log{2np}}{n}} \right\}$.  We will additionally need control on
            $|S(i,b)|$.
Let event $\mathcal{A}_B$ denote the event that the size of the Markov blankets
behave ``nicely", i.e. $\mathcal{A}_B = \left\{|S(i,b)| \leq
\frac{2}{C_B(\alpha, d_{\max})}\log{2np} \quad \forall i \in [n], b \in
[p]\right\}$ and notice that Lemma ~\ref{lem:generatingFunctionBoundary} implies
$\pr\left\{\mathcal{A}_B\right\} \geq 1 - \frac{1}{2np}$.  Now, on the event
$\mathcal{A}_O \cap \mathcal{A}_B$, we have:
\begin{align*}
\norm{\bW_{S_{(i,b)},
        S_{(i,b)}}}_{op} \leq
            \frac{64e(1-\alpha)\norm{\bSig}_{op}\left(1 + \frac{2}{C_B(\alpha, d_{\max})}\right)}{\alpha^2}\sqrt{\frac{\log{2np}}{n}}.
\end{align*}
Noticing that by assumption $\frac{64e(1-\alpha)\norm{\bSig}_{op}\left(1 +
\frac{2}{C_B(\alpha, d_{\max})}\right)}{\alpha^2}\sqrt{\frac{\log{2np}}{n}} <
1$, so we can take the power series expansion of $\left(\widetilde{\bSig}_{S_{(i,b)},
        S_{(i,b)}} - \bSig_{S_{(i,b)},
        S_{(i,b)}}\right)^{-1}$ and we see thus that:
        \begin{align*}
        \norm{\bSig_{S_{(i,b)},
        S_{(i,b)}}^{-1} - \widetilde{\bSig}_{S_{(i,b)},
        S_{(i,b)}}^{-1}}_{op} 
        &\leq \norm{\bSig_{S_{(i,b)},
        S_{(i,b)}}^{-1}}_{op}\sum_{k=1}^{\infty}\norm{\bSig_{S_{(i,b)},
        S_{(i,b)}}^{-1}}_{op}^{k}\norm{\bW_{S_{(i,b)},
        S_{(i,b)}}}_{op}^{k} \\
        &\leq
            \frac{128e(1-\alpha)\norm{\bSig}_{op}\left(1 +
\frac{2}{C_B(\alpha, d_{\max})}\right)}{\alpha^2
\lambda_{\text{min}}(\bSig)^2}\sqrt{\frac{\log{2np}}{n}},
    \end{align*}
    where in the first inequality we have used the power series expansion, the
    triangle inequality, and the sub-multiplicativity of the operator norm and
    the second inequality used the assumption that $\frac{64e(1-\alpha)\norm{\bSig}_{op}\left(1 +
\frac{2}{C_B(\alpha, d_{\max})}\right)}{\alpha^2
\lambda_{\text{min}}(\bSig)}\sqrt{\frac{\log{2np}}{n}} < \frac{1}{2}$. Now, let
$\mathcal{A}_M$ denote the event that $\norm{\widetilde{\bSig} - \bSig}_{\ell_1 \rightarrow
        \ell_{\infty}} \leq \frac{64
        e\norm{\bSig}_{op}}{\alpha^2}\sqrt{\frac{\log{p}}{n}}$.  Then on the
        event $\mathcal{A}_M \cap \mathcal{A}_O \cap \mathcal{A}_B$ and assuming
        that $\frac{64e}{\alpha^2}\sqrt{\frac{\log{p}}{n}} \leq 1$, we
    upper bound the overall term by:
    \begin{align*}
\norm{\widetilde{\bSig}_{b, S_{(i,b)}}\left(\bSig_{S_{(i,b)},
        S_{(i,b)}}^{-1} - \widetilde{\bSig}_{S_{(i,b)},
        S_{(i,b)}}^{-1}\right)}_{\infty} &\leq \norm{\widetilde{\bSig}_{b,
        S_{(i,b)}}}_{\infty}\sqrt{S_{(i,b)}}\norm{\bSig_{S_{(i,b)},
        S_{(i,b)}}^{-1} - \widetilde{\bSig}_{S_{(i,b)},
        S_{(i,b)}}^{-1}}_{op} \\
        &\leq \frac{128e(1-\alpha)\norm{\bSig}_{op}^{2}\left(1 +
\frac{2}{C_B(\alpha, d_{\max})}\right)}{\alpha^2
        \lambda_{\text{min}}(\bSig)^2}\sqrt{|S_{(i,b)}|}\sqrt{\frac{\log{2np}}{n}}.
    \end{align*}
    Similarly, we see that (working on the same event):
    \begin{align*}
    \norm{\left(\bSig_{b, S_{(i,b)}} -
        \widetilde{\bSig}_{b, S_{(i,b)}}\right)\bSig_{S_{(i,b)},
        S_{(i,b)}}^{-1}}_{\infty} &\leq \norm{\bSig_{b, S_{(i,b)}} -
        \widetilde{\bSig}_{b,
        S_{(i,b)}}}_{\infty}\sqrt{S(i,b)}\norm{\bSig_{S_{(i,b)},
        S_{(i,b)}}^{-1}}_{op} \\
        &\leq
        \frac{64e\norm{\bSig}_{op}}{\alpha^2\lambda_{\text{min}}(\bSig)}\sqrt{\frac{\log{p}}{n}}\sqrt{S_{(i,b)}}.
    \end{align*}
    The result follows immediately by recalling the assumption that $c_{\ell} \leq \lambda_{\min}(\bSig) \leq \norm{\bSig}_{op} \leq c_{\sigma}$, noting that $\pr\left\{\mathcal{A}_{O}\right\} \geq 1 - \frac{1}{2np}$ by lemma~\ref{lem:operatorMissing} i. and a union bound, and noting that $\pr\left\{\mathcal{A}_M\right\} \geq 1 - \frac{1}{p}$ by lemma~\ref{lem:operatorMissing} ii.
\end{proof}

\begin{lemma}
\label{lem:ucontrol}
Under the assumptions of theorem~\ref{thm:GaussianRes} and for any vector $\bu \in \mathbb{R}^{p}$, 
\[
    \frac{1}{n}\sum_{i=1}^{n}\left(\sum_{b=1}^{p}\left\lvert u_b \right\rvert\sqrt{S_{(i,b)}}\sum_{j
        \in S_{(i,b)}}\left\lvert X_{ij}\right\rvert\right)^2 \leq C(\alpha,
        d_{\max})\norm{\bu}_1^2,
\]
with probability at least $1 - \frac{1}{n}$, where $C(\alpha, d_{\max})$ is a constant depending only on $\alpha, d_{\max}$.
\end{lemma}
\begin{proof}
This involves a simple application of Chebyshev's inequality.  First, we let Let 
\[
Z_i = \left(\sum_{b=1}^{p}\left\lvert u_b \right\rvert\sqrt{S_{(i,b)}}\sum_{j
        \in S_{(i,b)}}\left\lvert X_{ij}\right\rvert\right)^2.
\]
Chebyshev's inequality then implies:
\[
\pr\left\{\left \lvert\frac{1}{n}\sum_{i=1}^{n} Z_i  - \E Z_i \right \rvert \geq \sqrt{\E Z_i^2}\right\} \leq \frac{1}{n}.
\]
Now, by lemma~\ref{lem:markovblanketmoments} and the assumption that $c_{\ell} \leq \lambda_{\min}(\bSig) \leq \norm{\bSig}_{op} \leq c_{\sigma}$ imply the result.
\end{proof}

\begin{lemma}
\label{lem:maxmarkovboundarycontrol}
    Under the assumptions of theorem~\ref{thm:GaussianRes}, 
    \[
\max_{a \in 1, 2, \dots, p}\frac{1}{n}\sum_{i=1}^{n}\left(\sqrt{S_{(i,a)}}\sum_{j
        \in S_{(i,a)}} \left\lvert X_{ij}\right\rvert\right)^2 \leq
        C(\alpha, d_{\max}), 
    \]
    with probability at least $1 - c_0 p^{-c_1}$, where $c_0, c_1$ are absolute constants. 
\end{lemma}
\begin{proof}
We proceed by obtaining tail bounds on the quantity $\left(\sqrt{S_{(i,a)}}\sum_{j
        \in S_{(i,a)}} \left\lvert X_{ij}\right\rvert\right)^2$ and using this
        to find an exponential tail bound on the empirical average from which
        point we can conclude with a union bound.  Notice that
    \begin{align*}
        \pr\left\{S_{(i,a)}\sum_{\ell,j \in S_{(i,a)} \times S_{(i,a)}}\left\lvert
        X_{ij}X_{i\ell}\right\rvert \geq t\right\} &= \sum_{k =1}^{p}
        \pr\left\{\sum_{\ell,j \in S_{(i,a)} \times S_{(i,a)}}\left\lvert
        X_{ij}X_{i\ell}\right\rvert \geq
        \frac{t}{k}\right\} \pr\left\{S_{(i,a)} = k\right\} \\
        &\leq
        \sum_{k=1}^{\infty}4\text{exp}\left\{-\frac{t}{16e\norm{\bSig}_{op} k^3} -
        C_B(\alpha, d_{\text{max}})k\right\},
    \end{align*}
    where the last inequality follows by lemma~\ref{lem:generatingFunctionBoundary} and noting that $\sum_{\ell,j \in S_{(i,a)} \times S_{(i,a)}}\left\lvert
        X_{ij}X_{i\ell}\right\rvert$ is a sub-exponential random variable.
    Now, we balance terms in the infinite sum around $t^{\frac{1}{4}}$ to see:
    \begin{align*}
    \sum_{k=1}^{\infty}4\exp\left\{-\frac{t}{16e\norm{\bSig}_{op} k^3} -
        C_B(\alpha, d_{\text{max}})k\right\} &=
        4e^{-at^{\frac{1}{4}}}\sum_{k=1}^{\infty}\exp\left\{at^{\frac{1}{4}}
        - \frac{t + 16e\norm{\bSig}_{op} C_B(\alpha,
        d_{\max})k^4}{16e\norm{\bSig}_{op}k^3}\right\} \\
        &= 4e^{-at^{\frac{1}{4}}} \Biggl(\sum_{k \leq t^{1/4}}\exp\left\{at^{\frac{1}{4}}
        - \frac{t + 16e\norm{\bSig}_{op} C_B(\alpha,
        d_{\max})k^4}{16e\norm{\bSig}_{op}k^3}\right\}\\
        & + \sum_{k > t^{1/4}}\exp\left\{at^{\frac{1}{4}}
        - \frac{t + 16e\norm{\bSig}_{op} C_B(\alpha,
        d_{\max})k^4}{16e\norm{\bSig}_{op}k^3}\right\}\Biggr) \\
        &\leq \frac{4}{1 - e^{-C_B(\alpha, d_{\max})/2}}e^{-at^{\frac{1}{4}}}
    \end{align*}
    where we have taken $a$ such that $a < \frac{1}{16e\norm{\bSig}_{op}} \wedge
    \frac{C_B(\alpha, d_{\max})}{2}$. Now, assuming 
    \[t < \frac{128e}{a -  ae^{-C_B(\alpha, d_{\max})/2}}\left(n
    \frac{a -  ae^{-C_B(\alpha, d_{\max})/2}}{1024e^2}\right)^{3}\] 
    by Lemma~\ref{lem:truncation}, we have: 
    \begin{align*}
        &\pr\left\{\frac{1}{n}\sum_{i=1}^{n}\left(\sqrt{S_{(i,a)}}\sum_{j
        \in S_{(i,a)}} \left\lvert X_{ij}\right\rvert\right)^2 \geq 2\left(t + \E \left(\sqrt{S_{(i,a)}}\sum_{j
        \in S_{(i,a)}} \left\lvert X_{ij}\right\rvert\right)^2\right)\right\}
        \\
        &\leq 2
        \text{exp}\left\{-a\left(n t^2\frac{a -  ae^{-C_B(\alpha, d_{\max})/2}}{1024e^2}\right)^{\frac{1}{7}}
        + \log{n\left(\frac{4}{1 - e^{-C_B(\alpha, d_{\max})/2}}\right)}\right\}
    \end{align*}
    Thus, let $\mathcal{A}_{T}$ be the event that 
    \[\max_{a \in 1, 2, \dots, p}\frac{1}{n}\sum_{i=1}^{n}\left(\sqrt{S_{(i,a)}}\sum_{j
        \in S_{(i,a)}} \left\lvert X_{ij}\right\rvert\right)^2 \leq
        2\left(A\sqrt{\frac{(\log{c_1 np})^7}{n}} + \E \left(\sqrt{S_{(i,a)}}\sum_{j
        \in S_{(i,a)}} \left\lvert X_{ij}\right\rvert\right)^2\right) \]
    where $c_1 = 2 + \frac{4}{1 - e^{-C_B(\alpha, d_{\max})/2}}$ and note that
    $\pr\left\{\mathcal{A}_T\right\} \geq 1 - \text{exp}\left\{\left(1 - c_2
    A^{2/7}\right)\log{c_1 np}\right\}$ where $c_2 = a \left(\frac{a -
    ae^{-C_B(\alpha, d_{\max})/2}}{1024e^2}\right)^{1/7}$.  Thus, on the event
    $\mathcal{A}_T$, and assuming $A\sqrt{\frac{(\log{c_1 np})^7}{n}} \leq
    \norm{\bSig}_{op} \E S_{(i,a)}^3$, we have by Lemma ~\ref{lem:markovblanketmoments}:
    \[\max_{a \in 1, 2, \dots, p}\frac{1}{n}\sum_{i=1}^{n}\left(\sqrt{S_{(i,a)}}\sum_{j
        \in S_{(i,a)}} \left\lvert\bX_{ij}\right\rvert\right)^2 \leq 4\norm{\bSig}_{op} \E S_{(i,a)}^3.\]
     The result follows immediately taking $A$ large enough and using lemma~\ref{lem:moments}.
\end{proof}

\subsection{Proof of lemma~\ref{lem:term1thmGaussian}}
\label{subsec:term1thmGaussian}
We re-state the lemma for convenience.
\begin{customlemma}{B.9}
Under the assumptions of theorem~\ref{thm:GaussianRes}, for any $\bu \in \mathbb{R}^{p}$, we have:
\[
\pr\left\{\norm{\frac{1}{n}\widehat{\bX}^T\left(\widetilde{\bX} - \widehat{\bX}\right)\bu}_{\infty} \geq C(\alpha, d_{\max})\norm{\bu}_1\sqrt{\frac{\log{np}}{n}}\right\} \leq n^{-1} + c_0p^{-c_1},
\]
where $c_0, c_1$ denote universal constants and $C(\alpha, d_{\max})$ a constant depending only on $\alpha, d_{\max}$.
\end{customlemma}
\begin{proof}
We write:
\begin{align*}
\norm{\frac{1}{n}\widehat{\bX}^T\left(\widetilde{\bX} -
\widehat{\bX}\right)\bu}_{\infty} &= \max_{a = 1, 2, \dots, p}\left\lvert\frac{1}{n}
    \sum_{i=1}^{n}\widehat{X}_{ia}\sum_{b=1}^{p}\left(\widetilde{X}_{ib}
-\widehat{X}_{ib}\right)u_b\right\rvert\\
    &\leq \max_{a \in
    1, 2, \dots,
    p}\sqrt{\frac{1}{n}\sum_{i=1}^{n}\left\lvert\widehat{X}_{ia}\right\rvert^2}
    \cdot \sqrt{\frac{1}{n}\sum_{i=1}^{n}\left(\sum_{b=1}^{p}\left(\widetilde{X}_{ib}
    -\widehat{X}_{ib}\right)u_b\right)^2} ,
\end{align*}
where the inequality is by Cauchy-Schwarz over the indices $i$. Our strategy is
then to upper bound each of the square root terms separately.  The more
    difficult term is the second, so we begin there, first by re-writing the
    approximate conditional expectation explicitly:
    \begin{align*}
    \sqrt{\frac{1}{n}\sum_{i=1}^{n}\left(\sum_{b=1}^{p}\left(\widetilde{X}_{ib}
        -\widehat{X}_{ib}\right)u_b\right)^2} &= \sqrt{\frac{1}{n}\sum_{i=1}^{n}\left(\sum_{b=1}^{p}\left\lvert u_b\left(\bSig_{b, S_{(i,b)}}\bSig_{S_{(i,b)},
        S_{(i,b)}}^{-1} -  \widetilde{\bSig}_{b, S_{(i,b)}}\widetilde{\bSig}_{S_{(i,b)},
        S_{(i,b)}}^{-1}\right)\bX_{S(i,b)}\right\rvert\right)^2}. 
    \end{align*}
    An application of H{\"o}lder's inequality (noting that $\left(\bSig_{b, S_{(i,b)}}\bSig_{S_{(i,b)},
        S_{(i,b)}}^{-1} -  \widetilde{\bSig}_{b, S_{(i,b)}}\widetilde{\bSig}_{S_{(i,b)},
        S_{(i,b)}}^{-1}\right)$ is a row vector) yields the following upper bound:
    \begin{equation}
    \label{eq:TermAUpperBound}
\sqrt{\frac{1}{n}\sum_{i=1}^{n}\left(\sum_{b=1}^{p}\left\lvert u_b \right\rvert\norm{\bSig_{b, S_{(i,b)}}\bSig_{S_{(i,b)},
        S_{(i,b)}}^{-1} -  \widetilde{\bSig}_{b, S_{(i,b)}}\widetilde{\bSig}_{S_{(i,b)},
        S_{(i,b)}}^{-1}}_{\infty}\sum_{j
    \in S_{(i,b)}}\left\lvert X_{ij}\right\rvert\right)^2}.
    \end{equation}
The key difficulty is that each of the terms is correlated across indices $i$ in
    a complicated manner,
    so we are unable to immediately use standard concentration inequalities on
    the sum.  To deal with this, we will uniformly upper-bound each of these
    terms which correlate the terms in the sum.
To this end, we have:
    \begin{align*}
\norm{\bSig_{b, S_{(i,b)}}\bSig_{S_{(i,b)},
        S_{(i,b)}}^{-1} -  \widetilde{\bSig}_{b, S_{(i,b)}}\widetilde{\bSig}_{S_{(i,b)},
        S_{(i,b)}}^{-1}}_{\infty} &\leq \norm{\left(\bSig_{b, S_{(i,b)}} -
        \widetilde{\bSig}_{b, S_{(i,b)}}\right)\bSig_{S_{(i,b)},
        S_{(i,b)}}^{-1}}_{\infty} + \\
        &\norm{\widetilde{\bSig}_{b, S_{(i,b)}}\left(\bSig_{S_{(i,b)},
        S_{(i,b)}}^{-1} - \widetilde{\bSig}_{S_{(i,b)},
        S_{(i,b)}}^{-1}\right)}_{\infty}.
    \end{align*}
    Now, by lemma~\ref{lem:differencecontrol}, we have:
    \[
    \norm{\bSig_{b, S_{(i,b)}}\bSig_{S_{(i,b)},
        S_{(i,b)}}^{-1} -  \widetilde{\bSig}_{b, S_{(i,b)}}\widetilde{\bSig}_{S_{(i,b)},
        S_{(i,b)}}^{-1}}_{\infty} \leq C(\alpha, d_{\max})\sqrt{\lvert S_{(i,b)} \rvert}\sqrt{\frac{\log{np}}{n}},
    \]
    with probability at least $1 - \frac{c}{p}$.  We thus upper bound ~\eqref{eq:TermAUpperBound} with:
    \[
    C(\alpha, d_{\max}) \sqrt{\frac{\log{np}}{n}} \sqrt{\frac{1}{n}\sum_{i=1}^{n}\left(\sum_{b=1}^{p}\left\lvert u_b \right\rvert\sqrt{S_{(i,b)}}\sum_{j
    \in S_{(i,b)}}\left\lvert X_{ij}\right\rvert\right)^2}.
    \]
    By lemma~\ref{lem:ucontrol}, with probability at least $1 - \frac{1}{n}$, 
    \[
    C(\alpha, d_{\max}) \sqrt{\frac{\log{np}}{n}} \sqrt{\frac{1}{n}\sum_{i=1}^{n}\left(\sum_{b=1}^{p}\left\lvert u_b \right\rvert\sqrt{S_{(i,b)}}\sum_{j
    \in S_{(i,b)}}\left\lvert X_{ij}\right\rvert\right)^2} \leq C(\alpha, d_{\max})\norm{\bu}_1\sqrt{\frac{\log{np}}{n}}.
    \]
    Now, note that $\max_{a \in [p]}\frac{1}{n}\sum_{i=1}^{n} \lvert \widehat{X}_{ia} \rvert ^2 \leq 18e\norm{\bSig}_{op}$ with probability at least $1 - c_0 p^{-c_1}$ by the same argument as lemma~\ref{lem:usefulARlemmas} iii. Combining these pieces implies the result. 
\end{proof}

\subsection{Proof of lemma~\ref{lem:term2thmGaussian}}
\label{subsec:term2thmGaussian}
We re-state the lemma for convenience:
\begin{customlemma}{B.11}
Under the assumptions of theorem~\ref{thm:GaussianRes}, we have:
\[
\pr\left\{\norm{\frac{1}{n}\left(\widetilde{\bX} - \widehat{\bX}\right)^T\left(\widetilde{\bX} - \widehat{\bX}\right)\bu}_{\infty} \geq C(\alpha, d_{\max})\norm{\bu}_1\frac{\log{np}}{n}\right\} \leq n^{-1} + c_0p^{-c_1},
\]
where $c_0, c_1$ denote universal constants and $C(\alpha, d_{\max})$ a constant depending only on $\alpha, d_{\max}$.
\end{customlemma}
\begin{proof}
We write:
        \begin{align*}
        \norm{\frac{1}{n}\left(\widetilde{\bX} -
            \widehat{\bX}\right)^T\left(\widetilde{\bX} -
            \widehat{\bX}\right)\bu}_{\infty} &= \max_{a = 1, 2, \dots, p}\left\lvert\frac{1}{n}
            \sum_{i=1}^{n}\left(\widetilde{X}_{ia} -
            \widehat{X}_{ia}\right)\sum_{b=1}^{p}\left(\widetilde{X}_{ib}
-\widehat{X}_{ib}\right)u_b\right\rvert\\
    &\leq \max_{a \in
    1, 2, \dots,
            p}\sqrt{\frac{1}{n}\sum_{i=1}^{n}\left(\widetilde{X}_{ia} - \widehat{X}_{ia}\right)^2}
    \cdot \sqrt{\frac{1}{n}\sum_{i=1}^{n}\left(\sum_{b=1}^{p}\left(\widetilde{X}_{ib}
    -\widehat{X}_{ib}\right)u_b\right)^2}.
        \end{align*}
    We tackle the first term first, noticing that with probability at least $1 - c_0p^{-c_1}$, 
    \begin{align*}
        \max_{a \in
    1, 2, \dots,
            p}\sqrt{\frac{1}{n}\sum_{i=1}^{n}\left(\widetilde{X}_{ia} -
            \widehat{X}_{ia}\right)^2} &\leq \max_{a \in [p]}\sqrt{\frac{1}{n}\sum_{i=1}^{n}\left(\norm{\bSig_{a, S_{(i,a)}}\bSig_{S_{(i,a)},
        S_{(i,a)}}^{-1} -  \widetilde{\bSig}_{a, S_{(i,a)}}\widetilde{\bSig}_{S_{(i,a)},
        S_{(i,a)}}^{-1}}_{\infty}\sum_{j
    \in S_{(i,a)}}\left\lvert X_{ij}\right\rvert\right)^2}\\
        &\leq C(\alpha, d_{\max})\sqrt{\frac{\log{np}}{n}}\max_{a \in
        [p]}\sqrt{\frac{1}{n}\sum_{i=1}^{n}\left(\sqrt{S_{(i,a)}}\sum_{j
        \in S_{(i,a)}} \left\lvert X_{ij}\right\rvert\right)^2} \\
        &\leq C(\alpha, d_{\max})\sqrt{\frac{\log{np}}{n}},
    \end{align*}
    where the second inequality follows by lemma~\ref{lem:differencecontrol} and the last inequality follows by lemma~\ref{lem:maxmarkovboundarycontrol}.  The second term we analyzed in lemma~\ref{lem:term1thmGaussian}, and thus with probability at least $1 - c_0p^{-1} - \frac{1}{n}$, 
    \[
\sqrt{\frac{1}{n}\sum_{i=1}^{n}\left(\sum_{b=1}^{p}\left(\widetilde{X}_{ib}
    -\widehat{X}_{ib}\right)u_b\right)^2} \leq C(\alpha, d_{\max})\norm{\bu}_1\sqrt{\frac{\log{np}}{n}}.
    \]
    The result follows immediately.
\end{proof}

\subsection{Proof of lemma~\ref{lem:term3thmGaussian}}
\label{subsec:term3thmGaussian}
We re-state the lemma:
\begin{customlemma}{B.12}
Under the assumptions of theorem~\ref{thm:GaussianRes}, for any $\bu \in \mathbb{R}^{p}$, we have:
\[
\pr\left\{\norm{\frac{1}{n}\left(\widetilde{\bX} - \widehat{\bX}\right)^T\left(\widehat{\bX} - \bX\right)\bu}_{\infty} \geq C(\alpha, d_{\max})\norm{\bu}_1 \sqrt{\frac{\log{np}}{n}}\right\} \leq n^{-1} + c_0p^{-1},
\]
where $c_0, c_1$ denote universal constants and $C(\alpha, d_{\max})$ a constant depending only on $\alpha, d_{\max}$.
\end{customlemma}
\begin{proof}
We begin by writing
\begin{align*}
    \norm{\frac{1}{n}\left(\widetilde{\bX} -
            \widehat{\bX}\right)^T\left(\widehat{\bX} -
            \bX\right)\bbe_0}_{\infty} &\leq \max_{a \in
    1, 2, \dots,
        p}\sqrt{\frac{1}{n}\sum_{i=1}^{n}\left(\widetilde{X}_{ia} -
        \widehat{X}_{ia}\right)^2}
        \cdot \norm{\bu}_2\sqrt{\frac{1}{n}\sum_{i=1}^{n}\left\langle \widehat{\bX}_i -
        \bX_i, \widetilde{\bu} \right\rangle^2},
    \end{align*}
    where $\widetilde{\bu} = \frac{\bu}{\norm{\bu}_2}$.  Notice that by fact~\ref{fact:sumsubgaussian},
    the vector $\widehat{\bX}_{i} - \bX_i$ is sub-Gaussian with
    parameter $4 \norm{\bSig}_{op}$.  Notice also that:
        \begin{align*}
        \E\left\langle \widehat{\bX}_i -
            \bX_i, \widetilde{\bu} \right\rangle^2 &\leq \sum_{a,b}
            \left\lvert \widetilde{u}_{a} \widetilde{u}_{b} \right  \rvert \E \left
            \lvert \widehat{X}_{ia} - X_{ia}\right\rvert^2  \\
            &\leq 4\sum_{a,b} \left\lvert \widetilde{u}_{a}
            \widetilde{u}_{b} \right \rvert \E
            X_{ia}^2 \\
            &\leq 4\norm{\bSig}_{op} \norm{\widetilde{\bu}}_1^2,
        \end{align*}
        where the first inequality is by Cauchy Schwarz and the second by
        Jensen's inequality.  
        Applying Lemma~\ref{lem:subGaussianSquared} and
       	lemma~\ref{lem:bernstein}, we obtain:
        \begin{align*}
            \pr\left\{\frac{1}{n}\sum_{i=1}^{n}\left\langle \widehat{\bX}_i -
        \bX_i, \widetilde{\bu} \right\rangle^2 -  \E\left\langle \widehat{\bX}_i -
            \bX_i, \widetilde{\bu} \right\rangle^2 \geq t\right\} &\leq
            \text{exp}\left\{-c\frac{nt^2}{\norm{\bSig}_{op}^2}\right\}
        \end{align*}
        for all $t < C\norm{\bSig}_{op}$.  Setting $t = \norm{\bSig}_{op}$
        implies
        \[
        \sqrt{\frac{1}{n}\sum_{i=1}^{n}\left\langle \widehat{\bX}_i -
        \bX_i, \widetilde{\bu} \right\rangle^2} \leq C \norm{\bu}_1^2.
        \]
        with probability at least $1 - e^{-c n}$.  Note that by the analysis of lemma~\ref{lem:term2thmGaussian}, with probability at least $1 - c_0p^{-1} - n^{-1}$
        \[
        \max_{a \in
    1, 2, \dots,
        p}\sqrt{\frac{1}{n}\sum_{i=1}^{n}\left(\widetilde{X}_{ia} -
        \widehat{X}_{ia}\right)^2} \leq C(\alpha, d_{\max})\sqrt{\frac{\log{np}}{n}}.
        \]
        The result follows immediately.
\end{proof}

\subsection{Proof of lemma~\ref{lem:term5thmGaussian}}
\label{subsec:term5thmGaussian}
We re-state the lemma for convenience. 
\begin{customlemma}{B.14}
Under the assumptions of theorem~\ref{thm:GaussianRes}, we have:
\[
\pr\left\{\norm{\frac{1}{n}\left(\widetilde{\bX} - \widehat{\bX}\right)^T\beps}_{\infty} \geq C(\alpha, d_{\max})\sigma\sqrt{\frac{\log{np}}{n}}\right\} \leq c_0n^{-1} + c_1p^{-c_2},
\]
where $c_0, c_1$ denote universal constants and $C(\alpha, d_{\max})$ a constant depending only on $\alpha, d_{\max}$.
\end{customlemma}

\begin{proof}
We write:
\begin{align*}
    \norm{\frac{1}{n}\left(\widetilde{\bX} -
        \widehat{\bX}\right)^T\beps}_{\infty} &= \max_{a = 1, 2, \dots, p}
        \left\lvert \frac{1}{n}\sum_{i=1}^{n} \left(\widetilde{X}_{ia} -
        \widehat{X}_{ia}\right)\epsilon_i \right\rvert \\
        &\leq \max_{a = 1, 2, \dots,
        p}\sqrt{\frac{1}{n}\sum_{i=1}^{n}\left(\widetilde{X}_{ia} -
        \widehat{X}_{ia}\right)^2}\sqrt{\frac{1}{n}\sum_{i=1}^{n}\epsilon_i^2}.
    \end{align*}
    Note that by the analysis of lemma~\ref{lem:term2thmGaussian}, with probability at least $1 - c_0p^{-1} - n^{-1}$
        \[
        \max_{a \in
    1, 2, \dots,
        p}\sqrt{\frac{1}{n}\sum_{i=1}^{n}\left(\widetilde{X}_{ia} -
        \widehat{X}_{ia}\right)^2} \leq C(\alpha, d_{\max})\sqrt{\frac{\log{np}}{n}}.
        \]
        Now define the event $\mathcal{A}_E$ such that $\frac{1}{n}\sum_{i=1}^{n}
    \epsilon_i^2 \leq 2\sigma^2$ and notice that by
    lemma~\ref{lem:subGaussianSquared} and lemma~\ref{lem:bernstein}, we have
    $\pr\left\{\mathcal{A}_E^C\right\} \leq
    \exp\left\{-Cn\right\}$.  The result follows from these obvservations.
\end{proof}


\section{Facts about sub-gaussian and sub-exponential random variables}
\label{sec:appendixsubgaussian}
This appendix contains a collection of facts about sub-gaussian random variables.

\begin{lemma}
\label{lem:productsubgaussian}
Assume $X$ is $\sigma_X^2$ sub-gaussian and $Y$ is $\sigma_Y^2$ sub-gaussian and that $XY$ is zero-mean.
Then, $\forall \theta \leq \frac{1}{8e\sigma_X \sigma_Y}$:
\[\E e^{\theta XY} \leq e^{32e^2\theta^2 \sigma_X^2 \sigma_Y^2}.\]
\end{lemma}
\begin{proof}
Taylor expanding $e^{\theta XY}$ gives: 
\begin{align*}
\E e^{\theta XY} \leq 1 +
\sum_{k=2}^{\infty} \frac{\theta^{k}\E |XY|^{k}}{k!} \leq 1 +
\sum_{k=2}^{\infty} \frac{\theta^{k} \sqrt{\E |X|^{2k} \E |Y|^{2k}}}{k!},
\end{align*}
where the first inequality follows since $\E XY = 0$ and $XY \leq |XY|$ and the
second inequality follows by Cauchy-Schwarz.  Now,
\begin{align}
\label{eq:2k_moment}
\E |X|^{2k} = \int_{0}^{\infty} 2k t^{2k-1} \Pr\left\{|X| \geq t\right\} dt \leq 2k
(2\sigma_X^2)^k \Gamma(k),
\end{align}
where $\Gamma(k) = \int_{0}^{\infty} t^{k-1}e^{-t}dt$.  Stirling's inequality
gives $\Gamma(k) \leq k^{k}$ and noting that $2k \leq 2^{k}$ for $k \geq 2$
implies $\E |X|^{2k} \leq (4\sigma_X^2)^{k} k^{k}$.  Proceeding similarly for
$\E |Y|^{2k}$ yields 
\[1 +
\sum_{k=2}^{\infty} \frac{\theta^{k} \sqrt{\E |X|^{2k} \E |Y|^{2k}}}{k!} \leq 1
+ \sum_{k=2}^{\infty} \frac{\theta^k (4\sigma_X \sigma_Y)^k k^k}{k!} \leq 1 +
\sum_{k=2}^{\infty}(4\theta e \sigma_X \sigma_Y)^k = 1 + \frac{(4\theta e
\sigma_X \sigma_Y)^{2}}{1 - 4\theta e \sigma_X \sigma_Y},\]
where the third inequality used the fact that $k! \geq \frac{k^k}{e^k}$ and the
last equality invoked the sum of a geometric series.  Finally, noting that
$\theta \leq \frac{1}{8e\sigma_X \sigma_Y}$, we can upper bound 
\[1 + \frac{(4\theta e
\sigma_X \sigma_Y)^{2}}{1 - 4\theta e \sigma_X \sigma_Y} \leq 1 +
32e^2\theta^2\sigma_X^2\sigma_Y^2 \leq e^{32 e^2 \theta^2 \sigma_X^2 \sigma_Y^2}.\]
\end{proof}

\noindent We will repeatedly make use of the following related lemma whose proof
is nearly identical.
\begin{lemma}
\label{lem:subGaussianSquared}
    Assume $X$ is $\sigma_X^2$ sub-gaussian.  Then, $\forall \theta \leq
    \frac{1}{16e\sigma_X^2}$:
    \[
    \E e^{\theta \left(X^2 - \E X^2\right)} \leq e^{128e^2\theta^2\sigma_X^4}.
    \]
\end{lemma}
\begin{proof}
As before, we begin by Taylor expanding $e^{\lambda XY}$ gives: 
\begin{align*}
\E e^{\lambda \left(X^2 - \E X^2\right)} \leq 1 +
\sum_{k=2}^{\infty} \frac{\lambda^{k}\E \left\lvert X^2 - \E X^2\right\rvert^{k}}{k!} \leq 1 +
    \sum_{k=2}^{\infty} \frac{\left(2\lambda\right)^{k} \E |X|^{2k}}{k!}
\end{align*}
    where the last inequality follows by writing $X^2 - \E X^2 = X^2 - \E_{X_1}
    X_1^2$ where $X_1$ is an i.i.d. copy of $X$ and using Jensen's inequality.  The remainder of the proof
    follows exactly as that of Lemma~\ref{lem:productsubgaussian}.
\end{proof}

\noindent We will use another similar lemma, whose proof we omit as it can be easily derived using the ideas of the previous two lemmas.
\begin{lemma}
\label{lem:absvalprodsubgaussian}
Assume $X$ is $\sigma_X^2$ sub-gaussian and $Y$ is $\sigma_Y^2$ sub-gaussian.  Then, for all $\lvert \theta \rvert \leq \frac{1}{16e\sigma_X \sigma_Y}$:
\[
\E \exp\left\{\theta \left(\lvert XY \rvert - \E \lvert XY \rvert \right)\right\} \leq \exp\left\{128e^2\theta^2 \sigma_X^2 \sigma_Y^2\right\}.
\]
\end{lemma}

\noindent We collect a few simple facts about sub-gaussian random variables, which we state without proof.
\begin{fact}
\label{fact:expectationsubgaussianproduct}
Assume $X$ is $\sigma_X^2$ sub-Gaussian and $Y$ is $\sigma_Y^2$ sub-gaussian, then $\E XY \leq \sigma_X \sigma_Y$.
\end{fact}

\begin{fact}
\label{fact:sumsubgaussian}
Assume $X$ is $\sigma_X^2$ sub-gaussian and $Y$ is $\sigma_Y^2$ sub-gaussian.  Then $X + Y$ is $2\left(\sigma_X^2 + \sigma_Y^2\right)$ sub-gaussian.
\end{fact}

\section{Concentration inequalities}
\label{sec:appendixConcentration}
In this appendix, we provide some useful concentration inequalities.  

\subsection{Sub-exponential concentration}
We will often be interested in the sub-gaussian portion of the tail of sub-exponential random variables.  The following well-known lemma formalizes this concentration.
\begin{lemma}
\label{lem:subexpconcentration}
Assume $Z_1, Z_2, \dots, Z_n$ are i.i.d. zero-mean random variables such that $\E e^{\theta Z_i} \leq e^{\theta^2 \sigma_Z^2 /2}$ for all $\lvert \theta \rvert \leq \widebar{\theta}$.  Then, for any $t \leq \sigma_Z^2 \widebar{\theta}$:
\[\pr\left\{\left \lvert \frac{1}{n}\sum_{i=1}^{n}Z_i \right \rvert \geq t\right\} \leq 2\exp\left\{-\frac{nt^2}{2\sigma_Z^2}\right\}\]
\end{lemma}
\begin{proof}
Note that $\pr\left\{\left \lvert \frac{1}{n}\sum_{i=1}^{n}Z_i \right \rvert \geq t\right\} = \pr\left\{\frac{1}{n}\sum_{i=1}^{n}Z_i \geq t\right\} + \pr\left\{\frac{1}{n}\sum_{i=1}^{n}Z_i  \leq -t\right\}$.  We analyze the first term:
\begin{align}
\label{eq:subexpconcentration}
\pr\left\{\frac{1}{n}\sum_{i=1}^{n}Z_i \geq t\right\} &\leq e^{-\theta t}e^{\frac{\theta^2 \sigma_Z^2}{2n}} = e^{-\frac{nt^2}{2\sigma_Z^2}}
\end{align}
where the first step has used Markov's inequality and the second step is by taking $\theta = \frac{nt}{\sigma_Z^2}$ which is by assumption $\leq \widebar{\theta}$.  The second term follows by considering $-Z_i$ which have the same property of the moment generating function, and we are done.
\end{proof}

The following lemma (Bernstein's inequality) is a generalization of this lemma, whose proof follows by optimizing the RHS of the inequality in~\eqref{eq:subexpconcentration} over all $\theta \in \mathbb{R}$.  The proof is straightforward and we omit it here:
\begin{lemma} (Bernstein Inequality)
\label{lem:bernstein}
Assume $Z_1, Z_2, \dots, Z_n$ are i.i.d. zero-mean random variables such that $\E e^{\theta Z_i} \leq e^{\theta^2 \sigma_Z^2 /2}$ for all $\lvert \theta \rvert \leq \widebar{\theta}$.  Then for any $t \geq 0$:
\[
\pr\left\{\left \lvert \frac{1}{n}\sum_{i=1}^{n}Z_i \right \rvert \geq t\right\} \leq 2\exp\left\{-\frac{n}{2}\min\left(\frac{t^2}{\sigma_Z^2}, t\widebar{\theta}\right)\right\}
\]
\end{lemma}
\subsection{Tail Bounds for Sums of Random Variables}
\begin{lemma}
\label{lem:truncation}
    Consider iid non-negative random variables $(Z_i)_{i \in [n]}$ and assume
    the tail bound:
    \begin{align*}
    \pr\left\{Z_i \geq t\right\} \leq Ce^{-ct^{\gamma}} 
    \end{align*}
    where $\gamma \in (0,1)$ and $C, c$ are constants.  Then, assuming $t < \frac{8eC}{\gamma
    c}\left(n \frac{\gamma c}{64e^2C}\right)^{\frac{1 - \gamma}{\gamma}}$:
   \begin{align*}
    \pr\left\{\frac{1}{n}\sum_{i=1}^{n} Z_i \geq 2(t + \E Z_1)\right\} \leq 2
        \text{exp}\left\{-C\left(\left(n t^2\right)^{\frac{\gamma}{2 - \gamma}}
        + \log{n}\right)\right\}
    \end{align*} 
    where $C$ is a universal constant.
\end{lemma}
\begin{proof}
We proceed by a truncation argument. Let $Z_i^{\downarrow} :=
    Z_i\mathbbm{1}\left\{Z_i \leq
    M\right\}$ and $Z_i^{\uparrow} := Z_i\mathbbm{1}\left\{Z_i \geq M\right\}$.  We then see that:
    \begin{align*}
        \pr\left\{\frac{1}{n}\sum_{i=1}^{n}Z_i \geq 2 ( t + \E Z_1)\right\} &\leq \underbrace{\pr\left\{\frac{1}{n}\sum_{i=1}^{n}Z_i^{\downarrow} \geq
        t + \E Z_1\right\}}_{A.} + \underbrace{\pr\left\{\frac{1}{n}\sum_{i=1}^{n}Z_i^{\uparrow} \geq
       t + \E Z_1\right\}}_{B.} 
    \end{align*}
    We examine term B. first.  Notice that $\pr\left\{\frac{1}{n}\sum_{i=1}^{n}Z_i^{\uparrow} \geq
        t + \E Z_1\right\} \leq n \pr\left\{Z_1^{\uparrow} \geq
        t + \E Z_1\right\}$.  Additionally, we have:
        \begin{align*}
        \pr\left\{Z_1^{\uparrow} \geq
            t + \E Z_1\right\} &= \pr\left\{Z_1^{\uparrow} \geq t + \E Z_1 \mid
            Z_1 < M\right\}\pr\left\{Z_1 <
            M\right\} + \\
            &\pr\left\{Z_1^{\uparrow} \geq t + \E Z_1 \mid
            Z_1 \geq M\right\}\pr\left\{Z_1 \geq M\right\}
            \\
            &\leq \pr\left\{Z_1 \geq M\right\}
        \end{align*}
    Thus:
    \begin{align*}
        \pr\left\{\frac{1}{n}\sum_{i=1}^{n}Z_i \geq 2(t + \E Z_1)\right\} &\leq
        \pr\left\{\frac{1}{n}\sum_{i=1}^{n}Z_i^{\downarrow} \geq
        t + \E Z_1\right\} + n\pr\left\{Z_1 \geq M\right\} \\
        &\leq 2\pr\left\{\frac{1}{n}\sum_{i=1}^{n}Z_i^{\downarrow} \geq
        t + \E Z_1\right\} \vee n\pr\left\{Z_1 \geq M\right\} \\
        &\leq 2\pr\left\{\frac{1}{n}\sum_{i=1}^{n}Z_i^{\downarrow} - \E
        Z_i^{\downarrow} \geq
        t\right\} \vee n\pr\left\{Z_1 \geq M\right\}
    \end{align*}
    where the last inequality follows because $\E Z_1 \geq \E
    Z_1^{\downarrow}$.  The truncated portion now has a moment generating function that exists, so
    we aim to compute $\pr\left\{\frac{1}{n}\sum_{i=1}^{n}Z_i^{\downarrow} - \E
    Z_i^{\downarrow} \geq
        t\right\}$ by bounding
        $\E\left\{e^{\lambda \left(Z_1^{\downarrow} - \E
        Z_1^{\downarrow}\right)}\right\}$.  We have:
    \begin{align*}
        \E\left\{e^{\lambda \left(Z_1^{\downarrow} - \E Z_1^{\downarrow}\right)}\right\} = \sum_{\ell =
        0}^{\infty}\frac{\lambda^{\ell} \E  \left(Z_1^{\downarrow} - \E
        Z_1^{\downarrow}
        \right)^{\ell}}{\ell!} \leq 1 + \sum_{\ell
        =2}^{\infty}\frac{\lambda^{\ell} 2^{\ell} \E  \left\lvert
        Z_1^{\downarrow} \right\rvert^{\ell}}{\ell!}
    \end{align*}
    where the inequality is by the same trick used in the proof of Lemma~\ref{lem:subGaussianSquared}.  Now, notice that 
    \begin{align*}
        \E \left \lvert Z_1^{\downarrow} \right \rvert ^{\ell} &= \int_{0}^{\infty}
        \ell y^{\ell - 1} \pr\left\{\left\lvert Z_1^{\downarrow} \right\rvert
        \geq y\right\}dy \\
        &= \int_{0}^{M}
        \ell y^{\ell - \ell_0}y^{\ell_0 - 1} \pr\left\{\left\lvert Z_1^{\downarrow} \right\rvert
        \geq y\right\}dy \\
        &\leq \frac{\ell M^{\ell - \ell_0}\Gamma\left(\ell_0
        \gamma^{-1}\right)}{\gamma c^{\ell_0 \gamma^{-1}}}
    \end{align*}
    Here, we have chosen $\ell_0$ arbitrarily, and we set it such that $\ell_0
    \gamma^{-1} = \ell$.  Thus, we have:
    \begin{align*}
    1 + \sum_{\ell
        =2}^{\infty}\frac{\lambda^{\ell} 2^{\ell} \E  \left\lvert
        Z_1^{\downarrow} \right\rvert^{\ell}}{\ell!} &\leq 1 + \sum_{\ell =
        2}^{\infty} \frac{(2\lambda)^{\ell}}{\ell!}\frac{\ell M^{\ell - \ell_0}\Gamma\left(\ell_0
        \gamma^{-1}\right)}{\gamma c^{\ell_0 \gamma^{-1}}} \\
        &\leq 1 + \frac{C}{\gamma}\sum_{\ell = 2}^{\infty} \left(\frac{4e\lambda
        M^{1 - \gamma}}{c}\right)^{\ell} 
    \end{align*}
    Now, under the assumption $\lambda < \frac{c M^{1 - \gamma}}{8e}$, we get
    the upper bound: 
    \begin{align*}
    \E\left\{e^{\lambda \left(Z_1^{\downarrow} - \E
        Z_1^{\downarrow}\right)}\right\} \leq \text{exp}\left\{\frac{32 e^2
        C}{\gamma c^2} M^{2 - 2\gamma} \lambda^2\right\} 
    \end{align*}
    Thus, we can see that under the assumption $t < \frac{8eC}{\gamma c} M^{1 -
    \gamma}$, we have:
    \begin{align*}
    \pr\left\{\frac{1}{n}\sum_{i=1}^{n}Z_i^{\downarrow} \geq
        t + \E Z_1\right\} \leq 2\text{exp}\left\{-\frac{nt^2 \gamma c^2}{64e^2C
        M^{2 - 2\gamma}}\right\} 
    \end{align*}
    Putting these pieces together, we see that under the same assumption on $t$:
    \begin{align*}
        \pr\left\{\frac{1}{n}\sum_{i=1}^{n} Z_i \geq 2(t + \E Z_1)\right\} \leq
        2\text{exp}\left\{-\frac{nt^2 \gamma c^2}{64e^2C
        M^{2 - 2\gamma}}\right\} \vee \text{exp}\left\{-cM^{\gamma} +
        \log{Cn}\right\}
    \end{align*}
    Balancing terms, we set the truncation $M = \left(nt^2 \frac{\gamma c}{64
    e^2 C}\right)^{\frac{1}{2 - \gamma}}$.  Thus, if $t < \frac{8eC}{\gamma
    c}\left(n \frac{\gamma c}{64e^2C}\right)^{1 - \gamma}{\gamma}$, we have:
    \begin{align*}
    \pr\left\{\frac{1}{n}\sum_{i=1}^{n} Z_i \geq 2(t + \E Z_1)\right\} \leq 2
        \text{exp}\left\{-c\left(n t^2 \frac{\gamma c}{64 e^2 C}\right)^{\frac{\gamma}{2 - \gamma}} + \log{C n}\right\}
    \end{align*}
\end{proof}

\end{document}